\documentclass[12pt]{amsart}

\usepackage{psfrag}
\usepackage{color}
\usepackage{mathrsfs}

\usepackage[mathscr]{euscript}

\usepackage[applemac]{inputenc}
\usepackage{graphicx,epsfig,amscd,graphics,xypic}
 \usepackage[all]{xy}
\usepackage{amsmath,amssymb}

\usepackage[margin=4cm]{geometry}

\usepackage{mathrsfs}

\usepackage{fourier}

\usepackage{pifont}

\usepackage{ulem}

\usepackage{multirow}

\usepackage{chngcntr}
\counterwithout{figure}{section}

\usepackage{hyperref}
\hypersetup{
    bookmarks=true,         
    unicode=false,          
    pdftoolbar=true,        
    pdfmenubar=true,        
    pdffitwindow=false,     
    pdfstartview={FitH},    
    pdftitle={My title},    
    pdfauthor={Author},     
    pdfsubject={Subject},   
    pdfcreator={Creator},   
    pdfproducer={Producer}, 
    pdfkeywords={keyword1} {key2} {key3}, 
    pdfnewwindow=true,      
    colorlinks=false,       
    linkcolor=red,          
    citecolor=green,        
    filecolor=magenta,      
    urlcolor=cyan           
}

\newcommand{\sk}{\smallskip}
\newcommand{\mk}{\medskip}
\newcommand{\bk}{\bigskip}

\newcommand{\F}{\mathscr{F}_{[\rho]}}

\renewenvironment{proof}{\noindent {\it Dï¿½monstration.}}{$\hfill \square$}

\newtheorem{thm}{Theorem}[section]

\newtheorem*{thm*}{Theorem}   
\newtheorem*{coro*}{Corollary}  
    
\newtheorem*{lemma*}{Lemma}

\newtheorem{prop}[thm]{Proposition}
\newtheorem{defi}{Definition}[section]
\newtheorem{quest}{Question} 
\newtheorem{lemma}[thm]{Lemma}
\newtheorem{rem}[thm]{Remark}

\newtheorem*{ex*}{Example \ref{exJ2} (continued)}

\renewenvironment{proof}{\hspace{-0.4cm}{\bfseries Proof.}}{\qed}

\title[{\bf Moduli spaces of flat tori}]{Moduli spaces of flat tori with prescribed holonomy}

\author[{\bf S. Ghazouani}]{\bf GHAZOUANI  Selim}

\address{S. Ghazouani, DMA - \'ENS,  45 rue d'Ulm  75230 Paris Cedex 05 - France}
\email{selim.ghazouani@ens.fr}

\author[{\bf L. Pirio}]{PIRIO Luc}
\address{L. Pirio, LMV,  UMR 8100 du CNRS,  
 Universit\'e Versailles--St.\,Quentin, 
  45 avenue des tats-Unis
78035 Versailles Cedex - France.}
\email{luc.pirio@uvsq.fr}

\begin{document}

\maketitle
\begin{abstract} 
We generalise to the genus one case several  results of Thurston concerning moduli spaces of flat Euclidean structures  with conical singularities on the two dimensional sphere. 

More precisely, we study the moduli space of flat tori with $n$ cone points and a prescribed holonomy $\rho$. In his paper {\it `Flat Surfaces'}\, Veech has established that under some assumptions on the cone angles,  such a moduli space $\F\subset \mathscr M_{1,n}$ carries a natural 
geometric structure modeled on the     
 complex hyperbolic space ${\mathbb C}{\mathbb{H} }^{n-1}$ which is not metrically complete.  Using surgeries for flat surfaces, 
we prove that the metric completion $\overline{\F}$ is obtained by adjoining  to $ \F $ certain strata that are themselves  moduli spaces  of flat surfaces of genus 0 or 1, obtained as degenerations of the flat tori whose moduli space is $ \F$.   
We show that the  ${\mathbb C}{\mathbb{H} }^{n-1}$-structure of  $ \F$   extends to a complex hyperbolic cone-manifold structure of finite volume on $ \overline{\F}$ and we compute the cone angles associated to the different strata of codimension 1. \sk

Finally, we address the question of whether or not the holonomy of\;Veech's ${\mathbb C}{\mathbb{H} }^{n-1}$-structure on 
$ \mathscr F_{[\rho]}$ has a discrete image in $ {\rm Aut}({\mathbb C}{\mathbb{H} }^{n-1})=\mathrm{PU}(1,n-1)$. We outline a general strategy to find moduli spaces $\mathscr F_{[\rho]}$ whose ${\mathbb C}{\mathbb{H} }^{n-1}$-holono\-my gives rise to lattices in $\mathrm{PU}(1,n-1)$ and eventually we give a finite list of $\mathscr F_{[\rho]}$'s whose holonomy is a complex hyperbolic arithmetic lattice.
\mk 

\end{abstract}




\section{\bf INTRODUCTION}

For any non-negative integers $g$ and $n$ such that $2g-2+n>0$, we denote by  $\mathscr M_{g,n}$  the moduli space of genus $g$ Riemann surfaces with $n$ marked points, viewed as a complex orbifold (see \cite[Chap. XII]{ACG} for instance).
\begin{center}
{$\star$} 
\end{center}

In their paper \cite{DeligneMostow} on the monodromy of Appell-Lauricella hypergeometric functions, Deligne and Mostow bring to light complex hyperbolic structures on $\mathscr{M}_{0,n}$ for $n\geq 4$,  parametrised by 
a $n$-tuple  $\mu = (\mu_1,\ldots,\mu_{n}) \in ]0,1[^{n}$ such that $\sum_{i=1}^n{\mu_i} = 2$. They prove that if $\mu$ verifies the arithmetic criterion 
 $$
 {\rm (INT)}\qquad \quad 
  \forall i ,  j
   \mbox{ with }\,   i \neq j \;\;  :\; \;   
    \mu_i+ \mu_j < 1  \;\;  \Longrightarrow  \;\;   \big(1 - \mu_i - \mu_j \big)^{-1}
  \in \mathbb{Z} \, , \qquad \quad$$
then the holonomy of the associated complex hyperbolic structure is a lattice in the automorphism group 
$$\mathrm{PU}(1, n-3) = \mathrm{Aut}\big({\mathbb C}{\mathbb{H} }^{n-3}\big)$$ of the $(n-3)$-dimensional complex hyperbolic space ${\mathbb C}{\mathbb{H} }^{n-3}$. The above criterion can be refined to the following (see \cite{MostowIHES}) : 
%
$$
 {\rm (\Sigma INT)}\qquad   \forall i ,  j
   \mbox{ with }\,   i \neq j \;\;  :\; \;   
    \mu_i+ \mu_j < 1  \;\;  \Longrightarrow  \;\;   \big(1 - \mu_i - \mu_j \big)^{-1}
    \in
     \begin{cases}
 \hspace{0.3cm} \mathbb{Z}    \quad  \text{if } \; \mu_i\neq \mu_j \, , \\ 
  \;  \frac{1}{2}\mathbb{Z} \quad 
 \text{if } \; \mu_i = \mu_j \, .
      \end{cases}
       $$
\sk

In  \cite{Thurston}, Thurston gives a geometric interpretation of these complex hyperbolic structures in terms of flat metrics with cone type singularities on the sphere $S^2$. 
Define $\theta = (\theta_1,\ldots,\theta_n)\in ]0,2\pi[^n$ by $\theta_i = 2\pi(1-\mu_i)$ for $i=1,\ldots,n$. One can think of $\mathscr{M}_{0,n}$ as the set of flat metrics with $n$ cone points of respective angles $\theta_1, \ldots, \theta_{n}$ on the sphere with area $1$ (up to isometry), which we denote by $\mathscr{M}_{0,\theta}$. Parametrising such flat structures 
 naturally endows 
 $\mathscr{M}_{0,\theta}$ with a complex hyperbolic structure (see \cite{Thurston}, \cite{Schwartz} or \cite{Parker}) which coincides with the one considered in \cite{DeligneMostow}. Thurston describes  the metric completion of  $\mathscr{M}_{0,\theta}$ in terms of degenerations of flat spheres and recovers that the criterion $\mathrm{({\Sigma}INT)}$ is {essentially} equivalent\footnote{{`Essentially equivalent' means `up to some particular cases' which all have been classified (in \cite{Mostow}). In particular, when $n\geq 5$, there is only a finite number of such particular cases.}} to the metric completion of $\mathscr{M}_{0,\theta}$ being an orbifold and therefore a lattice quotient of the complex hyperbolic space ${\mathbb C}{\mathbb{H} }^{n-3}$. \mk 

In \cite{Veech}, Veech extends to compact (oriented) surfaces of arbitrary genus several basic results of Thurston's approach. The starting point 
 is a theorem of Troyanov \cite{Troyanov} asserting that, given  $g\geq 0$ and $n>0$ such that $2g-2+n>0$, 
 if $\theta = (\theta_i)_{i=1}^n\in ]0,\infty[^n$ satisfies the following \textit{discrete Gau\ss-Bonnet formula}
\begin{equation}
\label{E:GBformula}
\sum_{i=1}^n{\big(2\pi - \theta_i\big)} = 2\pi\big(2-2g\big)
\end{equation}
\noindent then,  given a genus $g$ closed oriented surface $N_g$, a conformal structure on it and $n$ distinct points $p_1, \ldots, p_n$ on $N_g$, there exists a unique  flat structure of area 1 on $N_g$ which is compatible with the given conformal structure, singular exactly at $p_1, \ldots, p_n$ and such that   it is locally isometric at $p_i$ to a Euclidean cone of angle $\theta_i$, this for every $i=1,\ldots,n$.
\sk 

Troyanov's theorem gives a natural isomorphism between  
 $\mathscr{M}_{g,n}$ and the set, denoted by $\mathscr{M}_{g,\theta}$,  of 
isomorphism classes of 
flat structures on $N_g$ with $n$ cone points of angle data 
$\theta$.  
 The naive hope that a complex hyperbolic structure would arise when parametrising such a moduli space  
is doomed to failure. Such a fact actually happens in the genus 0 case because in that case prescribing the cone angles is equivalent to prescribing the parallel transport along any closed curve on the punctured surface. But this does not hold for a $n$-punctured surface $N_{g,n}$ of higher genus.
\sk

In \cite{Veech}, Veech shows that  the level sets of the (locally well defined) linear holonomy map


$$ \mathscr{M}_{g,\theta} \longrightarrow \mathrm{H}^1\big(N_{g,n}, \mathbb U\big) $$ 
form a real analytic foliation of $\mathscr{M}_{g,\theta}$ whose leaves are holomorphically embedded complex manifolds of dimension $2g-3+n$. Actually, this map is not well defined at the orbifold points of $\mathscr{M}_{g,\theta}$. To bypass this difficulty, one has to work, as Veech did with great care, not on this moduli space but on its orbifold universal covering. When $g=0$, this foliation is trivial and has only one leaf, which is the whole moduli space    $\mathscr{M}_{0,\theta}$, on which one can put a natural geometric structure modeled on a homogeneous space.  
 A geometric way to do this is as follows: given a  flat sphere with $n$ prescribed conical singularities, one can develop it into the Euclidean plane and get a  $(2n-2)$-gon from which the original flat sphere can be reconstructed.  The conical angles being prescribed, the polygons obtained this way depend only on  $n-2$ complex parameters and the area form is a non-degenerate Hermitian form in these parameters   if none of the conical angles $\theta$ is an integer multiple of $2\pi$.  This method, which was first introduced by Thurston in \cite{Thurston} in the genus 0 case,  extends 
 very naturally to the leaves of Veech's foliation, whatever the genus is: 
 one can parametrise locally such a leaf by means of Euclidean $(4g+2n-2)$-gons which depend only on  $2g+n-2$ complex parameters. We call {\it  linear parametrisation} such a parametrisation. 
\sk

Given two integers $p,q\in \mathbb N$, let $\mathrm U(p,q)$ be the group of linear automorphisms  of $\mathbb C^{p+q}$  which leave invariant the standard Hermitian form $h_{p,q}$ of signature $(p,q)$.  The projectivization of the set of 
 $z\in \mathbb C^{p+q}$ such that $h_{p,q}(z)>0$ is known as the {\it {\rm (}indefinite when $p>1${\rm )} complex hyperbolic space of type $(p,q)$} and will be denoted by $\mathbb C\mathbb H^{p+q-1}_{p}$.  It is homogeneous under $\mathrm{PU}(p,q)$, {\it cf.} \cite[\S12.2]{Wolf}.  \mk 

Let $\theta \in (\mathbb{R}_+^* \setminus 2\pi \mathbb{Z})^n$ be such that 
the Gau\ss-Bonnet relation \eqref{E:GBformula} holds true.
\label{ici}
\begin{thm*}[\cite{Veech}]
There exists $(p_{\theta}, q_{\theta}) \in \mathbb{N}^2$ with $p_\theta+q_\theta=2g+n-2$  such that the natural linear parametrisations  of the leaves of Veech's foliation together with their area form endow them with a $\big(\mathbb C\mathbb H^{2g+n-3}_{p_\theta} , \mathrm{PU}(p_{\theta}, q_{\theta}) \big)$-structure.
\end{thm*}

Moreover, in \cite[\S14]{Veech}, Veech performs a lengthy explicit calculation  leading to the conclusion that the geometric structure on the leaves of the preceding theorem is complex hyperbolic ({\it i.e.} $p_\theta=1$) in exactly two cases: 
\begin{enumerate}
\item[{\it (i)}] $g = 0$ and all the conical angles $\theta_i$ are in $] 0 , 2\pi [$; or 
\sk
\item[{\it (ii)}] $g=1$ and all the angles $\theta_i$ are in  $]0,2\pi[$ except  one which lies in $]2\pi, 4\pi[$.
\end{enumerate}

As said above, the former case was treated  in \cite{Thurston} (as well as in 
\cite{DeligneMostow} but with the approach involving hypergeometric functions). 
  In this paper we investigate the latter case. 
\mk

Let $\theta=(\theta_i)_{i=1}^n$ satisfying \eqref{E:GBformula} for $g=1$ and $n>1$ and such that condition {\it (ii)} above holds true.  
For any linear holonomy $\rho$, we denote by $\mathscr{F}_{[\rho]}$ the leaf of Veech's foliation on 
$\mathscr M_{1,\theta}$  that corresponds to (the orbit $[\rho]$ through the action of the pure mapping class group of) $\rho$. The analytic and very explicit description of Veech's foliation carried out in the twin paper \cite{GhazouaniPirio1} shows that when $\rho$ is rational (meaning that the subgroup $\mathrm{Im}(\rho) \subset \mathbb U$ is finite), $\F$ is an algebraic suborbifold of $\mathscr{M}_{1,n}$. 
\sk 

In this paper we give an extrinsic geometric description of the metric completion of such a leaf  $\F$ for the complex hyperbolic structure given by Veech's Theorem above. The main theorem of the paper is a generalisation of a result of Thurston in \cite{Thurston}.

\begin{thm*}
\label{thm}
Let $\rho \in  \mathrm{H}^1(N_{1,n}, \mathbb{U})$ be a rational linear holonomy data.
\begin{enumerate}
\item The metric completion 
 of $\F$ has a stratified analytic structure whose strata are finite unramified covers of lower dimensional rational leaves of Veech foliations  
 on $\mathscr M_{g',n'}$ with $g'=0$ and $n'\leq n+1$ or with $g'=1$ and $n'\leq n-1$; there is a finite number of such strata.
\sk 
\item This metric completion, denoted by $\overline{\F}$ hereafter, is a complex hyperbolic cone manifold of dimension $n-1$,  whose volume  is finite.
\sk 
\item The cone angles around strata of complex codimension $1$ of 
$\overline{\F}$ 
can be computed using appropriate surgeries.  
\end{enumerate}
\end{thm*} 

The genus 0 case invites us to wonder if some of these leaves $\overline{\F}$ are lattice quotients of $\mathbb{C} {\mathbb{H}}^{n-1}$. Unfortunately, the computation of the cone angles around codimension $1$ strata (point $(3)$ of the previous theorem) shows that as soon as $n>2$,  the cone angle around a certain stratum of codimension $1$ (formed by collisions involving the only cone point whose angle is larger than $2\pi$) is bigger than $2\pi$. This prevents any leaf  $\overline{\F}$ from being a lattice quotient provided that $n \geq 3$. 

Nevertheless, it does not exclude the possibility that the holonomy of the complex hyperbolic structure is a lattice in $\mathrm{PU}(1,n-1)$, as both Mostow and Sauter showed that it can happen when $g=0$ (see \cite{Mostow,Sauter}). As a  nice corollary of Theorem \ref{thm}, we obtain that the holonomy of a finite number of moduli spaces $\F$ is an arithmetic lattice. More precisely:

\begin{coro*}

Let $\rho \in  \mathrm{H}^1(N_{1,n}, \mathbb{U})$ be such that $\mathbb{Z}[\mathrm{Im}(\rho)] \subset \mathbb{C}$ is discrete. Then the image of the complex hyperbolic holonomy of (each connected component of) $\F$ is an arithmetic lattice in $\mathrm{PU}(1,n-1)$.

\end{coro*}

This leads us to ask the question of determining all $\rho$ such that the complex hyperbolic holonomy of $\F$ is a lattice. The moduli spaces $\F$ are not always connected; consequently it is more relevant to ask the aforementioned question for the connected components of $\F$. Of course, 
 one must first determine 
these components, which already seems interesting and not completely trivial.
\sk 

\vspace{-0.3cm}
We give in Section \ref{holonomy} some necessary conditions for a component of $\F$'s holonomy to be a lattice. These conditions should reduce the problem to the study of a finite number of candidates.
\mk 

Finally, we would like to draw attention on a possible interpretation of our work. If $\mathrm{Im}(\rho) = \langle \exp({2i\pi}/{n}) \rangle \subset \mathbb U$, the leaf $\F$ can be seen as a stratum of the space of meromorphic differential forms of order $n$ on elliptic curves.

\subsection{Organisation of the paper.}
{\bf Section 1} is the present Introduction.\sk 

{\bf Section \ref{flatsurfaces}} and {\bf Section \ref{Veechfoliation}} are dedicated to introducing the central objects of the article: flat surfaces and Veech isoholonomic foliations on $\mathscr{M}_{g,n}$ respectively.
\sk

 Building on \cite{Schwartz,Thurston,Veech}, we introduce in {\bf Section \ref{Charts}} natural parametrisations of the leaves of Veech's foliations that will be used 
 in the sequel.
\sk 

{\bf Section \ref{geometricproperties}} is devoted to proving technical lemmas on the geometry of flat surfaces that are crucial for describing the metric completion of $\F$.  According to us, some of them, such as Lemma \ref{collisions}, are missing in \cite{Thurston} and could help to complete some proofs in the genus 0 case.  
\sk 

We describe in {\bf Section \ref{surgeries}} several surgeries on flat surfaces which are the major tools of the paper. They allow us to reinterpret some results of \cite{Thurston} and to  formally understand the possible ways flat tori can geometrically degenerate. This leads to a definition of {\it `geometric convergence'} for flat surfaces distinguishing limits by  taking into account not only the isometry class of the limit metric space but also the way to degenerate to it in $\F$. This definition coincides with the  one of convergence for the complex hyperbolic metric but is susceptible to be generalised to cases when Veech's $\mathbb C\mathbb H^{p+q-1}_p$-structure of $\F$ is not Riemannian. Finally, using these surgeries and a simple inductive process, we compute by a geometric argument the signature of Veech's area form, recovering Veech's result.
\sk 

 Sections from \ref{completion} to \ref{cone-manifoldness} are devoted to analysing the geometric structure of $\F$. In {\bf Section \ref{completion}}, we  describe the metric completion of $\F$ in terms of the surgeries introduced in Section \ref{surgeries}, while in {\bf Section \ref{cusps}}, we prove that the complex hyperbolic volume of $\F$ is finite by performing an explicit calculation using special  coordinates. We finally prove in {\bf Section \ref{cone-manifoldness}} that the metric completion of $\F$ has the structure of a complex hyperbolic cone manifold, which is a refinement of the stratified structure brought to light in Section \ref{completion}.\sk 
 
After describing a general algorithm to determine the strata appearing in the metric completion of $\F$, we analyse in {\bf Section \ref{listing}} the particular case of tori with two cone points. In particular,  we show how the rather abstract material developed in this article is used to analyse the (one dimensional in that case) complex hyperbolic structure. We prove that the $\F$ are hyperbolic surfaces with a finite number of cone points and we compute their angles. This section strongly echoes the article \cite{GhazouaniPirio1}.  In particular, this analysis shows that some of them  are lattice quotients of ${\mathbb{C}} \mathbb{H}^1$.\sk 

 In {\bf Section \ref{holonomy}}, we draw a strategy to answer the following question : in the case of tori with {$n\geq 3$} conical points, when is the holonomy of the $ {(n-1)}$-dimensional complex hyperbolic structure on $\F $  a lattice in $\mathrm{PU}(1,{n-1})$ ? We give necessary conditions for the answer to be positive and use them to outline a strategy to reduce the question to a finite number of candidates. We also exhibit cases, {for $n=3,4,5,6$},  where the holonomy group of $\F$ is an arithmetic lattice in $\mathrm{PU}(1,{n-1})$. 
\sk 

The paper ends with two short appendices. In {\bf Appendix A}, after recalling some basic points of complex hyperbolic geometry, we introduce some special coordinates 
which appear  useful in our study. Finally, {\bf Appendix B} is devoted to the notion of cone manifold. We focus in particular on the case of complex hyperbolic cone manifolds.

\subsection{Notes and references}
We think it could be helpful to the reader  to mention 
the main other mathematical works to which the present paper is linked.\sk

As  is more than obvious from the previous lines, this text text must be seen as an attempt to generalize  some results of Thurston's seminal 
 paper \cite{Thurston} concerning moduli spaces of flat spheres to the case of tori.  Even if the term does not appear formally in Thurston's paper, we believe it is fair to say that the crucial  geometric tools used by Thurston are `{\it surgeries}' for objects of this type.   This is a standard but powerful technique to study flat surfaces which has been widely used in the more specific realm of (half-)translation surfaces, see for instance  \cite[Section 6]{MasurSmillie},  \cite{MasurZorich} or \cite{EskinMasurZorich} among many other papers of this field.   It is then not so surprising that surgeries play a central role in the present paper as well.   
 
Thurston's article \cite{Thurston}  has been very influential. 
Among  the papers deep\-ly relying on it  about the theory of conical flat structures on the Riemann sphere, one can mention  \cite{Weber}, \cite{Parker}, \cite{GonzalezLopezLopez}, \cite{BoadiParker} and \cite{Pasquinelli} where some particular cases are considered in detail. The recent paper \cite{McMullen} deserves to be mentioned as well: in it, the author gives a more detailed treatment of the notion of cone-manifold than in \cite{Thurston} and obtains a nice version of the Gau\ss-Bonnet theorem for complex hyperbolic cone-manifolds that he eventually  uses to compute the volumes of the Picard/Deligne-Mostow/Thurston's moduli spaces. 
\mk 
 
  As is well known, some of the main results of \cite{Thurston} coincide with some results obtained previously by Deligne and Mostow in the celebrated papers \cite{DeligneMostow} and \cite{Mostow} (see also their book \cite{DeligneMostow2}). In the true masterpiece \cite{DeligneMostow}, they pursue and obtain definitive results that conclude researches on the monodromy groups of Appell-Lauricella hypergeometric functions going back  at least to Picard. Their approach is not geometric as in \cite{Thurston} but relies essentially on arguments of analytic and/or cohomological nature. \mk

 In addition to \cite{Thurston} and \cite{DeligneMostow}, the other starting point of our research is the remarkable paper \cite{Veech} by Veech that concerns moduli spaces of flat surfaces with conical singularities and seems to have been deeply influenced by the two former articles.  In it, Veech establishes basic and important results concerning flat surfaces of arbitrary genus.  One can say that the methods used by Veech are a mix of the geometric ones of Thurston, and of the analytic ones of Deligne and Mostow.

 It seems to us that   this paper by Veech has not received the attention it deserved 
 despite the importance of the results obtained therein and the interesting problems it suggests. 
 Some of the reasons for this could be that \cite{Veech} is quite long and technical. If most of its  arguments are basically elementary, the analytic treatment used by Veech as well as 
some long computations at some points hide   at first sight the geometrical beauty of its main results. Note also that the topic discussed in \cite{Veech} is quite general since the linear holonomies of the flat surfaces considered in it, if unitary,   are not  just $\pm 1$. 
 It seems that the researchers interested in this subject, Veech included, have focused 
  on the  case of (half-)translations surfaces that is nowadays very   popular and for which a lot of deep results have been obtained during the last twenty years.
    \sk

 We believe that an interesting fact highlighted by our work  
  is that both  Thur\-ston's geometric approach and  Deligne-Mostow's hypergeometric one can be generalised to the genus 1 case.  
 As said above, this is what is done  for Thurston's approach in the present paper.  The hypergeometric approach  la Deligne-Mos\-tow  is developed in the dizygotic twin paper \cite{GhazouaniPirio1}. In it, we first prove that,  as in  the genus 0 case for which this is well-known, Veech's constructions of \cite{Veech} can be made completely explicit when working with elliptic curves. We then specialise    to the case of elliptic curves with two marked points and are able to describe exactly  the moduli spaces of such marked tori that are algebraic subvarieties of the moduli space $\mathscr M_{1,2}$: these are the modular curves $Y_1(N)$ for $N\geq 2$ and we can describe very precisely the complex hyperbolic structure constructed by Veech which each of them carries.     \sk 
 
 Readers are encouraged to take a look at \cite{GhazouaniPirio1} and to compare the methods and the results of the latter to the ones of the present text.

\subsection{Acknowledgements}

We are thankful to Adrien Boulanger for some interesting discussions about the notion of holonomy. 
  We are also grateful to Richard 
     Schwartz for kindly answering several technical questions about the notion of cone manifold, as well as to John Parker and Curtis 
   McMullen for useful correspondences. We are very thankful to Bertrand Deroin  
for the constant support and deep interest he has shown in our work since its very beginning. We are indebted to Irene Pasquinelli and John Parker for very many useful comments and suggestions on the substance and structure of this paper, which have rendered it much clearer. Finally, the second author thanks Brubru for her many    
 corrections.
\sk

While investigating moduli spaces of flat metric, we have taken an interest in historical questions related to the genesis of the papers \cite{DeligneMostow} and  \cite{Thurston}. We are 
  thankful to  Nicolas Bergeron,  Pierre Deligne, \'Etienne Ghys, Curtis McMullen, Jean-Pierre Otal, John Parker,  Marc Troyanov and William Veech for 
  taking  the time to answer some of our historically flavoured questions.
\mk 

{Finally, we are grateful to the two anonymous referees for their  careful readings of the paper and valuable comments
and suggestions.}


\mk

\section{\bf FLAT SURFACES}
\label{flatsurfaces}
We collect in this section some well-known notions and basic results on flat surfaces.
For some general references, 
see \cite{Troyanov,Veech,TroyanovModuli}.

\subsection{Generalities}
\label{S:Generalities}
The Gau\ss-Bonnet formula ensures that the only compact orientable surface carrying a flat metric is the torus. Nevertheless, relaxing the requirement that the metric is flat everywhere and allowing singular points make it possible to build flat surfaces in every genus. \sk 

\begin{figure}[!h]
\centering
\includegraphics[scale=0.4]{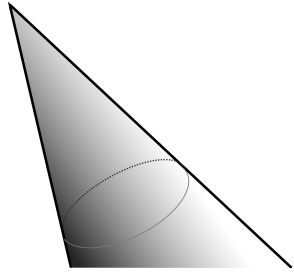}  
\caption{The Euclidean cone $C_\theta$ of angle $\theta\in ]0,2\pi[$ embeds in $\mathbb R^3$.}   
\label{cube}
\end{figure}

\subsubsection{}{\hspace{-0.2cm}}
\label{CTheta}
We first define the kind of singularities that will be allowed for flat surfaces in this paper. 
For any $\theta>0$ distinct from $2\pi$, the  {\bf Euclidean cone of angle} $\boldsymbol{\theta}$, denoted by $C_\theta$ throughout the paper, is the quotient of 
$\mathbb R^+\times (\mathbb R/\theta\mathbb Z)$ obtained by contracting $\{0\}\times (\mathbb R/\theta\mathbb Z)$ onto a point (called the {\bf apex} of the cone) endowed with the flat metric $dr^2+r^2dt^2$  in the standard  coordinates $(r,t)\in \mathbb R_{>0}\times (\mathbb R/\theta\mathbb Z)$.  

 For any positive $\epsilon$, we denote by $C_\theta(\epsilon)$ 
 the image of $[0,\epsilon]\times (\mathbb R/\theta\mathbb Z)$ into $C_\theta$ and the superscript symbol ${}^*$ will mean that the apex has been removed.

\begin{defi}[\cite{Troyanov}]
A {\bf flat surface with conical singularities} is an orientable  compact surface endowed with a flat Riemannian
  metric  
  singular at $n$ points $p_1, \ldots,  p_n$, such that  any  $p_i$ has a neighbourhood  isometric to a Euclidean cone.
\end{defi}

For the sake of simplicity, we will use \textbf{flat surface} throughout the paper instead of {\it flat surface with conical singularities}.  A singular point $p$ is called a \textbf{cone point} or a \textbf{conical point} and the angle $\theta_p$ of the associated Euclidean cone  its \textbf{cone angle}. The quantity $2\pi - \theta_p$ is called the \textbf{curvature at} $\boldsymbol{p}$.

\subsection{Examples.}
We describe below some classical examples of flat surfaces. 
\subsubsection{\bf}
A very intuitive  example of a flat structure is given by the surface of a  cube embedded in $\mathbb{R}^3$.
\begin{figure}[!h]
\centering
\includegraphics[scale=0.25]{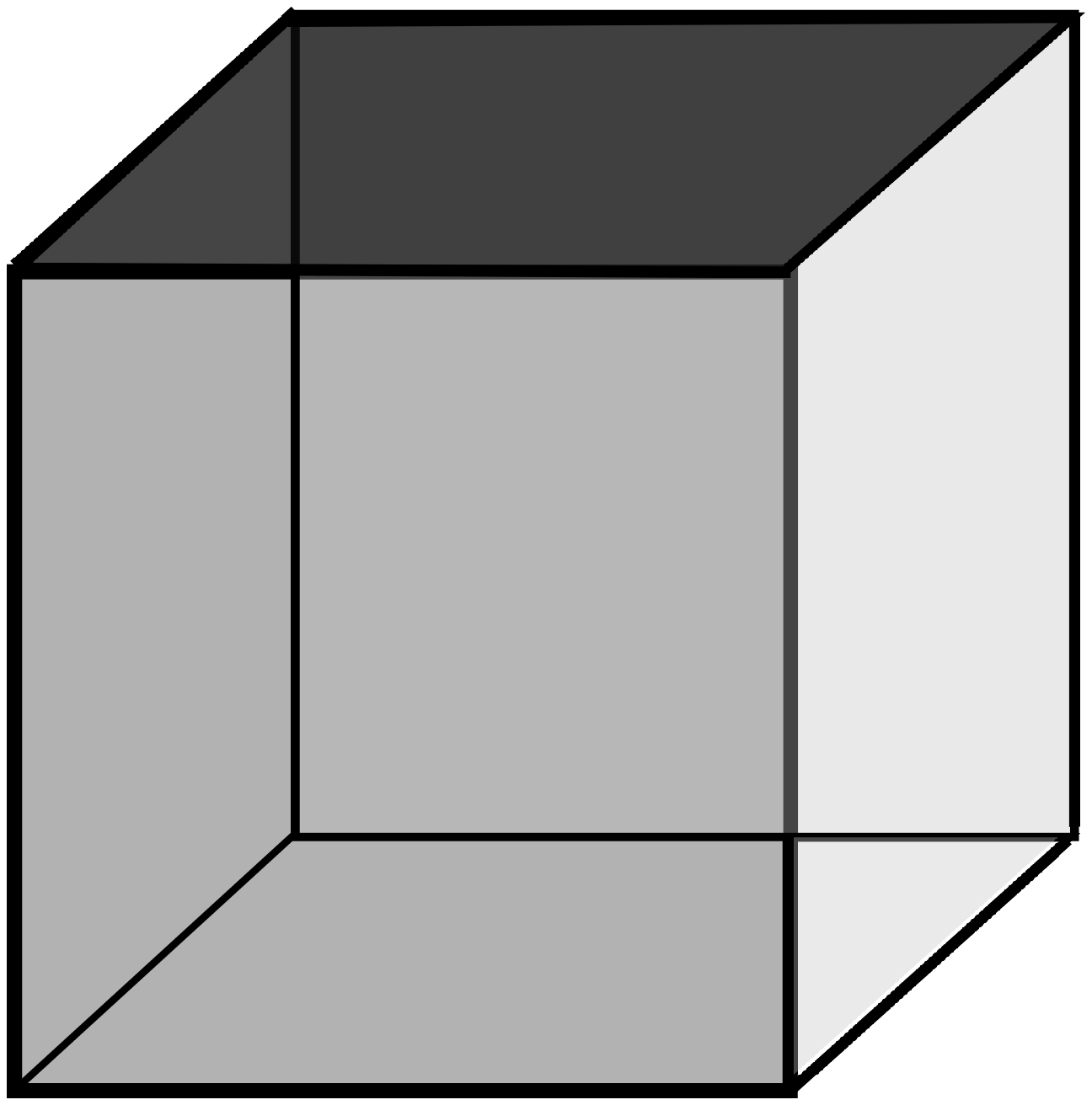} 
\label{cube}
\end{figure}
 The pull-back of the ambient metric defines a flat metric on the 2-dimensional sphere away from the edges and the vertices. On the edges, away from the vertices, the pulled-back metric can be extended in such a way that it is still flat on each edge (this corresponds to the intuitive operation of bending the faces around an edge). We have defined a flat metric on the $8$-punctured sphere. A neighbourhood of each vertex is isometric to a neighbourhood of a Euclidean cone of angle ${3\pi}/{2}$.
 
\subsubsection{\bf} The case of the cube considered above generalises in a straightforward manner to the boundary of any polyhedron $P$ in 
the 3-dimensional Euclidean space:
 the natural flat structures of the polygonal 2-faces of $\partial P$ glue together along the straight edges of $\partial P$ and induce a global flat structure which is regular outside 
 the vertices of $P$ and with conical singularities at these points.
%
 

\subsubsection{\bf} 
Another way to build flat surfaces consists  in gluing  isometrically the sides of only one  Euclidean polygon. 
 We will see later on in Section \ref{Charts} that,  in some sense to be made precise, every flat surface can be built this way. 
This approach is 
quite useful and  will be extensively used  throughout this paper.
\sk

 In order to give a concrete example, we consider the case of a hexagon. One can glue its sides  in three essentially different ways (see \cite[p.89]{JacksonVisentin}) to build (topologically) a torus  as in Figure \ref{patterns} beneath:
\begin{figure}[!h]
\centering
\includegraphics[scale=0.5]{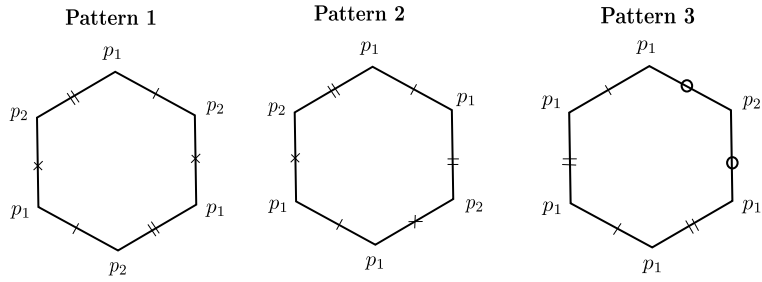}
\caption{Gluing patterns for flat tori with two singular points.}
\label{patterns}
\end{figure}\\
\noindent which respectively give after gluing  the three following  tori with 2 marked points:  
\begin{figure}[!h]
\centering
\includegraphics[scale=0.6]{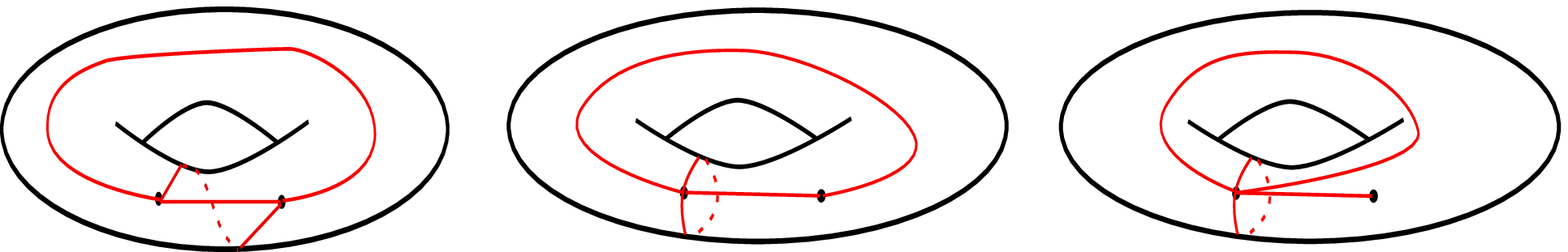}
\caption{}
\label{Fig:GraphsOnATorus}
\end{figure}

Now assume that $H$ is a Euclidean hexagon and choose one of the three gluing patterns of Figure \ref{patterns}, say Pattern 2.  
We choose $H$ such that the sides which are glued together have the same length (see Figure \ref{torus}). 
\begin{figure}[!h]
\centering
\includegraphics[scale=0.5]{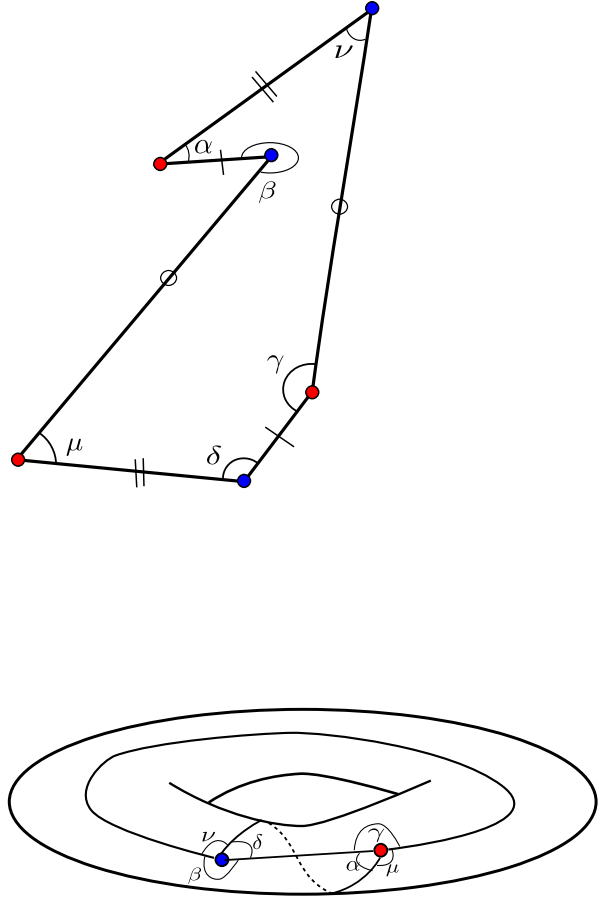}

\caption{A flat torus with two cone points built from gluing the sides of the Euclidean hexagon at the top.}
\label{torus}
\end{figure}
The Euclidean metric on the hexagon can be extended to the whole torus, except at the points corresponding to the vertices of the hexagon. These points have a punctured neighbourhood isometric to a Euclidean cone with angle $\alpha + \gamma + \mu$ for the first point and $ \beta + \delta + \nu$ for the second one. Since $H$ can be triangulated using four Euclidean triangles, one obtains that $ \alpha + \beta + \gamma + \delta + \mu + \nu = 4\pi$. 

Rewriting this equality 
$$ \big(2\pi - (\alpha + \gamma + \mu) \big) + \big(2\pi - (\beta + \delta + \nu) \big) = 0\, , 
 $$ 
one obtains that the sum of the curvatures at the singular points vanishes, which is exactly  Gau\ss-Bonnet formula in this case (see \S \ref{SS:GaussBonnetformula} below). 

\subsubsection{\bf} 
A 
 popular and very  much studied example of flat surfaces is  given by the so-called {\bf (half-)transla\-tion surfaces}, namely  pairs $(X,\omega)$ (resp. $(X,\eta)$) where $X$ is a compact Riemann surface and $\omega$  (resp.\;$\eta$) an abelian (resp.\;a quadratic)  differential on it.  The  flat metric associated to such a pair is just $\lvert \omega\lvert^2$ (resp.\;$\lvert \eta\lvert$).  Note that these objects can be characterised as the flat surfaces whose linear holonomy ({\it cf.}\;Section \ref{SS:AffineLinearHolonomy} below) is trivial (resp.\;has values in $\{\pm 1\}$) hence they form a very particular class of flat surfaces.

\subsubsection{\bf} A \textbf{flat cylinder} $C$ is the metric space one gets by gluing two opposite sides of a Euclidean  rectangle. 
 It is a flat surface with two totally geodesic boundary components. Its \textbf{length} is the length of the sides  glued together and its \textbf{width} is the length of one of its boundary component. 
More intrinsically, the length of $C$ is the distance between its two boundary components and its {width} is  its systole (its systole being the shortest essential closed curve).
\subsubsection{\bf} A \textbf{saddle connection} is a totally geodesic path joining two singular points and meeting no singular points in its interior.

\subsection{\bf On the geometry of flat surfaces}
\label{properties}
In the subsections below, we collect some classical material about the geometry of flat surfaces and fix some definitions and  notations that will be used in the sequel. 

\subsubsection{\bf Flat surfaces as length spaces.} 
We recall in this subsection basic but important properties of flat surfaces that will be used throughout the paper. For a general exposition of the theory of length spaces, we refer to \cite{BridsonHaefliger} or to \cite{Gromov} for the proofs of the results stated in this subsection. \sk

Let $N$ be an arbitrary flat surface with cone type singularities. If $\gamma : [0,1] \longrightarrow N$ is a piecewise $\mathrm{C}^1$-path, one defines its {\bf length} as 
$$ L(\gamma) = \int_{[0,1]}{|\gamma'(t)|dt} \, .$$

As usual, the {\bf distance between two points} $\boldsymbol{x,y \in N}$ is defined as
$$ d(x,y) = \inf_{\gamma}{L(\gamma)} $$
 where $\gamma$ is taken amongst all the 
  $\mathrm{C}^1$-paths  such that $\gamma(0) = x$ and $\gamma(1)=y$. \sk
  
  The following basic results will be extensively used throughout the article:

\begin{itemize}

\item The map  $d : N \times N \longrightarrow \mathbb{R}_+$ is a distance on $N$. Whenever we refer to a distance on the flat surface $N$, it will be this one. \sk

\item For any two points $x,y \in N$,  there exists a piecewise geodesic path $\gamma$ from $x$ to $y$ such that $d(x,y) = L(\gamma)$. \sk

\item For any non-trivial free homotopy class $c$ of closed curves on $N$, there exists a closed, piecewise geodesic path $\varphi$ whose free homotopy class is $c$  such that $L(\varphi) =  \inf_{ \{ \gamma \ | \ [\gamma] = c \}}{L(\gamma)}$. \sk

\item There exists a closed, piecewise geodesic path $\sigma$ whose free homotopy class is non-trivial such that $L(\sigma) =  \inf_{ \{ \gamma \ | \ [\gamma] \neq 0 \}}{L(\gamma)}$. \sk
\end{itemize}

From now on, we will use the notation $N$ for a surface, which depending on the context will be understood as endowed with either a topological, a flat or a conformal structure. We will also use the notation $N_{g,n}$ when we will need to specify the genus and/or the number of cone points (resp.\;marked points) of the flat 
(resp.\;conformal) structure. Finally whenever we will refer to the (co)homology groups $\mathrm{H}_{1}(N_{g,n}, G)$ or $\mathrm{H}^{1}(N_{g,n}, G)$ for a given group $G$, $N_{g,n}$ will stand for  the underlying topological surface of genus $g$ with $n$ punctures.

\subsubsection{\bf The Gau\ss-Bonnet formula} 
\label{SS:GaussBonnetformula}
  A regular Riemannian metric $h$ on a surface $\Sigma$ enjoys the fact that its curvature function $\kappa_h $ satisfies 

$$ \int_{\Sigma}{\kappa_h d\mu_h} = 2\pi\chi(\Sigma) $$ \noindent where $\mu_h$ is the measure on $\Sigma$ induced by $h$ and $\chi(\Sigma)$ is the Euler characteristic of $\Sigma$. This relation is called the Gau\ss-Bonnet formula and can be generalised to the case of flat surfaces with conical singularities. One must think of the associated singular Riemannian metric as a metric whose curvature is concentrated at its singular locus, and therefore think of its curvature function as a linear combination of Dirac masses at the singular points. \sk 

If $N$ is a compact orientable flat surface with $n$ singular points of respective cone angles $\theta_1, \ldots, \theta_n$, the following \textbf{Gau\ss-Bonnet formula} holds true:
$$ \sum_{i=1}^n{(2\pi - \theta_i)} = 2\pi \chi(N). $$   
We refer to \cite[\S3]{Troyanov} or \cite[\S3]{Veech} for proofs and more details on this matter. 

\subsubsection{\bf Exponential maps} 

Let $p$ be a regular point of $N$ and denote by $r_p$ the distance from $p$ to the set of singularities. For $r>0$, one denotes by $D(r)$ the Euclidean disk of radius $r$ centered at the origin. We will say that  `the' {\bf exponential map} $\boldsymbol{i_p}$ {\bf at $\boldsymbol{p}$} is the map (well defined and unique up to rotations)
$$  i_p : D(r_p) \longrightarrow N $$ 
such that $i_p(0) = p$ and which is a local isometry. This map can be extended to the whole Euclidean plane  except for a countable union of semi-lines  which correspond to the geodesics starting at $p$ which cannot be extended because they meet a singular point. The proof is elementary and left to the reader. \sk
 
This definition generalises at a singular point $p$ of $N$. If $\theta$ stands for the cone angle  at $p$, let $r_p$ be the biggest $r>0$ such that 
  the portion of cone $C_{\theta}(r)$  ({\it cf.} \S\ref{CTheta}) can be isometrically embedded in $N$ at $p$.  Then one defines the {\bf exponential map at the cone point} $\boldsymbol{p}$ as the corresponding embedding
  $$ 
i_p : C_{\theta}(r_p) \longrightarrow N 
$$ 
which is unique, up to the isometries ({\it i.e.}\;rotations) of the  
 cone  $C_\theta$. This map enjoys the same properties as the exponential map at a regular point.

\subsection{\bf Affine and linear holonomy of a flat surface}
\label{SS:AffineLinearHolonomy}
If $N$ is a flat surface, the punctured surface $N^*=N \setminus S$ (where $S$ is the set of singular points of $N$) is endowed with a (non complete) flat metric which is everywhere regular. Another way to phrase this is to say that $N \setminus S$ carries a $(\mathbb{C}, \mathrm{Iso}^+(\mathbb{C}))$-structure ($\mathrm{Iso}^+(\mathbb{C})$ denotes the group of orientation preserving Euclidean isometries of $\mathbb{C} \simeq \mathbb{R}^2$). 

With such a structure comes a \textbf{holonomy representation} 
\begin{equation}
\label{E:FullHolonomy}
 \mathrm{Hol} : \pi_1(N ^*) \longrightarrow \mathrm{Iso}^+(\mathbb{C})\, ,  
 \end{equation}
 whose class for the action by conjugation of  $\mathrm{Iso}^+(\mathbb{C})$,  is a geometric invariant of the flat structure. The group $ \mathrm{Iso}^+(\mathbb{C})$ being the set of affine transformations of the form $z \mapsto az+b $ with $a \in \mathbb{U}$ and $b \in \mathbb{C}$, it is isomorphic to the semi-direct product $\mathbb{U} \ltimes \mathbb{C}$. The projection onto the first factor $\mathbb{U}$ is a group homomorphism and post-composing $\mathrm{hol}$ by this projection produces a new representation 
$$ \rho : \pi_1\big(N^*\big) \longrightarrow \mathbb{U} $$
called the \textbf{linear holonomy} of the considered flat surface.\sk

 The group $\mathbb{U}$ being commutative, $\rho$ factors through the abelianisation of $\pi_1(N^*)$ namely the first homology group $\mathrm{H}_1(N^*, \mathbb{Z})$ of the punctured surface $N^*$.  Let $n$ be the cardinality of $S$ and denote by $p_1,\ldots,p_n$ the $n$ cone points of $N$. If $\theta = (\theta_1, \ldots, \theta_n)$ 
 is the associated angle datum ({\it i.e.} the cone angle at $p_k$ is $\theta_k$ for any $k$) and if $\delta_k$ is a simple closed curve turning anticlockwise around $p_k$, then necessarily $\rho(\delta_k) = e^{i\theta_k}$ for $k =1,\ldots,n$. We denote by $\mathrm{H}^1(N ^*, \mathbb{U}, \theta)$ the set of  $\mathbb Z$-linear forms on 
 $\mathrm{H}_1(N^*, \mathbb{Z})$ 
 which maps $\delta_k$ onto  $e^{i\theta_k}$ for every $k$: 
$$
 \mathrm{H}^1\big(N ^*, \mathbb{U}, \theta\big)=\Bigg\{ 
 \rho\in {\rm Hom}\Big( 
 \mathrm{H}_1\big(N ^*, \mathbb Z \big), \mathbb{U}
 \Big)\; \, \Big\lvert \;\, \rho\big(\delta_k\big)=e^{i\theta_k}\, \mbox{ for }\, k=1,\ldots,n\; 
 \Bigg\}\, .
 $$

In  what follows, we will consider $\rho$ as an element of this  space. 
 Remark that basically,  $\rho$ is nothing else but 
 the parallel transport of the flat Riemannian metric on $N^*$ being considered.

\subsection{\bf Isometries}  
\label{SS:IsometryGroup}
We end this section with a few words about  the 
group   $\boldsymbol{\mathrm{\bf Iso}^+(N)}$ of direct isometries  of a given compact flat surface $N$. \sk 

First, remark that this group is finite since it embeds into the group of biholomorphisms of the underlying Riemann surface which is known to be finite. 
Second, the subgroup of $\mathrm{Iso}^+(N)$ made of elements fixing a given cone  point must be cyclic, since its elements are completely determined by their differential at the fixed point which is a rotation.  It follows easily that the subgroup $\boldsymbol{\mathrm{\bf PIso}^+(N)}$ formed by pure direct isometries of $N$ (here `pure' means that the considered isometries fix pointwise the set of cone points) is necessarily cyclic.

\section{\bf VEECH'S ISOHOLONOMIC FOLIATIONS ON\;$\boldsymbol{\mathscr{M}_{g,n}}$}
\label{Veechfoliation}

\subsection{\bf Moduli spaces of flat surfaces and Troyanov's Theorem.}
\label{S:ModuliSpacesFlatSurfacesTroyanovTheorem}
Let $N$ be a compact oriented surface of genus $g$ with $n$ marked points $p_1, \ldots, p_n$ and $\theta =(\theta_i)_{i=1}^n$ a set of angle data satisfying the Gau\ss-Bonnet relation \eqref{E:GBformula}. Since we are only interested in this case, and because making such an assumption will simplify the exposition, we will always assume that 
\begin{equation}
\label{HYP}
\mbox{\it 
 none of the angles } \, \theta_i \; \mbox{\it  is  an integer multiple of }\, 2\pi.
\end{equation}


We define $\mathcal E_{g,\theta}$ as the set of flat structures on $N$ such that the metric is singular at $p_i$ with a cone angle $\theta_i$ at this point, up to the action of $\mathrm{Diff}_0^+(N,S)$ where $S= \{ p_1, \ldots, p_n\}$. One can think of $\mathcal E_{g,\theta}$ as the set of flat surfaces of genus $g$ with cone angles $\theta$ with a marking of its fundamental group. For more details on this construction and the ones to come, we refer to \cite[Theorem 1.13]{Veech}.
\mk 

 Notice that a flat structure defines canonically a conformal structure on $N$. Away from the singularities,  this conformal structure is given by the regular flat structure. At a singular point $p$ of cone angle $\theta_p$, there is an essentially unique local coordinate $z$ centered at $p$ such that the flat metric is $|z^{\alpha_p} dz|^2$
with $\alpha_p=({\theta_p}/{2\pi}) - 1$. By means of $z$, one extends the conformal structure of the punctured surface through $p$.  Since this can be done for every conical singularity of $N$, one obtains a well-defined map 
\begin{equation}
\label{E:E(g,n)->T(g,n)}
\mathcal E_{g,\theta}\longrightarrow {\mathcal T}\!\!\!{\it eich}_{g,n}\, , 
\end{equation}
where $ {\mathcal T}\!\!\!{\it eich}_{g,n}$ denotes the usual Teichm\"uller space of conformal structures on a surface of genus $g$ with $n$ marked points. The remarkable fact is that this map is one-to-one.  This is a consequence of 
 Troyanov's theorem stated below. 

\begin{thm*}[\cite{Troyanov}]
Every conformal structure on $N$ is induced by a flat metric with conical singularities of angle $\theta_i$ at $p_i$ for $i=1,\ldots,n$.  Moreover, this flat  metric is unique up to normalization.
\end{thm*}

The proof (given in \cite{Troyanov}) essentially consists in solving the PDE that the metric tensor associated to a given conformal structure must satisfy. For more details, we refer to the original article \cite{Troyanov} which is very pleasant to read. 
\sk 

Consequently, one has a one-to-one correspondance  $\mathcal E_{g,\theta}\simeq {\mathcal T}\!\!\!{\it eich}_{g,n}$ allowing these two moduli spaces to be identified. 
  In particular, this endows $\mathcal E_{g,\theta}$ with the structure of a complex manifold of dimension $3g-3+n$.

\subsection{Veech's foliations}
Since we are considering marked flat structures, the {\bf linear holonomy map}
\begin{equation}
\label{E:LinearHolonomyMap}
 {\rm hol}={\rm hol}_\theta: {\mathcal{T}\!\!\!{\it eich}}_{g,n} \simeq  \mathcal{E}_{g,\theta}  \longrightarrow \mathrm{H}^1\big(N_{g,n}, \mathbb{U}\big) 
 \end{equation}
 which associates its linear holonomy morphism to a flat structure, is well defined.
 Clearly, ${\rm hol}$  maps ${\mathcal{T}\!\!\!{\it eich}}_{g,n}$  into  $\mathrm{H}^1(N_{g,n}, \mathbb{U},\theta)$. \mk

From  hypothesis \eqref{HYP},  it follows  that the trivial character (the one sending any holomogy class  onto $1\in \mathbb U$) does not belong to $\mathrm{H}^1(N_{g,n}, \mathbb{U},\theta)$. This case being excluded, the following theorem holds true:

\begin{thm*}[\cite{Veech}]
The linear holonomy map 
\eqref{E:LinearHolonomyMap}
 is an open real-analytic submersion.   Moreover, for any 
$\rho\in {\rm Im}({\rm hol})$, the level set ${\rm hol}^{-1}(\rho)$ is a complex submanifold of\,\;${\mathcal{T}\!\!\!{\it eich}}_{g,n}$ of complex dimension $2g-3+n$.
\end{thm*}

  This result  implies in particular  that the level sets $\mathcal{F}_{\rho} ={\rm hol}^{-1}(\rho)$ for $\rho \in 
 {\rm Im}({\rm hol})$  
  form a real-analytic foliation  by complex submanifolds of $\mathcal{T}\!\!\!{\it eich}_{g,n}$. This foliation will be denoted by $\mathcal F(\theta)$ (or just $\mathcal F$ for short, when $\theta$ has been fixed) and will be called the {\bf Veech foliation of $\boldsymbol{
  {\mathcal T}  \!\!\!{\it \boldsymbol{eich}}_{g,n}}$ associated to  $\boldsymbol{\theta}$}.

\subsection{Invariance by the pure  mapping class group}

\label{invariance}
We now explain how this foliation descends to $\mathscr{M}_{g,n}$. The pure mapping class group $\mathrm{PMCG}_{g,n}$
acts on $\mathcal{T}\!\!\!{\it eich}_{g,n}$ preserving Veech's foliation: namely any element $f \in \mathrm{PMCG}_{g,n}$ sends $\mathcal{F}_{\rho}$ onto $\mathcal{F}_{f^*\rho}$. Hence the foliation $\mathcal{F}(\theta)$ factors through the projection  
\begin{equation}
\label{E:Teichgn-to-Mgn}
\mathcal{T}\!\!\!{\it eich}_{g,n} \longrightarrow \mathscr{M}_{g,n} = \mathcal{T}\!\!\!{\it eich}_{g,n}/\mathrm{PMCG}_{g,n}
\end{equation} 
to define a singular foliation on the moduli space  $\mathscr{M}_{g,n}$. The latter is denoted by $\boldsymbol{\mathscr F(\theta)}$ (or just by $\boldsymbol{\mathscr F}$ when $\theta$ is fixed) and will also be called {\bf Veech's foliation}. Strictly speaking, since ${\rm PMCG}_{g,n}$ acts with fixed points on $\mathcal{T}\!\!\!{\it eich}_{g,n}$, one should more rigorously speak of ${\mathscr F(\theta)}$ as an `{\it orbifoliation}' on $\mathscr{M}_{g,n}$. However, because it will not be the source of real problems, we will ignore this subtlety in the whole paper.

\begin{itemize}
\item We will now refer to a specific leaf $\mathscr{F}_{[\rho]}$ where $[\rho]$ is the orbit of an element of $\mathrm{H}^1(N_{g,n}, \mathbb{U}, \theta)$ under the action of     
$\mathrm{PMCG}_{g,n}$. 
 Note that it is the image of $\mathcal{F}_{\rho} 
\subset \mathcal{T}\!\!\!{\it eich}_{g,n}$ by the quotient map 
 \eqref{E:Teichgn-to-Mgn}. 
\sk 
\item 
 Since $\mathrm{PMCG}_{g,n}$ acts on $\mathrm{H}^1(N_{g,n}, \mathbb{U},\theta)$ preserving its symplectic form, the foliation has a transverse symplectic structure of dimension $2g$.
\sk 
\item We say that $\mathscr{F}_{[\rho]}$ is a leaf of Veech's foliation. That is not rigorously correct, because usually, in foliation theory, one demands that leaves  be  connected. It is actually proven in \cite[\S4.2.5]{GhazouaniPirio1}, through some explicit analytic computations,  that $\mathscr{F}_{[\rho]}$ can have several  distinct connected components.  Nevertheless, we will refer below to the $\mathscr{F}_{[\rho]}$'s as leaves for convenience.
\end{itemize}

\subsection{Geometric structures on the leaves}

Let $p$ and $q$ be non-negative integers and let 
$$h_{p,q} :  (z,w) \longmapsto \sum_{i=1}^p{z_i\overline{w_i}} - \sum_{j=p+1}^{p+q}{z_j\overline{w_j}}$$
 be the standard Hermitian form of signature $(p,q)$ on $V=\mathbb{C}^{p+q}$. Let $V_+$ be the set of  elements $ z \in \mathbb{C}^{p+q}$ such that $h_{p,q}(z,z) > 0$ and let $\mathbb C\mathbb H^{p+q-1}_p$ be the image of $V_+$ in $\mathbb{CP}^{p+q-1}$. The group of automorphisms of $h_{p,q}$, namely $\mathrm{PU}(p,q) $, acts transitively by biholomorphisms on $\mathbb C\mathbb H^{p+q-1}_p$, see \cite[\S12.2]{Wolf}.  
 Note that for $p=1$, $\mathbb C\mathbb H^{q}_1$ is nothing else but the usual complex hyperbolic space $\mathbb C\mathbb H^{q}$.
\mk

Recall that we are assuming that hypothesis \eqref{HYP} holds true: the angle datum 
 $\theta=(\theta_i)_{i=1}^n\in ]0,+\infty[^n$ is supposed to be such that  $\theta_i \notin 2\pi \mathbb{Z}$ for any $i=1,\ldots,n$.

\begin{thm*}[\cite{Veech}]

There exists a pair of integers $(p,q)=(p_\theta,q_\theta) \in \mathbb{N}^2$  with $p+q=2g-2+n$,  such that the leaves of Veech's foliation 
 $\mathcal F(\theta)$ on $\mathcal{T}\!\!\!{\it eich}_{g,n}$ are endowed with {natural} $(\mathbb C\mathbb H^{p+q-1}_{p}, \mathrm{PU}(p,q) )$-structures. These  geometric structures 
 are invariant under the action of the pure mapping class group  
hence can be pushed-forward on the leaves of Veech's foliation 
$\mathscr F(\theta)$ 
on  $\mathscr{M}_{g,n}$.
\end{thm*}

In \cite[\S14]{Veech}, Veech gives an explicit closed formula 
 for the signature $(p_{\theta},q_{\theta})$ as a  function of $\theta$. We explain briefly where this geometric structure comes from. Consider $\rho \in  \mathrm{H}^1(N_{g,n}, \mathbb{U},\theta)$ in the image of  ${\rm hol}_\theta$
and consider the associated leaf $\mathcal{F}_{\rho} \subset \mathcal{T}\!\!\!{\it eich}_{g,n}$. Given a flat surface in it, one can consider its full Euclidean holonomy 
\eqref{E:FullHolonomy}. Since its linear part  is fixed (and equal to $\rho$), the meaningful geometric information is contained in the translation part of this full holonomy, which  can be viewed as an element of the projectivization  of a certain twisted cohomology group denoted here by $\mathrm{H}^1_{\rho}(N_{g,n}, \mathbb{C})$.  One can then construct a  {\bf relative period map} 
\begin{equation}
\label{E:RelativePeriodMorphism}
 \mathcal{F}_{\rho} \longrightarrow  \mathbf{P}\Big(\mathrm{H}^1_{\rho}\big(N_{g,n}, \mathbb{C}\big)\Big) \, .
 \end{equation} 
 
Veech (and previously Thurston in \cite{Thurston} for  flat surfaces of genus $0$) proves that the preceding map  is a local biholomorphism (see \cite[Theorem 0.6]{Veech}). Moreover, one can define a non-degenerate Hermitian form $h_{\rho}$ on $\mathrm{H}^1_{\rho}(N_{g,n}, \mathbb{C})$ (which is actually the area of the corresponding flat surface) such that the relative period map  
 lands in 
 $\mathbf{P}_+(\mathrm{H}^1_{\rho}(N_{g,n}, \mathbb{C}))$, where the latter stands for           
the set of complex lines in $\mathrm{H}^1_{\rho}(N_{g,n}, \mathbb{C})$ on which $h_\rho$ is positive. Thus the target space of \eqref{E:RelativePeriodMorphism} is nothing else but a model of $\mathbb C\mathbb H^{p+q-1}_p$ and any element 
 $f$ of the    $ \mathrm{PMCG}_{g,n}$  
 induces an isomorphism of $(\mathbb C\mathbb H^{p+q-1}_p,{\rm PU}(p,q))$-structures  
 $$f : \big(\mathcal{F}_{\rho}, h_{\rho}\big) \longrightarrow \big(\mathcal{F}_{f^* \rho}, h_{f^*\rho}\big)\,  $$ (see \cite[Theorem 0.7]{Veech}; this  amounts 
 to saying that changing the marking of a flat surface does not change its area).\sk


 Except for very few cases, those $\mathbb C\mathbb H^{p+q-1}_p$-structures are known not to be complete. The term {\it `complete'} has to be understood here  in the sense of geometric structures, {\it cf.} \cite[\S3.5]{Thurston2}. For geometric structures modeled on a  (possibly indefinite) complex hyperbolic space $\mathbb C\mathbb H^{p+q-1}_p$, this  coincides with the fact of being  geodesically complete for the associated Levi-Civita connection (see Proposition 1.2 of \cite{Tholozan} for instance). A more geometric description of some of those structures in the following sections will make this fact obvious.


\section{\bf LINEAR CHARTS ON THE LEAVES OF VEECH'S FOLIATION. }
\label{Charts}
In this section we present material about local parametrisations of moduli spaces of flat surfaces. Although this material is well known, there is no standard point of view or unified theory of these parametrisations. Depending on the context, parameters obtained from gluing Euclidean polygons, analytic calculations, twisted cohomology or a combination of several of these are better suited to formalize an idea or to simply perform a computation. Nevertheless all these points of view (to be detailed) are essentially the same. References developing various material are \cite[Section 3]{Thurston}, \cite{GhazouaniPirio1}, \cite{Schwartz} and \cite[Sections 9,10 and 11]{Veech}. 
\sk

For the remainder of the section $g$, $n$ and $\rho : \mathrm{H}_1(N_{g,n} , \mathbb{Z}) \longrightarrow \mathbb{U}$ are fixed. 
We also suppose that $2g + n - 3 > 0$ in order for $\F$ to have positive dimension.

\subsection{Polygonal models for flat surfaces.}
\label{polygonalmodel}

We describe here a geometrically intuitive local parametrisation of $\mathcal{F}_{\rho}$. Take $N$ a flat surface in $\mathcal{F}_{\rho}$ such that $N$ can be recovered from gluing isometrically suitable sides of a Euclidean polygon $P$ with $2k$ sides whose vertices project to singular points. Necessarily, $k = 2g-1 + n$. We identify to $\mathbb{C}$ the Euclidean plane in which $P$ lies in and associate to each side the corresponding complex number $z_i$, $1 \leq i \leq 2k $ (with the convention that $P$ is positively oriented relatively to its interior)  defined by 
$$z_i=\mbox{end point of the side }  - \mbox{ initial point of the side\,.}$$

\begin{figure}[!h]
\centering
\includegraphics[scale=0.7]{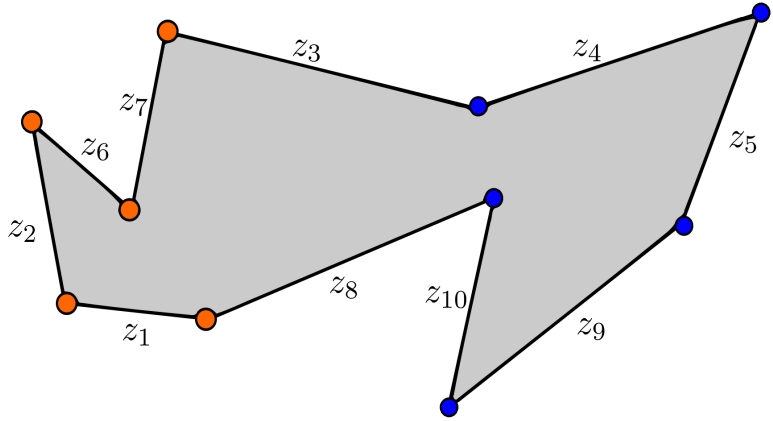}
\caption{Polygonal model for a genus $2$ surface with two cone points}
\label{polygonalmodel2}
\end{figure}

 \noindent Assuming that (the side associated to the complex number) $z_i$ is paired with (the side associated to) $z_{k+i}$, the $2k$-tuple $(z_1, \ldots, z_{2k})$ must satisfy the following relations :
$$ \sum_{i=1}^{2k}{z_i} = 0 
\qquad  \mbox{and} \qquad
 \big| z_{i}\big|  = \big|z_{k+i}\big|  \quad \mbox{for } \, i=1, \ldots, k \, .
$$

Note that since we require that $z_i$ be paired with $z_{i+k}$, the $z_j$'s for $j=1,\ldots,2k$ do not necessarily appear in cyclic order (see Figure \ref{polygonalmodel2} for instance). 
Each complex number $\rho_i ={z_i}/{z_{k+i}} \in \mathbb{U}$ is the holonomy of a  curve (which is closed in the corresponding 
flat surface $N$) joining the middle of $z_i$ to the middle of $z_{k+i}$ and therefore belongs to $\mathrm{Im}(\rho)$. 
  One rewrites the previous equations as 
\begin{equation}
\label{E:RelationPolygonalFRho}
 \sum_{i=1}^{2k}{z_i} = 0 
\qquad  \mbox{and} \qquad
 z_{i}  = \rho_i \, z_{k+i}  \quad \mbox{for } \, i=1, \ldots, k \, .
\end{equation}

After eliminating $z_k, z_{k+1}, \ldots, z_{2k}$, one sees that $z = (z_1, \ldots, z_{k-1})\in \mathbb C^{k-1}$ completely characterises  the polygon $P$ and therefore the associated flat surface $N$. 

Any $(k-1)$-tuple $ u=(u_1, \ldots, u_{k-1})$ close to $z $ in $\mathbb C^{k-1}$ defines a polygon $P_u$ whose sides satisfy the equations above. Performing the associated gluing (meaning that one glues the side associated to $u_i$ to the one associated to  $u_{i+k} = \rho_i^{-1}u_i $ 
for $i=1,\ldots,k-1$) builds another element of $\mathcal{F}_{\rho}$.  

 Let $U \subset \mathbb{C}^{k-1}$ be a small open subset containing $z$ such that all the $2k$-gons corresponding to elements  of $U$ are non-degenerate. 
  One defines a map 
$$ \varphi : U \longrightarrow \mathcal{F}_{\rho} $$
by associating to each $u\in U$ the 
\textbf{renormalized flat surface  associated to $P_u$}, that is the 
one 
of area one. Notice that $\varphi$ is not locally injective since $\varphi( \lambda u) =  \varphi( u)$ for all $(\lambda,u) \in \mathbb{C}^*\times U $  such that $\lambda u\in U$.  That being said, $\varphi$ induces a map $\psi : V 
 \longrightarrow \mathcal{F}_{\rho}$ where $V=\mathbf P U $ is the image of $U$ in $\mathbf P(\mathbb{C}^{k-1})$. It is possible to prove that $\psi$ is a local biholomorphism for the structure inherited as a leaf of a foliation of $\mathcal{T}\!\!\!{\it eich}_{g,n}$ as has been done by Veech, see \cite[Lemma 10.23]{Veech}. 
  Nonetheless, we want to adopt an intrinsic point of view on the geometry of $\mathcal{F}_{\rho}$ and will therefore ignore Veech's results. 
 
  We remark that $\psi : V \longrightarrow \mathcal{F}_{\rho}$ 
   is a local homeomorphism : 
\begin{itemize}
\item The fact that $\psi$ is one-to-one follows straightforwardly from the following remark: since we are looking at marked flat structures, any isometry preserving the marking between close surfaces in the parametrisation $V$ must come from an isometry of the polygons themselves 
being the identity on the boundary of the polygon; and therefore be the identity.\sk 
\item The fact that $\psi$ is onto is a consequence of the fact that the polygonal model survives small deformations. 

\end{itemize}

We will actually ignore the second point and define a structure of (complex) manifold on $\mathcal{F}_{\rho}$ using $\psi$. Two details remain to be settled :
\begin{enumerate}
 \item we have been able to build $\psi$ only if $N$ is built out from gluing sides of a polygon. We now need to extend this construction to the general case;
 \sk 
 \item we need to prove that if two charts have overlapping images, then  the transition maps are biholomorphisms. 
\end{enumerate}
 
The first difficulty can be settled by introducing the notion of \textit{pseudo-polygon}. We follow here \cite{Schwartz}.  
 A \textbf{pseudo-polygon} is a flat metric on a (closed) disk whose boundary is locally isometric to a piecewise geodesic path in $\mathbb{C}$. By developing a pseudo-polygon, we can also define it as an immersion of the closed disk into the plane whose boundary is piecewise geodesic.
\begin{prop}
\label{pp}
Every flat surface $N$ can be built out from gluing sides of a pseudo-polygon.

\end{prop}

The proof of this proposition is carefully done in the case $g = 0$ in \cite{Schwartz}, and in the general case in \cite{Veech}. The crucial point is the existence of a totally geodesic triangulation (see Lemma 6.23 in \cite{Veech} or 
 the construction of the Delaunay decomposition that we will detail in Section \ref{geometricproperties}) for a given flat surface $N$). Starting from there, one easily checks that for any graph $\Gamma$ in the $1$-skeleton of such a triangulation such that $N \setminus \Gamma $ is simply connected, then (the metric completion for the length metric of) the latter is a pseudo-polygon. 
\sk

The main remark at this point is that the parametrisation built when $N$ comes from a polygonal model straightforwardly generalises to the case when $N$ is built out from a pseudo-polygonal model, simply by immersing (using the developing map of the flat structure) such a pseudo-polygon in $\mathbb{C}$. According to Proposition \ref{pp}, every surface has a pseudo-polygonal model and therefore the maps $\psi$ built this way form an atlas of charts for $\mathcal{F}_{\rho}$. \sk 

From now on, a local parametrisation $(z_1, \ldots, z_{k-1})$ arising in this way will be referred to as a \textbf{polygonal parametrisation}.

\subsection{Area form and linear parametrisation.} Another very important remark at this point is that a polygonal parametrisation comes with a natural Hermitian form which is the signed area of the corresponding flat surface. If $U$ is an open subset of $\mathbb{C}^{k-1}$ on which is defined  a polygonal parametrisation $\varphi$ of $\mathcal{F}_{\rho}$, we denote by $A_{\varphi, U}$ the corresponding Hermitian form. 
\sk

The proof that $A_{\varphi, U}$ is actually a Hermitian form in $z=(z_1, \ldots, z_{k-1})$ goes the following way : every immersed pseudo-polygon can be triangulated in such  a way that each side is a geodesic path joining two edges. Let $T_1, \ldots, T_L$ be the triangles of the triangulation. For any $l$, the area of $T_l$ is 
$$ A(T_l) =  \frac{1}{2} \mathrm{Im}\left(z_{T_l} \overline{w_{T_l}}\right)\, $$ 
 where $z_{T_l}$ and  $w_{T_l}$ are the complex numbers associated to two consecutive sides of $T_l$,  oriented in such a way that they form a direct basis of $\mathbb{C}$ (seen as a $2$-dimensional real vector space). Both $z_{T_l}$ and  $w_{T_l}$ are  linear combinations of $z_1, \ldots, z_{2k}$ and therefore of $z_1, \ldots, z_{k-1} $ thanks to \eqref{E:RelationPolygonalFRho}.  For any $l$, the area $A(T_l)$ is a Hermitian form in $z$. Since the area of the whole surface is given by 
$$A_{\varphi, U}(N) = \sum_{l=1}^L{A(T_l)} \, , $$ 
it follows that  $A_{\varphi, U}$ is indeed a Hermitian form in $(z_1, \ldots, z_{k-1})$.

\vspace{3mm}

The next proposition describes the regularity of the transition maps and settles point $(2)$ of Section \ref{polygonalmodel}.

\begin{prop}
\label{linear}

Let $(\varphi_1, U_1)$ and $(\varphi_2, U_2)$ be two polygonal parametrisations of $\mathcal{F}_{\rho}$ such that $W = \varphi_1(U_1) \cap \varphi_2(U_2) 
\subset \mathcal{F}_{\rho}$ is non-empty, connected and sufficiently small for that the projectivizations $\psi_i : \mathbf PU_i\rightarrow \mathcal{F}_{\rho}$ of the $\varphi_i$'s (see above) induce isomorphisms between $\psi_i^{-1}(W)$ and $W$,  for $i=1,2$.

Then $ \psi_2^{-1} \circ \psi_1 \; :\;  \psi_1^{-1}\big(W\big) \longrightarrow  \psi^{-1}_2\big(W\big) $ 
 is the restriction of the projectivization of a linear map $g\in 
{\rm GL}_{k-1}(\mathbb{C})$ such that $g^*A_{\varphi_1,U_1} = A_{\varphi_2, U_2} $.
\end{prop}

\begin{proof} Let $P$ and $Q$ be two polygonal models for a flat surface $N$, and immerse $P$ in $\mathbb{C}$. Let $z_1, \ldots, z_{2k}$ be the complex numbers associated to the sides of $P$. Consider now a side $w_i$ of $Q$, and develop it in $\mathbb{C}$ starting from an initial copy of $P$ (say $P_0$) and gluing  a copy $P_{i+1}$ of $P$ to a side of $P_i$ every time it is necessary to keep track of $w_i$. Thus one can express  $w_i$ as a linear combination of the complex number associated to the sides of $P$ and find an expression of $w_i$ of the form 
$$ w_i = \sum_{j=1}^{k-1}{\alpha_{i,j} z_j} $$
where the $\alpha_{i,j}$ are constants depending only on $\rho$ and the combinatorics of the side $w_i$ relatively to $P$. Therefore the coordinates $(w_1, \ldots, w_{k-1})$ depend linearly on $(z_1, \ldots, z_{k-1})$. Swapping the roles of the two charts, one gets that the transition map actually lies in $\mathrm{GL}_{k-1}(\mathbb{C})$. The area  only depends on the underlying surface and therefore does not depend on the parametrisation. 
\end{proof}
\mk

The proposition above tells us that the polygonal charts endow $ \mathcal{F}_{\rho} $ with a complex projective structure (and with an additional structure coming from the preserved area form, which will be investigated later). The previous analysis invites us to define a more general class of parametrisations : 

\begin{defi}
A local holomorphic parametrisation $(z_1, \ldots, z_{k-1})$ 
of  $\mathcal{F}_{\rho}$ is called a \textbf{linear parametrisation} if it depends linearly on a polygonal parametrisation. 
\end{defi}

This class is much more convenient than the class of polygonal parametrisation because it is the larger class of holomorphic charts enjoying the property that the area form is Hermitian in the associated coordinates. We will also see in Section \ref{topgluing} that it is possible to build  other such linear parametrisations in a natural way  which will be extensively used throughout the article.

\subsection{Projection onto $\F$.}
We have built in this section projective charts on $\mathcal{F}_{\rho} 
\subset \mathcal T\!\!\!{\it eich}_{g,n}$. If $N \in \F$ is a regular point of $\F$ \textit{i.e.} if the projection 
$$ \pi : \mathcal{F}_{\rho} \longrightarrow \F$$
 is a local homeomorphism at $N$, any chart at $\widetilde{N} \in \pi^{-1}(N)$ can be pushed forward and gives a chart at $N$. The fact that $N$ is not regular is equivalent to the fact that $\mathrm{PIso}^+(\widetilde{N})$, the group of pure direct isometries of the flat surface $\widetilde{N}$ (see \S\ref{SS:IsometryGroup}),  is non-trivial. In that case any chart at $\tilde{N}$ gives a non-injective local parametrisation of a neighbourhood of $N$ in $\F$ whose transformation group is the stabilizer of $\widetilde{N}$ in ${\rm PMCG}_{1,n}$
which is isomorphic to  $\mathrm{PIso}^+(\widetilde{N})$.

\subsection{Parametrisations coming from topological gluing.}
\label{topgluing}
Here  we describe pa\-rametrisations which are generalisations of polygonal parametrisations: we are just going to relax the condition that the sides of the polygon we are gluing be geodesic.

Consider a (topological) triangulation $\mathscr T$  of $N_{g,n}$ such that the set of vertices is exactly  the set of cone points of $N$.  As explained in \cite{Thurston} in genus $0$  (see \cite{Schwartz} for details) and \cite[$\S$10]{Veech} in arbitrary genus, one can find a graph  in the $1$-skeleton of $\mathscr T$,   such that its complement $Q$ in $N_{g,n}$ is simply connected.  $Q$ is a topological disk endowed with a flat metric whose boundary corresponds to consecutive edges of triangulation.  Let $F : Q \longrightarrow \mathbb{C}$ be  a developing map of the flat metric on $Q$ and let $q_1, \ldots, q_{2k}, q_{2k+1} =q_1 $ be  the vertices of  the boundary $\partial Q$  of the metric completion $\overline{Q}$ of $Q$ for the length distance induced by the flat structure of $Q$. The map $F$ extends continuously to $\overline{Q}$ and one sets $\xi_i = F(q_{i+1}) - F(q_i) $ for $i=1,\ldots,2k$.   The following proposition  holds true: 
\begin{prop}
\label{top}
For an appropriate choice of pairwise distinct indices $i_1, \ldots, i_{k-1} $ in $ \{ 1, \ldots, 2k\}$, the parameters $(\xi_{i_1}, \ldots, \xi_{i_{k-1}})$ form a linear parametrisation of\;${\mathcal F}_\rho$.
\end{prop}

Notice that if the triangulation $\mathscr T$ used to construct them was totally geodesic then these coordinates would form a polygonal parametrisation. The proof  
 uses arguments similar  to those of the proof of Proposition \ref{linear}.

\section{\bf GEOMETRIC PROPERTIES OF FLAT SURFACES \\ AND CHARACTERISTIC FUNCTIONS}
\label{geometricproperties}
In this section we develop material and prove several technical lemmas about the intrinsic geometry of flat surfaces which will be used in Sections \ref{completion}  and \ref{cusps} in order to understand the geometry of the moduli spaces $\F$. Most of the work done in this paper is about reinterpreting questions regarding the geometry of these moduli spaces, in terms of how  flat surfaces can degenerate. The material developed below goes some way to answering these questions.
\mk

We denote  by $h$ the flat metric on a given flat surface $N$ and  by $d_h$ (or just by $d$ for short) the induced distance (see Section \ref{properties}). We also denote by $S \subset N$ the set of conical points of $N$ (for the flat structure induced by $h$).

\subsection{Characteristic functions.}

We define four quantities associated to $N$:

\begin{enumerate}

\item[--] its \textbf{systole}\footnote{There in no systole when $g=0$.} : $$  \sigma(N) = \sigma(N,h)  = \inf_{}{ \Big\{L_h(\gamma)\ \big| \ \gamma \ 
 \text{simple}  \ \text{essential}  \ \text{closed}  \ \text{curve}  \Big\} }; $$

\item[--] its \textbf{relative systole}  : $$   \delta(N)= \delta(N,h) = \inf_{}{ \Big\{ L_h (\gamma) \ \big| \ \gamma \ 
 \text{joining}  \ \text{two} \ \text{distinct}  \ \text{singular}  \ \text{points}\Big\} }; $$

\item[--] its \textbf{diameter}: $$ D(N) = D(N,h) = \sup_{x,y \in N}{d_h(x,y)}; $$

\item[--] its \textbf{relative diameter} : $$  s(N) = s(N,h) = \sup_{x \in N}{d_h(x,S)}. $$
\end{enumerate}
(The terminology \textit{relative} is inspired by the terminology used for translation surfaces, where a \textit{relative period} of an abelian   form on a Riemann surface is the value of the integral of this $1$-form on a path linking two of its zeroes).   

Note that these four quantities all depend linearly on a rescaling of $h$. Most of the time, we will consider them under the supplementary assumption that the area of $N$ is 1. In this case, one gets geometric invariants attached to $N$.
\sk 

A classical fact from Riemannian geometry (see Section \ref{properties}) is that $\sigma, \delta$, $D$ and $s$ all are realised by piecewise geodesic paths, singular only at points where they cross singular points of $N$.

\begin{prop}
\label{compare}
The  following four  inequalities hold true:
\begin{align*}
& {\rm (1)}\; D(N) \geq \delta(N) ;  &&  {\rm (2)}\; D(N) \geq \sigma(N) /2;\\ 
& {\rm (3)}\;D(N) \geq s(N);  &&  {\rm (4)}\;  s(N)  \geq D(N)/(2n)\, .
\end{align*}
\end{prop}

\begin{proof} The first and third inequalities are obvious. We now prove the second one. Consider $c$ a curve realising $\sigma(N)$. Let $p$ and $q$ be two points on $c$ diametrically opposed (by this we mean that they cut $c$ into two parts of equal length). We claim that $d(p,q) = \sigma(N)/2$. Otherwise there would be a path of length strictly smaller than $\sigma(N)/2$ going from $p$ to $q$. This path completed with one of the parts of $c$ going from $p$ to $q$ would form an essential closed curve of length smaller than $\sigma(N)$. Since $d(p,q) = 
 \sigma(N)/2$, we have $D(N) \geq {\sigma(N)}/{2}$. 
\sk

Finally we prove (4). Let now $p$ and $q$ be two points realising  $D(N)$. The point $p$ can be joined to a  point $s_p\in S$ by a path of length at most $s(N)$, and $q$ to $s_{q} \in S$ by a path of length at most $s(N)$. Note that given $s'\in S$, there exists $s''\in S$ distinct from $s'$ which can be joined to the latter by a path of length at most $2s(N)$. 
 We now prove that one can join $s_p$ and $s_{q}$ by a path going from singular point to singular point with leaps of length less than $2s(N)$.
Let $\Gamma$ be the graph whose set of vertices is $S$
and for which there is an edge between two singular points if they are at distance less than $2s(N)$. We claim that  $\Gamma$ is connected. If not, 
let $(\Gamma_1, \Gamma_2)$ be a pair of two distinct connected components of $\Gamma$ with $d(\Gamma_1,\Gamma_2)$ minimal and pick 
two singular points $s_i\in \Gamma_i$ for $i=1,2$ such that $d(s_1,s_2)=d(\Gamma_1,\Gamma_2)$.
Since the distance between the two considered connected components has been chosen  minimal, the (piecewise) geodesic path realising the distance between $s_1 $ and $s_2$ cannot meet any other singular point on the way. If its length was more than $2s(N)$,  its middle point would be at a distance larger than $s(N)$ from the set of singular points which is impossible. This  means that we can build the announced path. Remark that we can find such 	a path which  visits each singular point only once. Such a path has length at most $(n-1)2s(N)$, hence $D(N) \leq 2s(N)+ (n-1)2s(N)=2ns(N)$.
 \end{proof}

\subsection{Voronoi decomposition and Delaunay triangulation.}
\label{S:VoronoiDelaunay}
We explain briefly a well-known but important construction in the realm of flat surfaces. 
We omit the proofs below and refer to \cite{MasurSmillie} for a careful and detailed treatment.  
\sk 

The \textbf{Voronoi decomposition} of $N$ is defined as follows:

\begin{itemize}

\item the $2$-cells are the connected components of the set of points $p\in N$ such that $d(p,S)$ is realised by a unique geodesic path; \sk

\item the $1$-cells are the connected components of the set of points $p\in N$ such that $d(p,S)$ is realised by exactly two distinct geodesic paths; \sk

\item  the $0$-cells are the connected components of the set of points $p\in N$ such that $d(p,S)$ is realised by at least three distinct geodesic paths.
\end{itemize}

\noindent It is checked in \cite{MasurSmillie} (see Proposition 4.1) that $0$-cells are points and $1$-cells are totally geodesic paths.
\sk 

The \textbf{Delaunay decomposition} is defined as the polygonal decomposition which is dual to the Voronoi decomposition in the following way. One checks that  $D_p$,  the Euclidean disk of radius $d(p,S)$,  injects at $p$ for any $p$ being a $0$-cell of the Voronoi decomposition. A Delaunay $2$-cell is defined 
 as  the convex hull of the elements of $S$ belonging to $\partial D_p$. A $1$-cell is a connected component of the boundary in $N\setminus S$ of such a convex hull and a $0$-cell is a element of $S$. \sk

 In \cite[Lemma 4.3 and Theorem 4.4]{MasurSmillie},  it is checked that :

\begin{itemize}

\item the set of $0$-cells is exactly $S$; \sk 

\item $1$-cells are saddle connections; \sk

\item for each $1$-cell $C_1$, there are two distinct $2$-cells $C_2$ and $C_2'$ such that $C_1 \sqcup C_2 \sqcup C_2'$ is a neighbourhood of $C_1$ in $N$; \sk

\item a Delaunay $2$-cell is isometric to a convex Euclidean polygon inscribed in a circle of radius less than $s(N)$; \sk

\item Delaunay $1$-cells have length smaller than or equal to $2s(N)$.

\end{itemize}

From the Delaunay decomposition (which is unique and only depends on the geometry of $N$) one can get a \textbf{Delaunay triangulation} by subdividing the $2$-cells into triangles. Notice that a Delaunay triangulation is not necessarily a simplicial triangulation since a triangle might not be determined by its vertices.  We have now as an immediate corollary of this construction and of Proposition \ref{compare}:

\begin{prop}
\label{triangulation}
The length of any 1-cell of any Delaunay triangulation of $N$ is always smaller than $2 D(N)$.
\end{prop}

We also prove the following lemma:
\begin{lemma}
\label{realising}
The interior of any path in $N$ realising $\delta(N)$ is a $1$-cell of the Delaunay decomposition of $N$ (hence is a 1-cell  of any Delaunay triangulation of~$N$).
\end{lemma}
\begin{proof} Remark that if a saddle connection is such that the only paths realising the distance of its middle point to $S$ are the two paths connecting the middle point to the end points, then it is a $1$-cell of the Delaunay decomposition. This is a direct consequence of the construction of the latter. We now check that such a saddle connection $\gamma$ realising $\delta(N)$ must verify the above property.

Assume that there is a second  path $u$ going from $p \in S$ to the middle point of $\gamma$ whose length is less than ${\delta(N)}/{2}$. The point $p$ must be different from one of the two endpoints of $\gamma$, and if concatenating the half of $\gamma$ starting from this point and $u$,  one gets a path $v$ of length less than $\delta(N)$ going from two distinct elements of $S$. Being singular at the middle of $\gamma$, $v$ can be shortened in order to get a path whose length is strictly less than $\delta(N)$ which is impossible. Therefore $\gamma$ must be a $1$-cell of the Delaunay decomposition.
\end{proof}

\subsection{Surfaces with large diameter.}

The aim of this subsection is to prove that flat surfaces with large diameter and finite linear holonomy must necessarily contain long flat cylinders. If one dismisses the hypothesis that the linear holonomy is finite, one can build counterexamples by gluing cones of very small angle. This was already known for spheres (see \cite{Thurston}) or when the linear monodromy ranges in $\{ -1, 1\}$ (see \cite[Corollary 5.5]{MasurSmillie}). The proof of Proposition \ref{diameter} below is highly inspired by the techniques developed in \cite{MasurSmillie}.
\mk

\paragraph{\bf Elementary facts about cones.}
We remind the reader that  $C_{\theta}$ stands for the (Euclidean) cone of angle $\theta\in ]0, +\infty[$, namely the metric space obtained by gluing the sides of a plane sector of angle $\theta$.  Its vertex is denoted by $0$ and one sets 
$C^*_{\theta}=C_{\theta}\setminus \{0\}$.  {This cone with the apex removed, does not contain closed regular geodesic but,  when $\theta < \pi$, it contains piecewise geodesic paths with only one angular point. More precisely:}

\begin{prop}
\label{cone}
If  $\theta < \pi$ then for any point $p\in C^*_{\theta}$: 
\begin{itemize}
\item there exists a unique closed simple piecewise geodesic  path in $C^*_{\theta}$ singular only at $p$;
\sk
\item the interior angle of the latter at the angular point is $\pi - \theta$;
\sk
\item the length of this piecewise geodesic path is $2 \sin({\theta}/{2}) \cdot d(0,p) $.  
\end{itemize}
\end{prop}  
\begin{proof} The proof of the proposition is straightforward after noticing that such a cone is obtained after doing the gluing pictured in Figure 
\ref{closedgeodesic}.
 \end{proof}
\begin{figure}[!h]
\centering
\psfrag{t}[][][1]{$\theta $}
\psfrag{p}[][][1]{$p $}
\psfrag{q}[][][0.7]{$\quad \pi-\theta $}
\includegraphics[scale=0.9]{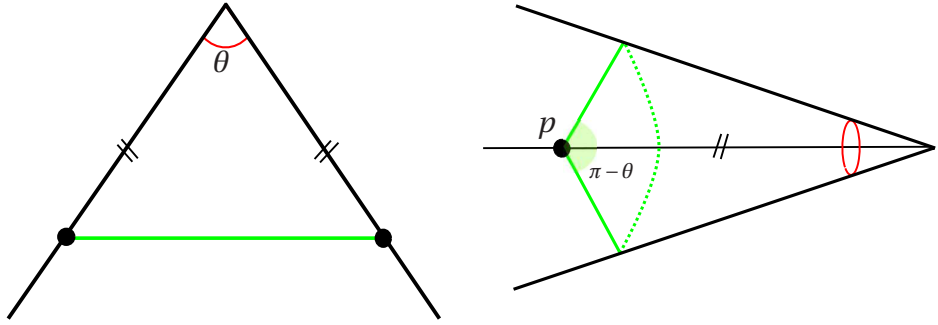}
\caption{The simple closed piecewise geodesic path 
with one angular point at $p$ on $C_{\theta}$ (in green).}
\label{closedgeodesic}
\end{figure}
%
%

\begin{lemma}
\label{distance}
Let $N$ be a flat surface and $\gamma$ be a piecewise geodesic path of length $L(\gamma)$ on $N$ with one angular point  which avoids conical points. 
Assume that in a small neighbourhood of its angular point, $\gamma$ cuts $N$  into two angular sectors of angles $\pi + \theta$ and $ \pi - \theta$ respectively, with $0 < \theta < \pi$.
  Then
\begin{enumerate}
\item  the linear holonomy along $\gamma$ is $e^{i\theta}$ or $e^{-i\theta}$;
\sk 
\item there is a cone point $q$ of $N$ such that $d(q, \gamma) \leq {L(\gamma)}/\big({2\tan ({\theta}/{2})}\big)$.
\end{enumerate}
\end{lemma}
\begin{proof} The point is that such a geodesic $\gamma$ has a neighbourhood that is isometric to a neighbourhood of the unique (up to isometry) closed geodesic of length $L(\gamma)$ of the cone of angle $\theta$. The only obstruction for this isometry to extend to the whole cone is that the boundary of its definition domain meets a singular point of $N$ (one can use the exponential map along $\gamma$). Otherwise $\gamma$ is on the cone of a cone point of $N$ whose associated conical angle is  $\theta$. In any case,  there is a singular point of $N$, 
whose distance to  $\gamma$ is less than the distance from the geodesic of length $L(\gamma)$ in $C_{\theta}$ to the cone point of $C_{\theta}$. This distance is exactly 
${L(\gamma)}/({2\tan ({\theta}/{2})})$.
\end{proof}

\begin{prop}
\label{diameter}
Let $\rho \in \mathrm{H}^1(N,  \mathbb{U}, \theta)$ be such that $\mathrm{Im}(\rho)$ is finite. There exist two positive constants $K_1(\rho)$ and $K_2(\rho) $ such that for every flat surface $N \in \mathcal{F}_{\rho}$ normalised such that its area is $1$, if \,$D(N) > K_1(\rho)$ then  $N$ contains an embedded flat cylinder of length at least $K_2(\rho) D(N) $.
\end{prop}

\begin{proof} Let $N$ be an element of $\mathcal{F}_{\rho}$. Let $p \in N $ be a point maximizing the distance $s$ to $S$ the set of singularities, \textit{i.e.} such that $s = s(N) = d(p,S)$ where $S \subset N$ stands for the set of singular points of $N$. Throughout the proof, we will mainly work with $s = s(N)$ which,  as a function,  is of same order as $D(N)$
according to the last two  points of Proposition \ref{compare}).
 
\sk

Let $r_p$ be the injectivity radius at $p$. Then $r_p <{1}/{\sqrt{\pi}}$ since the area of $N$ is one. If  $ s  > r_p $  then $\overline{D(r_p)}$,  the closed Euclidean  disk of radius $r_p$,  can be immersed in $N$ at $p$ (since $s$ is realised at $p$). There are two distinct points $a$ and $b$ on the boundary of $\overline{D(r_p)}$ which project onto the same point in $N$ and the immersion $i :\overline{D(r_p)} \longrightarrow N$ is injective on  ${D(r_p)}$, by definition of $r_p$. Therefore the chord joining $a$ and $b$ maps to a piecewise closed geodesic $\gamma$ path in $N$, with one angular point  at $i(a) = i(b)$.
\sk
 
We claim that if $s$ is large enough, then the linear holonomy along $\gamma$ must be trivial. This is a corollary of Lemma \ref{distance}. More precisely, if $\mathrm{Im}(\rho) = \left\langle e^{2i\pi/m} \right\rangle $ and $ s > r_p(1 + {\tan({\pi}/{m}})^{-1}) $, $\gamma$ cuts $N$ at $i(a)$ into two angular sectors both of angles $\pi$.  Hence $\gamma$ is a closed regular geodesic which belongs to a flat cylinder $C$ and the holonomy along $\gamma$ is $1$.  Moreover, $a$ and $b$ must be diametrically opposed and $\gamma$ must have length $2r_p$. Otherwise one side of the cylinder $C$ would be covered by $D(r_p)$. But then $r_p$ would not be the injectivity radius at $p$. The closed geodesic $\gamma$ contains $p$ and the cylinder $C$ containing $\gamma$ has length at least $2 \sqrt{s^2 - r_p^2}$, because any cylinder on a flat surface can be extended until its boundary meets a singular point. 
\sk 

 If one assumes that $s = s(N) \geq 2 /{\sqrt{\pi}} \geq 2r_p$, then the cylinder we have found has length at least $s(N)\sqrt{3}$ hence at least $ D(N){\sqrt{3}}/{(2n)} $ according to Proposition \ref{compare}. Recall that to ensure that the linear holonomy along $\gamma$ is trivial and therefore that $\gamma$ is a closed geodesic belonging to a cylinder, we had to assume $ s(N) > r_p\big(1 + {\tan({\pi}/{m}})^{-1}\big) $. 
 Since $r_p \leq \frac{1}{\sqrt{\pi}}$,  the 
  statement of the lemma follows if one takes $K_1(\rho) = \frac{2n}{\sqrt{\pi}} \max\big\{{2, \big( 1 + {\tan({\pi}/{m})^{-1}}\big) \big\}}$ and $K_2(\rho) = {\sqrt{3}}/{(2n)}$.
\end{proof}

\subsection{Collisions.}
A very important feature of the description of the metric completion of the moduli spaces $\F$ is to characterize geometrically what happens when two singular points collide, \textit{i.e.} when $\delta(N)$ goes to zero. 
We prove below two results  describing situations when such a collision cannot occur, at least without the diameter going to infinity.

\begin{lemma}
\label{collisions}
Let $\theta_1$ and $\theta_2$ be two positive angles such that $\theta_1 + \theta_2 < 2 \pi $. 
There exists a constant $K(\theta_1, \theta_2) > 0 $ such that if  $\Sigma$ is any flat sphere with $n$ conical singularities satisfying the three following conditions:
\begin{itemize}
\item all the cone points $p_1, \ldots, p_n$ of\;\,$\Sigma$ have positive curvature; 
\sk
\item the cone angles of\;\,$\Sigma$ at $p_1$ and $p_2$ are  
$\theta_1$ and $\theta_2$ respectively; 
 \sk
\item the area of \,$\Sigma$ is $1$;
\end{itemize}
then the following holds true:  $ d(p_1,p_2) \geq K\big(\theta_1, \theta_2\big) $.
\end{lemma}

This lemma  tells us that two  too positively curved singular points cannot collide.  We would like to draw  attention to the fact that, in the authors' opinion, this lemma (and in particular, its proof!) is  missing in \cite{Thurston}. \sk

\begin{proof}
The idea of the proof is to compare this situation to the case of the sphere $\Sigma^0$ with three cone points,  of respective angles $\theta_1$, $\theta_2$ and $2\pi - \theta_1 - \theta_2$. Such a sphere is unique up to dilatation of the flat metric and is built by gluing two isometric triangles of angles ${\theta_1}/{2}$, ${\theta_2}/{2}$ and $({2\pi - \theta_1 - \theta_2})/{2}$. \sk 

 Let $p_1^0$, $p_2^0$ and $p_3^0$ be the cone points on $\Sigma^0$ of respective angles $\theta_1, \theta_2$ and $2\pi - \theta_1 - \theta_2$. Normalise $\Sigma^0$ so  that the length of the unique geodesic $l^0$ from $p_1^0$ to $p_2^0$ has same the length as the one from $p_1$ to $p_2$ on $\Sigma$, denoted by $l$. Remark that $\Sigma^0$ is the disjoint union of geodesic paths going from $p_3^0$ to points of $l^0$.\sk

A neighbourhood of $l^0$ in $\Sigma^0$ is isometric to a neighbourhood of $l$ in $\Sigma$. We extend such an isometric identification using the remark above, developing the geodesics of the decomposition. 
The only obstruction to do so appears if such a geodesic meets a singular point, which can only happen for a finite number of such geodesics. 
We denote by $A$ the finite union of those parts of geodesics on which the isometry cannot be extended.  

We have thus defined a local isometry
$$i : \Sigma^0 \setminus A \longrightarrow \Sigma\, . $$

 Since all the singular points of $\Sigma$ have positive curvature, the closure of $i(\Sigma^0)$ must also be open and since $i$ is a local isometry, one gets 
$$ \mathrm{area}\big(\Sigma^0\big) \geq \mathrm{area}(\Sigma) = 1.$$
The uniform bound on the area of $\Sigma^0$ gives a uniform bound on $d(p_1^0,p_2^0)$ which equals $d(p_1,p_2)$ by construction. \sk
\end{proof}

\begin{lemma}
\label{second}
Let $M$ be a flat torus with two cone points $p_1$ and $p_2$. There exists a pseudo-hexagon $P$ such that $M$ is isometric to $P/ \sim$ where $\sim$ is one of the three gluing patterns of Figure \ref{patterns}. 
\end{lemma}

\begin{proof} 
Let $\Gamma$ be a connected graph in the $1$-skeleton of the Delaunay decomposition of $M$  such that $M \setminus \Gamma$ is connected and simply connected. $\Gamma$ has exactly for vertices the two cone points of $M$. By a Euler characteristic argument, its number of edges $e$ must satisfy $ 2 - e + 1 = \chi(M) = 0 $ and  therefore $e = 3$.\sk 

 One easily checks that the only connected graphs with two vertices and three edges that one can draw on a torus are the three  graphs  represented on Figure 
 \ref{Fig:GraphsOnATorus}. 
 Then cutting along $\Gamma$ gives the expected pseudo-polygonal model for $M$.
\end{proof}

\begin{prop}
\label{collisions2}
Let $(M_\ell)_{\ell \in \mathbb{N}}$ be a sequence of flat tori with two cone points belonging to a leaf\;$\F$ with ${\rm Im}(\rho)$ finite. Assume that for all $\ell \in \mathbb{N}$, $M_\ell$ has area $1$. If\,  $\lim_{\ell\rightarrow +\infty}\delta(M_\ell) = 0$ then $D(M_\ell) \rightarrow +\infty$ as $\ell$ goes to infinity.
\end{prop}
\begin{proof} Suppose that $(M_\ell)_{\ell \in \mathbb{N}}$ and $\rho$ are as in the statement and assume that  $D(M_\ell)$ does not go to infinity although $\delta(M_\ell)$ tends to zero when $\ell\rightarrow +\infty$. Then,  up to extracting an appropriate subsequence, we can assume that the $D(M_\ell)$'s  are bounded. 
 For any $\ell$, consider the Delaunay decomposition of $M_\ell$ and take in its $1$-skeleton a graph $\Gamma_{\!\!\ell}$ such that 

\begin{itemize}

\item $\Gamma_{\!\!\ell}$ contains a curve realising $\delta(M_\ell)$, as guaranteed by Lemma \ref{realising}; \sk

\item the set of vertices of $\Gamma_{\!\!\ell}$ is equal to $S$ the set of singular points;
\sk 

\item $Q_\ell=M_\ell \setminus \Gamma_{\!\!\ell}$ is simply connected.

\end{itemize}

 According to Lemma \ref{second}, for any $\ell \in \mathbb{N}$,  the metric completion $\overline{Q_\ell}$ of $Q_\ell $ is a pseudo-hexagon ({\it i.e.}\,a pseudo-polygon with six sides) whose lengths of the sides are uniformly bounded (according to Proposition \ref{triangulation}) and the gluing pattern to recover $M_\ell$ is one of the three patterns of Figure \ref{patterns}. 
\sk

Again up to extracting a subsequence,  we can assume that for any $\ell\in \mathbb N$: 
\begin{enumerate}

\item $M_\ell$ can be obtained from $\overline{Q_\ell}$ by using the same gluing pattern; \sk

\item the sides glued together always form the same angle ; 
\sk 

\item the length of each side converges.

\end{enumerate}
(To assume (2), one has to use that $\rho$ has finite image. 
That one can assume that  (3) holds true  as well follows from Proposition \ref{triangulation}.)\sk

Since the lengths of two sides go to zero (the ones which are identified by the gluing with the curve realising $\delta(M_\ell)$ in $M_\ell$), the sequence of pseudo-hexagons $(Q_\ell)_{\ell\in \mathbb N}$ converges to a quadrilateral whose opposite sides have the same length and therefore are parallel. Since ${\rm Im}(\rho)$ is finite,  this implies that  the corresponding sides in 
$\partial Q_\ell= \overline{Q_\ell}\setminus Q_\ell$ were parallel 
for all $\ell$ sufficiently large. This forces the gluing pattern to be Pattern $1$ or $2$ of Figure \ref{patterns}. But a hexagon glued with one of these pattern and having two pairs  of sides glued together parallel  must be a regular torus with no singular point (this is an easy exercise left to the reader). 
This would force $\F$ to contain regular tori, which is impossible since we have supposed that its elements have exactly two singular points. Therefore  the sequence of diameters $(D(M_\ell))_{n\in \mathbb N}$ must go to infinity as $\ell$ does. 
\end{proof}

\subsection{Closed curves realising the systole.}

As well as collisions, the ways in which simple closed curves can collapse are also very important to characterise.

\begin{lemma}
\label{L:RealisationOfTheSystol}
Let $N$ be a flat torus with $n \geq 2$ cone points and suppose that $p_1$ is the only cone point which has  negative curvature. The systole $\sigma(N)$ is realised by a  simple closed piecewise geodesic  which meets the set of cone-points only once at $p_1$. Moreover, the only point at which it might not be smooth is $p_1$.
\end{lemma}

\begin{proof} Consider the set of non homotopically trivial closed curves. Classical Riemannian geometry (see Section \ref{properties}) ensures
there  exists 
 a minimiser of the length functional on this  set and that it is piecewise geodesic. \sk 

 We claim that such a minimiser is simple. Otherwise it could be decomposed into two closed curves of strictly shorter length with at least one of these two being essential. \sk

 A minimiser cannot pass through a point of positive curvature because otherwise one can deform  it in order that it avoids the cone point and that its length is shorter. Therefore the only cone point it might pass through is $p_1$, and one can always make sure that there is a minimising path passing through $p_1$ : otherwise the path is actually totally geodesic and a neighbourhood of this path is a flat cylinder which can be extended until meeting a cone point which must be $p_1$. Any boundary component of this extended flat cylinder would be a required path. 
\end{proof}\mk 

A path realising $\sigma(N)$ cuts the surface at $p_1$ in two angle sectors, whose angle must be bigger than $\pi$ (otherwise one can shorten the path by passing on the side where the angle is smaller than $\pi$). Two possibilities can occur : 

\begin{enumerate}

\item one of the angle equals $\pi$; in this case such a path bounds a flat cylinder; \sk

\vspace{-0.35cm}
\item both angles are strictly bigger than $\pi$.\sk
\end{enumerate}

For our purpose, it is important to distinguish these two situations.

 In the case we are mostly interested in (when $g=1$ and $\theta=(\theta_i)_{i=1}^n$ is such that only the point of cone angle $\theta_1$ carries negative curvature), they actually correspond to two geometric aspects of $\F$ : flat tori verifying $(1)$ are in a cusp while those verifying $(2)$ are close to a stratum corresponding to the Devil's surgery ${\mathcal{S}}_3$, see Section \ref{devil}. 
The proposition below proves that, in the very specific case when $g=1$ and $\rho$ is rational,  if the diameter remains bounded and the systole goes to zero, we are in situation $(2)$.
\begin{prop}
\label{systol}
Assume that $g=1$, $\theta=(\theta_i)_{i=1}^n$ is such that only \ $\theta_1$ is bigger than $2\pi$ and $\rho \in \mathrm{H}^1(N, \mathbb{U}, \theta)$ has finite image.
For all $K > 0$, there exists a constant $\epsilon(K) > 0$ such that for $N \in \F$ normalised such that its area equals 1 the following holds true.
If $D(N) \leq K$ and $ \sigma (N) \leq \epsilon(K)$, then any curve $c$ realising the systole and passing through $p_1$ the point of negative curvature of $N$  cuts $p$ into two angular sectors whose angles both  are strictly larger
than $\pi$.\end{prop}

\begin{proof} We argue by contradiction.
Assume that there exist a constant $K$ and a sequence $(N_m)_{m\in \mathbb N}$ of flat tori such that for all $m \in \mathbb{N}$, one has : 
\begin{itemize}

\item $N_m \in \F$  and its area is equal to 1; \sk

\item $\sigma(N_m) \leq \frac{1}{m}$; \sk

\item  $N_m$ contains a cylinder $C_m$ of width $\sigma(N_m)$ (\textit{i.e.} we are in situation (1) described above); \sk

\item $D(N_m) \leq K $.
\end{itemize}

For all $m \in \mathbb{N}$, $N_m \setminus C_m$ is a sphere whose boundary is the union of two piecewise geodesic closed curves of the same length $\sigma(N_m)$ touching at the only point where they both are singular, namely $p_1(m)$ the cone point of negative curvature of $N_m$. One can cut at the point where the two boundary circles touch, and glue together the two geodesic parts of the new boundary circles (which have the same length) to get a flat sphere $S_m$. Since the cone angle $\theta_1$ at $p_1(m)$  is supposed to belong to  $]2\pi,4\pi[$, the resulting sphere has only positively curved cone points. The angles 
$\theta'_1(m)$ and $\theta'_2(m)$ 
at the two `new' cone points of $S_n$ created by the previous cutting and pasting operation, must satisfy 
$ \theta'_1(m) + \theta'_2(m) + \pi + \pi = \theta_1 $. 
Since $\theta_1$ is strictly smaller than $4\pi$, we get: 
$$ \theta'_1(m) + \theta'_2(m)  < 2\pi \, .$$

For $k = 1,2$, the cone angle $\theta'_k(m)$  must be such that $e^{i\theta'_k(m)} \in \mathrm{Im}(\rho)$ because $e^{i\theta'_k(m)} $ is the linear holonomy of a curve in the free homotopy class of the curve realising the systole. Therefore these two angles  can take only a finite number of values.  So,  up to extracting a subsequence, one can assume that these two cone angles  are independent of $m$. 
\noindent The fact that the sequence of diameters $(D(M_m))_{m\in \mathbb \mathbb{N}}$ is bounded by $K$ implies  that the length of $C_m$ is bounded by $2K$. Therefore the area of $S_m$,  which is larger than $1 -2  \sigma(N_m)K$,   is bigger than ${1}/{2}$ provided that $m$ is large enough. Hence  $ \theta'_1(m) + \theta'_2(m)  < 2\pi $ and the distance between the associated cone points, which equals $\sigma(N_m)$ by construction, goes to zero. This contradicts Lemma \ref{collisions} and proves the proposition. \end{proof}

\section{\bf SURGERIES}
\label{surgeries}

 A {\bf surgery} is a procedure  through which a new flat surface with conical singularities is produced from another one by means of geometrical gluing and pasting relying on elementary Euclidean geometry. 
  This notion naturally appears  when studying moduli spaces of flat surfaces (implicitly in \cite{Thurston} but also more explicitly  in \cite{KontsevichZorich}).  
    In the previous section, we have studied different ways for flat surfaces to degenerate, namely sequences of surfaces containing very large embedded flat cylinders, essential curves which collapse or cone points colliding. 
     In the present section, we introduce several surgeries which are to be seen as the inverse processes of the aforementioned degenerations. \sk 

 We will distinguish five distinct types of surgeries: 
\sk 
\begin{itemize}

\item the first one, denoted by ${\mathcal{S}}_1$, was known and implicitly considered by Thurston in \cite{Thurston}. It consists in blowing up a singular point of positive curvature into two singular points of positive curvature; \sk

\item  the second surgery, denoted by ${\mathcal{S}}_2$,  is a straightforward generalisation of the first one, which allows to blow up points of negative curvature. 
We will therefore refer to both ${\mathcal{S}}_1$ and ${\mathcal{S}}_2$ as \textbf{Thurston's surgeries};  \sk

\item it seems to us that the third surgery ${\mathcal{S}}_3$ is new. We call it the \textbf{Devil's surgery}. 
It consists in creating a handle by removing the neighbourhoods of two singular points and gluing their boundaries together; 
\sk 

\item the fourth surgery ${\mathcal{S}}_4$ consists in blowing up a regular point into three singular points. We call it the \textbf{Kite surgery}; \sk

\item  the last surgery ${\mathcal{S}}_5$ consists in creating a handle by adding a long flat cylinder to any flat surface having two isometric totally geodesic boundary components.
\bk 
\end{itemize}

Surgeries  ${\mathcal{S}}_1$,  ${\mathcal{S}}_2$ and  ${\mathcal{S}}_4$ could have been seen as the same in a more general presentation but
we find more convenient to differentiate them for our purpose.  
At the end of this Section,  we compute the signature of the area form in the case we are interested by  using surgeries and we give a definition of the notion of \textit{geometric convergence} which will be central in the description of the metric completion of $\F$ carried on in Section \ref{completion}.

\subsection{ Thurston's surgery ${\mathcal{S}}_1$ for a cone angle smaller than $2\pi$.}
\label{Thsurgery}
Let $N$ be a flat surface of genus $g$ with cone angles $\theta = (\theta_1, \ldots, \theta_n)$ at $p_1, \ldots, p_n \in N$. Let $\F$ be the leaf of Veech's foliation to which $N$ belongs. Suppose that $ \theta_1 < 2\pi$ and let $\theta_1'$ and $\theta_1''$ be two angles smaller than $2\pi$ such that 
\begin{equation}
\label{E:AnglesSurgery1}
 2\pi - \theta_1 = \big(2\pi - \theta_1'\big) + \big(2\pi - \theta_1''\big). 
 \end{equation}

 In this subsection, we describe  a surgery building out  flat surfaces of genus $g$ with $n+1$ singular points of cone angle $ \theta' = (\theta_1', \theta_1'', \theta_2, \ldots, \theta_n)$ from $N$ (note that because we have assumed \eqref{E:AnglesSurgery1},  the new angle datum $\theta'$ still satisfies Gau\ss-Bonnet formula  \eqref{E:GBformula}).  
  The surgery is local on $N$, in the sense that it is performed on a small neighbourhood of $p_1$ without modifying the rest of the surface.  
\mk 

 Choose a point $p$ in a small neighbourhood $C$ of $p_1$ isomorphic to the  
 {portion of  cone} $C_{\theta_1}(\epsilon)$ for a certain $\epsilon>0$ (see \S\ref{S:Generalities}). As $\theta_1'$ is bigger than $\theta_1$, there are exactly two distinct segments of the same length issuing from $p$ which meet at their endpoints and form an interior angle equal to $2\pi - \theta_1'$ at $p$ (see Figure \ref{S1}).  
\begin{figure}[!h]
\centering
\psfrag{p}[][][1]{$p$}
\psfrag{s}[][][1]{$\textcolor{blue}{\frac{\theta_1''}{2}}$}
\psfrag{t}[][][1]{$\textcolor{blue}{\frac{\theta_1''}{2}}$}
\psfrag{t1}[][][1]{$\theta_1'$}
\psfrag{t1p}[][][1]{$\theta_1'$\,}
\psfrag{p}[][][1]{$p$}
\psfrag{B}[][][1]{$B$}
\psfrag{p1}[][][1]{${p_1}$}
\psfrag{1}[][][1]{$(1)$}
\psfrag{2}[][][1]{$(2)$}
\psfrag{2pim}[][][0.75]{$2\pi-\theta_1\,$}
\includegraphics[scale=1]{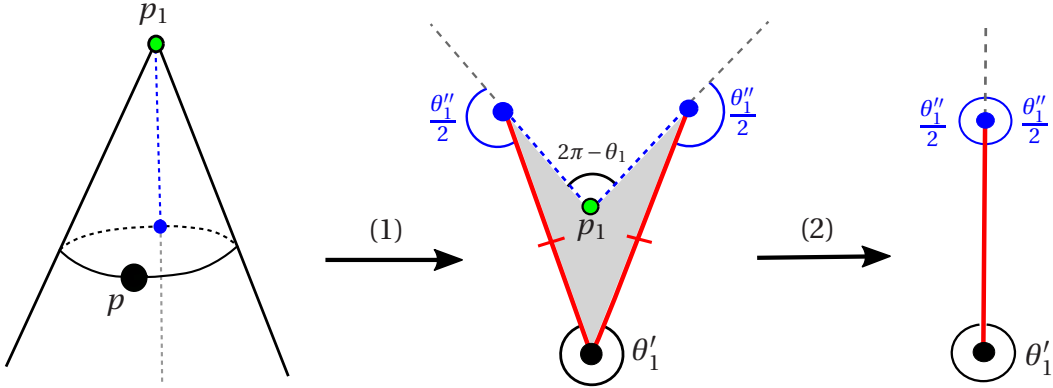}
\caption{Thurston's surgery ${\mathcal{S}}_1$ consists in (1) cutting the surface along the dashed blue segment; then (2) removing the grey piece of the surface and gluing the two red segments together.}
\label{S1}
\end{figure}
\sk 

 The surgery works the following way : delete the bigon on $C$ 
 which corresponds to the quadrilateral $B$ in grey on Figure \ref{S1}. Its sides are two geodesics which have the same length and the same endpoints $p_1'$ and $p_1''$. Removing the bigon and gluing these two segments together, one gets a new flat surface $N'$ having two cone points of angle $\theta_1'$ and $\theta_1''$ at $p_1'$ and $p_1''$.

 Recall that $\mathscr{F}_{[\rho]}$ is the leaf of Veech's foliation to which $N$ belongs. There exists a neighbourhood $U$ of $N$ in $\mathscr{F}_{[\rho]}$ and $\epsilon > 0$  sufficiently small so that for all flat surfaces in $U$, the previous surgery can be performed for all $p$ in a disc of radius $\epsilon$ centered at $p_1$,  the cone point of angle $\theta_1$. Remark that the class $[\rho']$  such that $N' \in \mathscr{F}_{[\rho']}$ does not depend on the choice of $N$ in $U$.  
 \sk 
 
 This allows us to define a map
\begin{align}
\label{E:S1OnAProduct}
{\mathcal{S}}_1\, : \;  C_{\theta_1}^*(\epsilon)  \times U  &\longrightarrow  \mathscr{F}_{[\rho']} \\
  \big(p,  N\big)  & \longmapsto N' \, , \nonumber
\end{align}
where $C_{\theta_1}^*(\epsilon)$ is $C_{\theta_1}(\epsilon)$ minus 
its apex (see Section \ref{flatsurfaces}). This definition requires an identification of $ C_{\theta_1}(\epsilon)$ {with a neighborhood of the cone point of angle $\theta_1$
in each flat surface} element of $U$. We do this by choosing a geodesic path $c$ joining $p_1$ and $p_2$. This path survives in a neighbourhood of $N$ in $U$. We decide that  $ C_{\theta_1}(\epsilon)$ is embedded in an element of $U$ in such a way that it always meets the previous geodesic path in the same locus - this latter requirement defining unambiguously such an embedding. Let $z_0 \in \mathbb{C}$ be a (germ of) linear parametrisation of $C_{\theta_1}^*(\epsilon)$ such that $z_0 \in \mathbb{R}_+$ if and only if the corresponding point $p$ belongs to the aforementioned geodesic path joining $p_1$ to $p_2$. \sk 

 There is a little ambiguity for the choice of the path $c$ whenever some elements of $U$ have non-trivial isometries; in that case the identification of $C_{\theta_1}^*(\epsilon)$ on elements of $U$ cannot be made continuous.  
 We chose to  ignore this difficulty for a moment and then we will address it in Remark \ref{onetoone} below.

\begin{prop}
We use the notations introduced just above. 
\label{param}
\begin{enumerate}
\item If $(z_1, \ldots, z_m)$ is a linear parametrisation of $U$ then  $(z_0, z_1, \ldots, z_m)$ is a linear parametrisation of $U' \subset \mathscr{F}_{[\rho']}$.
\sk 
\item The map ${\mathcal{S}}_1$ is a local biholomorphism.
\sk 
\item If all elements of $U$ have no non-trivial isometry then ${\mathcal{S}}_1$ is one-to-one.
\end{enumerate}
\end{prop}

\begin{proof} Consider a topological polygonal model of $N \in U$ which is such that $c$ is an edge of this polygon. This model can be extended to a model of $N'$ by adding a point on the side representing the class of $c$, see Figure \ref{F:surgery3}.
\begin{figure}[!h]
\centering
\psfrag{C}[][][1]{$C$}
\psfrag{p2}[][][1]{$p_2$}
\psfrag{p1}[][][1]{$\;\;\; \textcolor{red}{p_1}$}
\psfrag{p1pp}[][][0.9]{$ \!\!\!\!\!\!\textcolor{blue}{p_1'}$}
\psfrag{p1p}[][][0.9]{$\!\! \textcolor{blue}{p_1''}$}
  \includegraphics[scale=0.66]{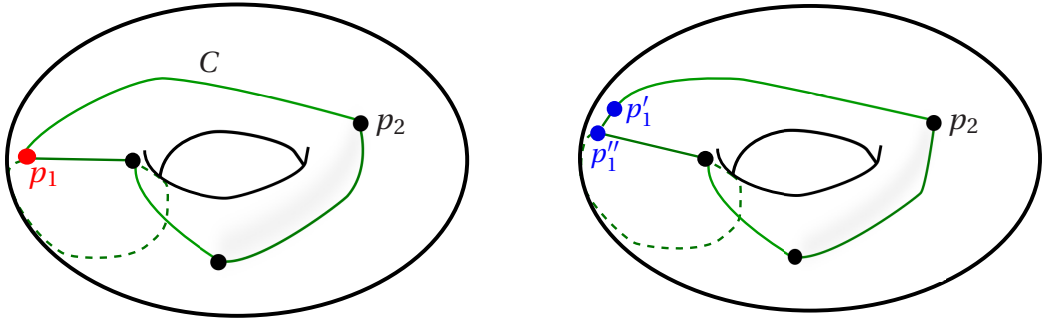}
 \caption{The surface $N$ before surgery ${\mathcal{S}}_1$ on the left and the  surface $N'$ obtained after surgery on the right.} 
 \label{F:surgery3}
\end{figure}
\sk

Let $(z_1, \ldots, z_m)$ be a linear parametrisation  of $U$ such that the geodesic path $C$ from $p_1$ to $p_2$ develops on $z_1$, and let $z_0$ be the complex number onto which the geodesic path from $p_1$ to $p$ (which is going to become $p'_1$ after the surgery) develops. Let $N'$ be a flat surface obtained after applying a ${\mathcal{S}}_1$ surgery to $N$. Let $(w_0, \ldots, w_m)$ be a linear parametrisation of a neighbourhood of $N'$ in $\mathscr{F}_{[\rho']}$ associated to the extended polygonal model of $N$ such that  $w_0$ represents the shortest geodesic path from $p_1'$ to $p_1''$ on $N'$ and $w_1$ represents a path joining $p_1'$ to $p_2$, $w_2$ a path joining $p_1''$ to $p_3$ (note that $p_3$ and $w_2$ do not appear explicitly on Figures \ref{F:surgery3} and \ref{F:surgery4}). 

\begin{figure}[!h]
\centering
\psfrag{p1}[][][1]{$p_1$}
\psfrag{p1p}[][][1]{$p_1'$}
\psfrag{p1pp}[][][1]{$\;p_1''\quad$}
\psfrag{q1pp}[][][1]{$p_1''$}
\psfrag{p2}[][][1]{$p_2$}
\psfrag{t1p}[][][1]{$\theta_1'$\;}
\psfrag{2pimt1}[][][1]{$2\pi-\theta_1$}
\psfrag{t1pp2}[][][1]{$\frac{\theta_1''}{2}$\;}
\psfrag{m}[][][1]{$\frac{\theta_1''}{2}$}
\psfrag{z0}[][][1]{$\textcolor{green}{z_0}\;\;\,$}
\psfrag{z1}[][][1]{$z_1\;$}
\psfrag{w0}[][][1]{$\;\;\textcolor{red}{w_0}$}
\psfrag{w1}[][][1]{$\textcolor{blue}{w_1}$\;\;}
 \includegraphics[scale=1.2]{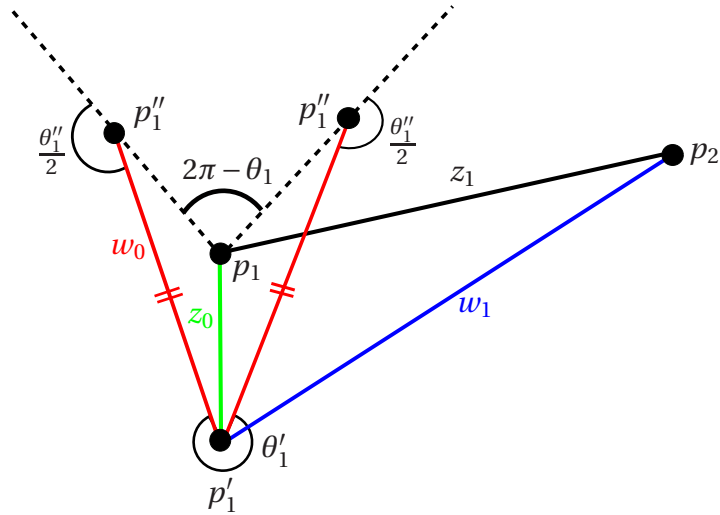}
 \caption{A superposition of the parametrisations $w$ and $z$ before and after surgery, near $p_1$.}
\label{F:surgery4}
\end{figure}

According to Figure \ref{F:surgery4}, which is a superposition of the developing maps of $N$ and $N'$ near the point $p_1$, where the parametrisations $w$ and $z$ correspond respectively to before (the surface $N$) and after (the surface $N'$) proceeding to the surgery, 
the following relations hold true 
$$
 w_{0}  = \rho_0 z_{0}, \quad   
 w_ 1 =z_1 + z_0 \quad  \mbox{ and } \quad
 w_2  = z_2 + \rho_1z_0 \, , $$ 
 where $\rho_0$ and $\rho_1$ are constants (which can be made explicit by means of elementary geometry of Euclidean triangles) depending only on $\theta_1, \theta_1'$ and $\theta_1''$. All the other $w_i$'s can be expressed in a similar fashion. Furthermore, if $w_i$ represents a path involving end points different from $p_1'$ and $p_1''$, it is equal to one of the $z_j$'s.

Therefore $(z_0, z_1, \ldots, z_m)$ is a linear parametrisation of a neighbourhood of $N'$ in $\mathscr{F}_{[\rho']}$. This implies directly the two first points of the proposition, in particular the fact that ${\mathcal{S}}_1$ is a local biholomorphism. It remains to prove the injectivity of ${\mathcal{S}}_1$ under the additional hypothesis that all the  elements of $U$ have a trivial isometry group. The length of the shortest path from $p_1'$ to $p_1''$ (which equals $|z_0|$ provided that the latter  is small enough in the area $1$ normalisation) is a geometric invariant. The surface $N$ from which $N'$ is obtained from also is  a geometric invariant. Assume that there exist two points $p$ and $p'$  on $C_{\theta_1}^*(\epsilon)$ such that the resulting surfaces from the surgery at $p$ and $p'$ are the same. This would imply that the initial surface has an isometry fixing $p_1$ and sending $p$ to $p'$. The (pure) isometry group of a surface being finite, ${\mathcal{S}}_1$ is a local biholomorphism which is one-to-one if all the elements of $U$ are all isometry free. 
\end{proof}

\begin{rem}
\label{onetoone}
\rm $(1)$ It is worth giving a more abstract and intrinsic definition of the surgery introduced above. Let $U\subset \F$ as above and assume that none of its elements 
admits a nontrivial  isometry. Then there exists a `{\it universal flat curve $\nu_U: \mathcal T_U\rightarrow U$ over $U$}' : it is  a map  such that
the fiber over  a flat surface $N$ viewed as a point of  $U$ is $N$ itself,  but this time viewed as a 2-dimensional flat surface with conical singularities.  This family of surfaces comes with $n$ sections $p_i: U\rightarrow \mathcal T_U$ which are such that $p_i(N)$ is the $i$-th cone point of $N$ for every $i=1,\ldots,n$.  One denotes by $P_i$ the image of $p_i$ for every $i$ and by 
$\mathcal T_U^*=\mathcal T_U\setminus \cup_{i=1}^n P_i$ the `{\it $n$-punctured universal flat curve over $U$}'
\sk

\rm


Within this formalism, one can verify that Thurston's surgery $\mathcal S_1$ admits an intrinsic definition  as the (germ of) map $(\mathcal T_U^*, P_1) 
\rightarrow  \mathscr F_{[\rho']}$ which, for any $p\in \mathcal T_U^*$ sufficiently close to $P_1$  associates the flat surface $N'$ obtained by performing the surgery described by Figure \ref{S1} above on the surface $N=\nu_U(p)$ with respect to $p$ and the cone point $p_1(N)$. 
Clearly, obtaining   the more explicit definition \eqref{E:S1OnAProduct} just amounts to trivializing $\mathcal T_U\rightarrow U$ along $P_1$. \sk

\rm $(2)$ The interest of the preceding, more conceptual, approach is that it 
points out the main issue  when some of the elements of $U$ admit nontrivial isometries and how to deal with it.  Indeed, in this case, there is no universal curve over $U$ but one
 will exist over a non-trivial orbifold cover $\widetilde{U}$ of $U$  and working with the latter,  one can define Thurston's surgery the same way than above. \sk 

\rm For instance and more concretely,  if  $N_0 \in U$ 
is such that ${\rm PIso}^+(N_0)$ is  non-trivial, then it is necessarily cyclic of finite order,  say $m$, according to \S\ref{SS:IsometryGroup}. In this case there exists $\tilde{U}\rightarrow U$ an orbifold cover of order $m$ of $U$,   whose deck transformation group  is isomorphic to the isometry group of $N_0$, on which the identification of $C_{\theta_1}^*(\epsilon)$ with some neighborhoods of the corresponding cone points 
in flat surfaces belonging to $U$ can be made in a continuous way. 

\rm Therefore the surgery  still defines a map 
\begin{align*}
{\mathcal{S}}_1 \; : \, C_{\theta_1}^*(\epsilon)  \times \widetilde{U} & \longrightarrow  \mathscr{F}_{[\rho']} \\
 \big(p,  N\big)  & \longmapsto   N'
\end{align*} 
which is equivariant under the action of the isometry group of $N_0$. 
\end{rem}

\subsection{Thurston's surgery ${\mathcal{S}}_2$ for a cone angle greater than $2\pi$.}
Assume now that $\theta_1 > 2\pi$. Let $\theta_1' > 2\pi$ and $\theta_1''< 2\pi$ be such that 
$$ 2\pi - \theta_1 = \big(2\pi - \theta_1'\big) + \big(2\pi - \theta_1''\big)\, . $$

Let  $V_1$ be a neighbourhood  of $p_1$  isometric to a portion of cone $C_{\theta_1}(\epsilon)$ for a certain $\epsilon > 0$.
  Define $\eta = \theta_1 - 2\pi$ and let $p$ be a point of $V_1$. If $p$ is close enough to the singular point $p_1$ there is a unique $4$-gon $P$ in $V_1$ having the following properties (see Figure \ref{surgery1}) :  

\begin{itemize}

\item $p$ and $p_1$ are opposite vertices of $P$ 
and the external angles of the latter at these two points are 
$\theta_1''$ and $\eta$ respectively; 
\sk
\item the external angles at  the two other vertices of $P$ both are ${\theta'_1}/{2}$;
\sk 

\item the sides of $P$ meeting at $p$ (resp.\;at $p_1$) have the same length. \mk
\end{itemize}


\begin{figure}[!h]
\centering
\psfrag{A}[][][1]{
\begin{tabular}{l}{\it A neighbourhood of}  \\ {\it the singular point $p_1$}
\end{tabular}}
\psfrag{B}[][][1]{
\begin{tabular}{l}  {\it Surgery creating two points}\\ {\it with cone angles $\theta_1'$ and $\theta_1''$}
\end{tabular}
}
\psfrag{tpp}[][][1]{$ \textcolor{red}{\theta_1''}$}
\psfrag{tp2}[][][1]{$ \textcolor{blue}{\frac{\theta_1'}{2}}\; $}
\psfrag{ttt}[][][1]{$\textcolor{blue}{\frac{\theta_1'}{2}}$}
\psfrag{p1}[][][1]{$p_1$}
\psfrag{p}[][][1]{$ \textcolor{red}{p}$}
\psfrag{p11}[][][1]{$p_1$}
\psfrag{P}[][][1]{$P$}
\psfrag{2p+e}[][][1]{$2\pi+\eta\; $}
\psfrag{e}[][][1]{$\eta$}
\includegraphics[scale=0.9]{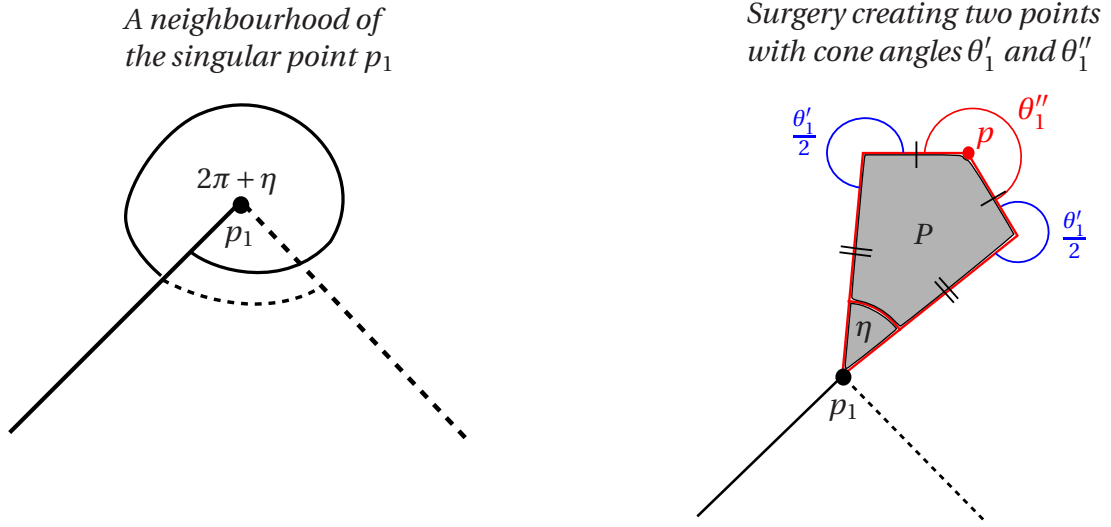}
\caption{Thurston's surgery ${\mathcal{S}}_2$. }
\label{surgery1}
\end{figure}

We build a new flat surface $N'$ in the following way : we remove the interior of $P$ and glue together the sides meeting at $p_1$ and $p$. We obtain a flat surface with a singularity of angle $\theta_1''$  and another singularity of angle $\theta_1'$. As when $\theta_1 < 2\pi$,  the class $[\rho']$ such that $N' \in \mathscr{F}_{[\rho']}$ does not depend on the choice of $N$. For a neighbourhood $U$ in $\F$, this allows us to define  a map 
\begin{align*}
{\mathcal{S}}_2 \;  : \,  C^*_{\theta_1}(\epsilon)  \times U & \longrightarrow  \mathscr{F}_{[\rho']} \\
  \big(p, N\big)  & \longmapsto  N'  \, . 
\end{align*}

Similarly to the case when $\theta_1 < 2\pi$, the map ${\mathcal{S}}_2$ is a local biholomorphism on its image, and for any linear parametrisation $(z_1, \ldots, z_m)$ of $U$, $(z_0, z_1, \ldots, z_m)$ is a linear parametrisation of the image of ${\mathcal{S}}_2$, where $z_0$ is the complex number on which the segment $[p_1,p]$ develops. The proof is exactly the same as in the $\theta_1 < 2\pi$ case and is left to the reader. 
 A similar remark to Remark \ref{onetoone} also holds true for ${\mathcal{S}}_2$ as well.
\sk 

 The following remark will play a crucial role in the proof of Proposition \ref{complet} which is one of the main results of the paper.
\begin{rem}
\label{formula}
{\rm 
If $l$ is the length of the segment between the points of angle $\theta'_1$ and $\theta''_1$ and $L$ the length  between $p_1$ and the point of angle $\theta'_1$ (after surgery), then  }
\begin{equation}
\label{E:m=Equation*l}
 L = \frac{ \sin \left(\frac{\theta_1''}{2}\right)}{\sin \left(\frac{\theta_1' + \theta_1''}{2}\right)} l \, .
 \end{equation}

{\rm Given a flat surface $N$ with two conical points $p_1'$ and $p_1''$ of angles $\theta'_1>2\pi$ and $\theta''_1<2\pi$ which are linked by a saddle connection $c$ of length $l$, one can wonder when it is possible to reverse Thurston's surgery ${\mathcal{S}}_2$. The only obstruction to doing so is that $c$ can be extended on the side of $p_1'$ (the point of negative curvature) on a distance  
equal to the right-hand side of \eqref{E:m=Equation*l},  
 while cutting $\theta_1'$ in half.  The proof of this claim is elementary and left to the reader.}
\end{rem}

\subsection{The Devil's surgery ${\mathcal{S}}_3$}
\label{devil}
Let $\theta_1', \theta_1'', \theta_2, \ldots, \theta_n$ be the respective cone angles at the cone points $p_1', p_1'', p_2, \ldots, p_n$ of a flat surface $N$. We make the simplifying assumption that $\theta_1'$ and $\theta_1''$ are (strictly) less than $\pi$, but we will see later on that the surgery we are going to describe can still be performed for 
 $\theta_1'$ and $\theta_1''$ less than $2\pi$. 
\mk 

Let $C_1'$ and $C_1''$ be two neighbourhoods of $p_1'$ and $p_1''$, isometric to 
$C_{\theta_1'}(\epsilon)$ and $C_{\theta_1''}(\epsilon)$ respectively, with $\epsilon>0$ sufficiently small. 
Consider two points {$q'$ and $q''$} in $C_1'$ and $C_1''$ respectively such that the unique closed geodesic paths $c'$ and $c''$ in $C_1'$ and $C_1''$ respectively,  singular only at  $q'$ and $q''$ have the same length. Removing the `upper  parts' of the two cones $C_1'$ and $C_1''$
 by cutting along $c'$ and $c''$ respectively, one can glue $c'$ and $c''$ isometrically in such a way that $q'$ and $q''$ are glued together. One gets a flat surface $\widehat N$ of genus $g+1$ with $n$ cone points of angle $\theta_1, \theta_2, \ldots, \theta_n$ with $\theta_1 = 2\pi + \theta'_1 + \theta_1'' $.
\mk


\begin{figure}[!h]
\centering
\psfrag{Beta}[][][1]{$\textcolor{red}{\beta} $}
\psfrag{H}[][][1]{$\textcolor{red}{\widehat \beta} $}
\psfrag{1}[][][0.9]{$\;\quad p_1'$}
\psfrag{2}[][][0.9]{$\;\;\; p_1''$}
\psfrag{c1}[][][0.8]{$\textcolor{green}{c'}$}
\psfrag{c2}[][][0.8]{$\textcolor{green}{c''}$}
\psfrag{w1}[][][0.8]{$\textcolor{black}{\pi+\theta_1'}\quad \,  \qquad $}
\psfrag{w2}[][][0.8]{$\quad \; \textcolor{black}{\pi+\theta_1''}$}
\psfrag{qp}[][][0.9]{$q'$}
\psfrag{qpp}[][][0.9]{$q''$}
\psfrag{t1}[][][0.6]{$\, \theta_1'$}
\psfrag{t2}[][][0.6]{$\; \theta_1''$}
\psfrag{p1}[][][0.7]{${p_1}$}
\psfrag{p2}[][][0.9]{$p_2$}
\psfrag{p3}[][][0.9]{$p_3\; $}
\includegraphics[scale=0.57]{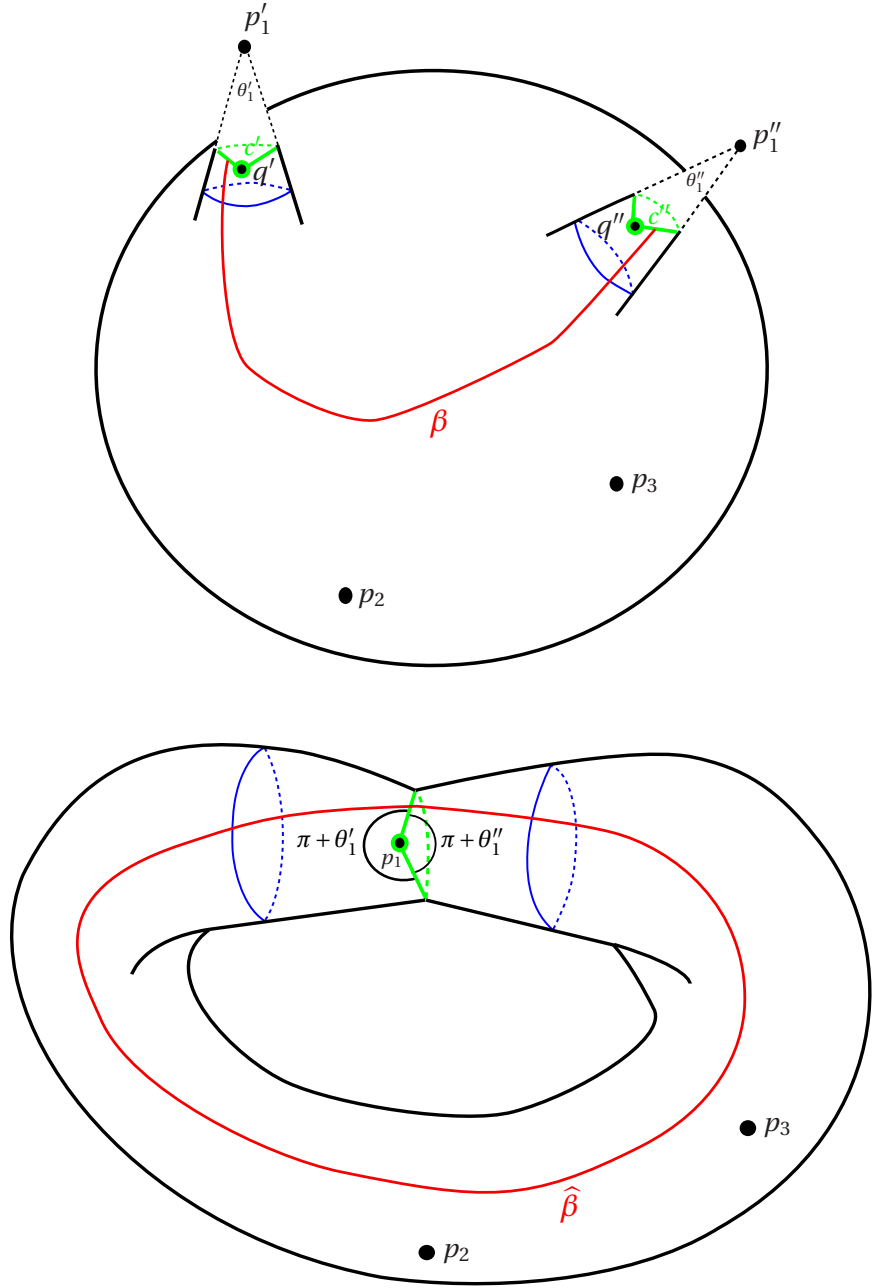}
\caption{The Devil's surgery from a flat sphere to a flat torus.}
\label{gluing}
\end{figure}


Let $\beta$ be a simple curve in $N$, avoiding its conical points, joining two regular points of $c'$ and $c''$ respectively which are glued together by the identification considered above. Denote by 
$\widehat \beta$  the simple loop in the regular part of $\widehat N$ obtained from $\beta$ by this gluing.
   We want to perform the surgery in such a way that the linear holonomy of the resulting flat 
  surface does not change when $q'$ and $q''$ move 
  on the portions of cones $C_1'$ and $C_1''$ respectively  and equals a certain (class under the action of the pure mapping class group of) $\widehat \rho$. It is equivalent to the fact that the holonomy along $\widehat \beta$ does not vary, because the holonomy of the resulting flat surface 
  $\widehat N$
  is totally determined by the holonomy of the original flat surface $N$ and the holonomy along 
  $\widehat \beta$.\sk


 When $q'$ moves on the circle it belongs to (namely, the set of points on 
 $C_1'$ whose distance to the apex $p_1'$ is precisely $d(p_1',q')$), 
(a lift to $\mathbb R=\widetilde{ \mathbb U}$ of)
  the linear holonomy  along 
$\widehat \beta$  increases exactly by the angle that $q'$ makes relatively to its initial position. Hence if we want to keep the linear holonomy  along $\widehat \beta$ constant, we have to move $q''$ by the same angle as $q'$. 

This allows us to build a map : 
 \begin{align*}
\widetilde{{\mathcal{S}}_3} \,  : \,  \widetilde{C^*_{\theta_1'}(\epsilon)} \times U & \longrightarrow  \mathscr{F}_{[\widehat \rho]} \\
\big(p,  N\big)  & \longmapsto   \widehat N
\end{align*}
where $U$ is an open subset of $\F$,  the moduli space to which the original surface belongs.

\begin{prop}

Suppose that every element of $U$ has no non-trivial isometry. Then $\widetilde{{\mathcal{S}}_3}$ is a covering map onto its image. If $\theta_1'$ and $\theta_1''$ are not commensurable (\textit{i.e.} if ${\theta_1'}/{\theta_1''} \notin \mathbb{Q}$) then it is a biholomorphism onto its image.\sk

  However, if  ${\theta_1'}/ {\theta_1''}=k/l$ for some coprime positive integers  $k$ and $l$, then the deck group is the group generated by the rotation of angle $l \theta_1' = k \theta_1''$. 
   In particular $\widetilde{{\mathcal{S}}_3}$ factors through 
 \begin{align*}
{\mathcal{S}}_3 \,  : \,  C^*_{l\theta_1'}(\epsilon) \times U & \longrightarrow  \mathscr{F}_{[\widehat \rho]} \\
 \big(p,  N\big)  & \longmapsto  \widehat N
\end{align*}
which is a biholomorphism onto its image.
\end{prop}

\begin{proof} The fact that $\widetilde{{\mathcal{S}}_3}$ is a local biholomorphism just relies on the fact that one can get a linear parametrisation of $\mathscr{F}_{[\widehat \rho]}$ by adding the parameter $z_0$ to any linear parametrisation of $U$, with $z_0$ a linear parametrisation of $C^*_{l\theta_1'}(\epsilon)$. 

The key fact is that two surfaces resulting from the surgery under consideration are isometric if and only if the two points $q'$ and $q''$ are the same because we have supposed that the elements of $U$ do not have non-trivial isometries (see the proof of Proposition \ref{param}). When $q'$ varies in the universal covering of $
 {C_{\theta_1'}^*(\epsilon)}$, the point $q''$ eventually comes back to its initial position if and only if the two cone angles $\theta_1'$ and $\theta_1''$ are rationally related : the lack of injectivity appears when $q'$ and $q''$ come back for the first time to their initial position as $q'$ turns around $p_1'$ and $q''$ follows, which happens if and only if an equation of the form $k\theta_1' = l\theta_1''$ (for some non-trivial pair of coprime integers $(k,l)$) is satisfied. In this case, $l$ is the exact number of times $q''$ turns around $p_1''$ while $q'$ turns $k$ times around $p_1'$.
\end{proof}\mk

When $\theta_1'$ (resp. $\theta_1''$) is greater than $\pi$, we cannot cut the cone of angle $\theta_1'$ (resp. $\theta_1''$)  in the way it has been done previously. We let the reader verify that one only has to replace the truncated cone of angle $\theta_1'$ at $p_1'$ (resp.\;of angle $\theta_1''$ at $p_1''$) by the metric space obtained by gluing the sides $a$ and $b$ on Figure \ref{surgery5}  below. 
This metric space is the cone of angle $\theta_1'$ (resp. $\theta_1''$) to which one  has added the triangle $T$ appearing in grey on Figure \ref{surgery5} with the aforementioned identifications.
\mk

Regarding the lack of injectivity of the map $\widetilde{{\mathcal{S}}_3}$ (or ${\mathcal{S}}_3$) when some elements of $U$ have non-trivial isometries, one can make a statement similar to Remark \ref{onetoone} to address the question. 
\mk 

\begin{figure}[!h]
\centering
\psfrag{t}[][][1.2]{$T$}
\includegraphics[scale=0.5]{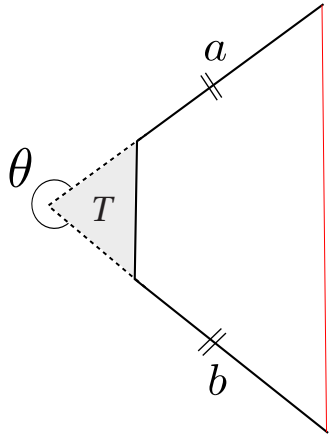}
\caption{The modified cone of angle $\theta > \pi$.}
\label{surgery5}
\end{figure}

\begin{rem} \label{different}
{\rm In contrast with Thurston's surgery, there might be \textbf{essentially different} ways to perform Devil's surgery when $\mathrm{Im}(\rho')$ is finite. What we mean by that is the following (where we continue to use the notations introduced above):   given a base surface
%
$N$, a target moduli space $\mathscr{F}_{[\widehat \rho]}$ in which one lands when performing Devil's surgery on $N$ relatively to  two cone points $p_1'$ and $p_1''$ on it,  different initial choices of pairs $(q',q'')$ might yield surgery maps $\mathcal{S}_3$ whose image are pairwise distinct open subsets of $\mathscr{F}_{[\widehat \rho]}$, a fact which might be a bit surprising at first glance.
\sk

This curiosity comes from the fact that the surfaces obtained via Devil's sur\-gery which belong to the same $\mathscr{F}_{[\widehat \rho]}$ are exactly the ones corresponding to admissible pairs $(q',q'') \in C_{\theta_1'}(\epsilon) \times C_{\theta_1''}(\epsilon)$ 
such that  $\widehat \rho(\widehat \beta) \in \mathrm{Im}(\rho)$.  The solutions to this equation is  a subsurface of $ C_{\theta_1'}(\epsilon) \times C_{\theta_1''}(\epsilon)$ which might have more than one connected component. These connected components \textbf{are} the different ways to perform Devil's surgery and yield different open subsets in $\mathscr{F}_{[\widehat \rho]}$. We carry an   explicit analysis of this phenomenon in the paragraph `Cone points' of  \S\ref{1dimension} which,  even if it concerns only a particular case, 
should allow the reader to understand in full generality  the `curiosity' we are talking about here.}

\end{rem}

\subsection{The kite surgery ${\mathcal{S}}_4$}
\label{SS:KiteSurgery}
Fix $\theta_1 \in  ]2\pi, 4\pi[$ and  $\theta_2, \theta_3 \in ]0, 2\pi[ $   such that 
$$\big(2\pi - \theta_1\big) + \big(2\pi - \theta_2\big) + \big(2\pi - \theta_3\big) = 0. $$
We describe in this section a local surgery building from a regular flat torus a new one with three conical points of respective angles $\theta_1, \theta_2$ and $\theta_3$.  
\sk 

 Let $T$  be a regular flat torus (\textit{i.e.}\;without singular points for the flat metric), and $p_2$ a point on $T$. If $p_3$ is a point close enough to $p_2$, there exists a unique kite contained in $T$ with opposite vertices  $p_2$ and $p_3$ and such that the external angle at these points are  $\theta_2$ and $\theta_3$ respectively. The external angles at the two other vertices of the kite are necessarily equal to the half of ${\theta_1}$, see Figure \ref{gluing}. 
\begin{figure}[!h]
\centering
\psfrag{C}[][][1]{\qquad \qquad Cutting and gluing}
\psfrag{t1}[][][1]{$\;\;\,  \textcolor{blue}{\theta_1}$}
\psfrag{t2}[][][1]{$\; \;  \textcolor{red}{\theta_2}$}
\psfrag{t3}[][][1]{$\textcolor{green}{\theta_3}$}
\psfrag{t12}[][][1]{$ \textcolor{blue}{\frac{\theta_1}{2}}$}
\includegraphics[scale=0.85]{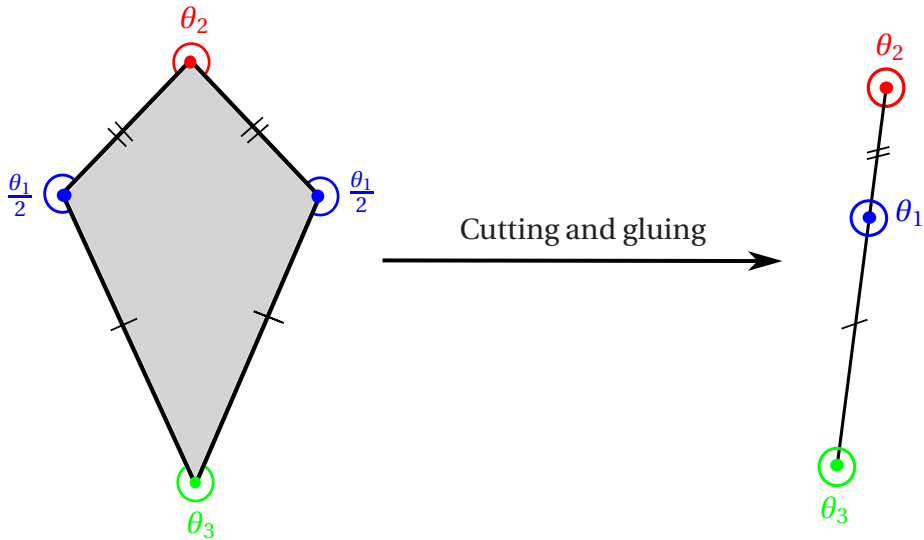}
\caption{Performing the kite surgery}
\label{gluing}
\end{figure}

The {\bf kite surgery} consists in removing the kite (in grey on the above picture) and gluing the adjacent sides in order to get three singular points of respective angles $\theta_1, \theta_2$ and $\theta_3$. Since the flat torus  $T$ is determined by a lattice in $\mathbb{C}$, \textit{i.e.} by two $\mathbb{R}$-linearly independent complex numbers $z_1, z_2$ such that $T = \mathbb{C} / (\mathbb{Z}z_1 + \mathbb{Z}z_2)$, one can perform the kite surgery on $T$ by placing $p_2$ at $0$ and $p_3$ at $z_0$ for any given $z_0$ sufficiently small. 

Up to renormalisation,  we can assume that $z_1 = 1$ and $z_2 = \tau \in \mathbb{H}$. Let $T'$ be the resulting torus (more precisely the class of tori up to renormalisation by an element of $\mathbb{C}$). Let $ \mathscr{F}_{[\rho']}$ be the leaf to which this surgery sends $T'$. Because $z_0$ and $-z_0$ are equivalent under the action of the hyperelliptic involution, the kite surgery performed at these two parameters gives isometric surfaces. 
Similarly to the previous cases, one can build a map
$$ \begin{array}{ccrcc}
{\mathcal{S}}_4\!\! & : \!\! & C^*_{\pi}(\epsilon) \times \mathbb{H} & \longrightarrow & \mathscr{F}_{[\rho']} \vspace{0.15cm}\\
 & & \big(z_0,  \tau\big)  & \longmapsto & T' 
\end{array} $$  
which is a local biholomorphism onto its image, except at points having exceptional symmetries where it is an orbifold covering onto its image. The proof is similar to the proof of Proposition \ref{param}, one just has to take a suitable topological model for $T$ that makes $(z_0, z_1, z_2)$ a linear parametrisation of $\mathcal{F}_{[\rho']}$.

\subsection{The surgery ${\mathcal{S}}_5$ : building flat surfaces with a Euclidean cylinder.}
\label{bswec}
In this section we do not make any assumptions on the cone angles $\theta_1,\ldots,\theta_n$. We explain a simple surgery (to which we shall refer as ${\mathcal{S}}_5$) building flat surfaces of genus $g+1$ and with $n-1$ cone points having an arbitrarily long Euclidean cylinder out of an initial flat surface $N$ of genus $g$ and with $n$ cone points, the new cone point being of angle larger than $2\pi$.\sk

Let $\gamma$ be a geodesic path joining $p_1$ and $p_2$, the two conical points of $N$ of respective angles $\theta_1$ and $\theta_2$. Cut along $\gamma$ to get a flat surface with one boundary component and then glue $p_1$ and $p_2$ together. The resulting surface has a boundary consisting of two simple closed geodesics touching at one point where they are singular. Then we glue a flat cylinder along these two boundary components to get a new flat surface $N'$ of genus $g+1$ with an embedded cylinder (see Figure \ref{F:S5}). Note that the cone angle of $N'$ at its new singular point (namely the one obtained after having identified $p_1$ and $p_2$) is easily seen to be $\theta_1+\theta_2+2\pi$.\sk

There are two real parameters for the aforementioned gluing of the flat cylinder: its length and a twisting parameter(starting from one given way to glue a given cylinder, one can compose the gluing function at one of its extremities by a rotation, the angle of this rotation being the twisting parameter). Both can be encoded by a single complex number $z_0$ whose imaginary part is positive,  such that the cylinder we glue identifies with the one of base $1$ and height $z_0$. This makes sense because there always exists a normalization of $N$ such that the geodesic joining $p_1$ and $p_2$ in the initial surface develops onto the segment $[0,1]$. \sk 

\begin{center}
\begin{figure}[!h]
\psfrag{p1}[][][0.73]{$\quad \;\;\;{p_1} $}
\psfrag{pp1}[][][0.73]{$\;\;\;\, p_1 $}
\psfrag{p2}[][][0.73]{$p_2 $}
\psfrag{pp2}[][][0.73]{$p_2 \, $}
\includegraphics[scale=0.35]{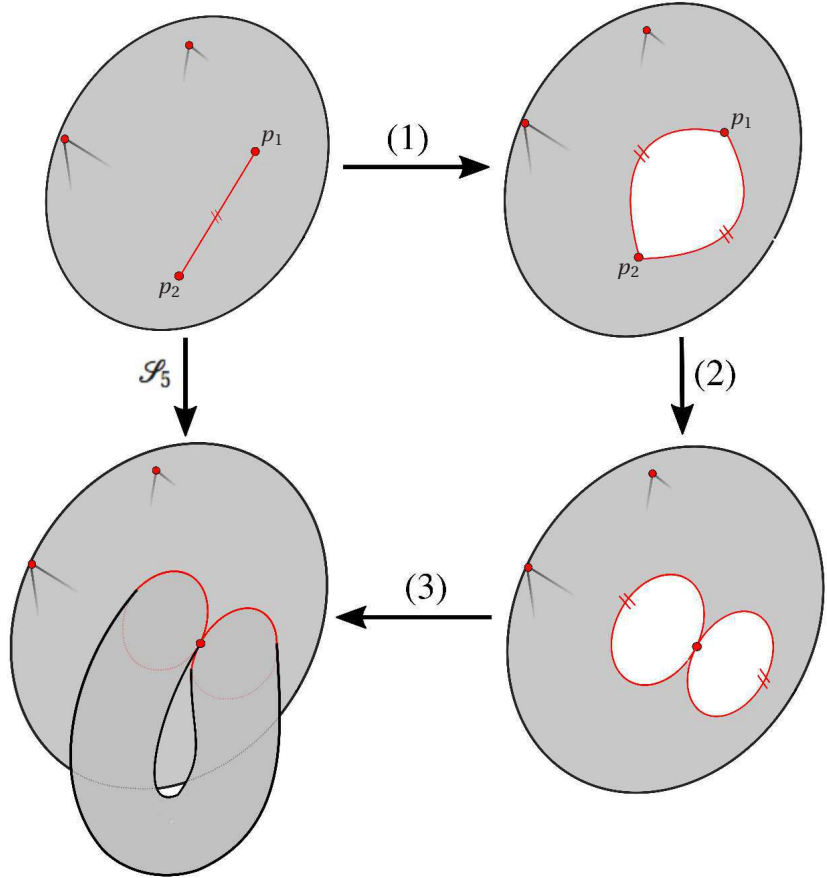}
\caption{The surgery $\mathcal S_5$ performed on a flat sphere with four cone points: (1) we cut along the geodesic segment between two 
of them;  (2) we identify the two corresponding points on the boundary; (3) then we glue a flat cylinder in order to  obtain a flat tori with three cone points.}
\label{F:S5}
\end{figure}
\end{center}

If $U\subset $ is a neighbourhood of the initial surface in $\F$, we can build a natural map
$$ \begin{array}{ccccc}
{\mathcal{S}}_5 & : & \mathbb{A} \times U & \longrightarrow & \mathscr{F}_{[\rho']} \\
 & & (z_0, N)  & \longmapsto & N'\; , 
\end{array} $$ 
 where $\mathbb{A}$ stands for the infinite cylinder $\mathbb{H}/ (z \sim z+1)$.  
 
 The surgery  ${\mathcal{S}}_5$ associates to $N$ and $z_0$ the surface obtained after gluing  the flat cylinder of height the parameter $z_0$ to (the good normalisation of) $N$.

\begin{prop}
The surgery ${\mathcal{S}}_5$ is a local biholomorphism onto its image.
\end{prop}
\begin{proof} Let $(z_1, \ldots, z_m)$ be a linear parametrisation of $U$ such that $z_1$ parametrises the geodesic segment along which the surgery is performed. Then one verifies that $(z_0z_1, z_1, \ldots, z_m)$ is a linear parametrisation of the image of ${\mathcal{S}}_5$. Indeed, one needs to rescale the cylinder of base and height  $(1,z_0)$ to $(z_1, z_1z_0)$ so it fits the segment parametrised by $z_1$ onto which it is glued.  The rest of the proof works in the same way as in Proposition \ref{param}.\end{proof}

\subsection{Calculation of the signature of the area form in particular cases.}

A corollary of  the description of these surgeries is an easy inductive computation of the signature of the Veech form in the specific cases we are interested in. In order to perform this computation, we will need the following definition:
\begin{defi}
\label{D:Reduction}
A unitary character 
$\tilde{\rho} \in \mathrm{H}^1(M, \mathbb{U})$ is a {\bf reduction} of $\rho \in \mathrm{H}^1(N, \mathbb{U})$ if there exists an injective diffeomorphism $i :M \longrightarrow N$ such that $\tilde{\rho} = i^*\rho $.

\end{defi} In particular, the linear holonomy of the base surface of one the surgeries we have described is a reduction of the linear holonomy of the surface obtained by one of these surgeries. We have the following:

\begin{prop}
\label{signature}
\begin{enumerate}
\item  Suppose that $g=0$, $n \geq 3$ and that 
$ 0 < \theta_i < 2\pi $
for all $i=1,\ldots,n$. Then Veech's area form has signature $(1,n-3)$. 
\sk 
\item Suppose that $g=1$, $n\geq 2$,  $ 2\pi < \theta_1 < 4\pi$ and that 
$ 0 < \theta_i < 2\pi $ 
for all $ i$ such that $1 < i \leq n$. Then Veech's area form has signature $(1, n-1)$.  
\end{enumerate}
\end{prop}
\begin{proof} The proof goes by induction in both cases. We explain only $(2)$ since the proof of $(1)$ is basically the same but simpler and  roughly sketched in \cite{Thurston}. 
\sk 

We suppose that $g=1$, $n\geq 3$ and $\rho \in \mathrm{H}^1(N, \mathbb{U},\theta)$. Let $\rho'$ be the reduction (see Section \ref{completion}) associated to a collision between two points of angles $\theta_i$ and $\theta_j$. Such a collision can actually happen if and only if $(2\pi - \theta_i) + (2\pi - \theta_j)  < 2\pi$, and two such points always exist since there is initially only one conical point of negative curvature which is also smaller than $4\pi$, provided that $n \geq 3$. Therefore the leaf associated to $\rho'$ is not empty and one can perform Thurston's surgery on elements of $\mathscr{F}_{[\rho']}$ to get elements of the $n$-dimensional leaf $\F$. Let $(z_0, z_1, \ldots, z_{n-1})$ be a linear parametrisation of $\mathscr{F}_{[\rho]}$ such that $(z_1, \ldots, z_{n-1})$ is a linear parametrisation of $\mathscr{F}_{[\rho']}$ and $z_0$ is such that $(z_0, z_1, \ldots, z_{n-1})$ represents the element that one gets after performing Thurston's surgery with parameter $z_0$ on the surface represented by $(z_1, \ldots, z_{n-1})$. If $A$ is the area form for the parametrisation $(z_0, z_1, \ldots, z_{n-1})$ and $A'$ the one  for $(z_1, \ldots, z_{n-1})$, then  one has
$$ A(z_0, \ldots, z_{n-1}) = A'(z_1, \ldots, z_{n-1}) - \mu |z_0|^2 $$
for a certain  positive constant $\mu$  which depends only on the value of the angle of the cone point on which the surgery is performed. Indeed, $\mu |z_0|^2$ is the area of the portion of cone removed when proceeding to the surgery which is an isosceles triangle with fixed centre angle and whose base has length linearly depending on $|z_0|$. The induction hypothesis ensures that  $A'$ has signature $(1,n-2)$ and therefore $A$ has signature $(1, n-1)$. 
\mk 

The case $n = 2$ remains to be handled.
With a similar argument, but using Devil's surgery instead of  Thurston's one, we find that in that case the signature is $(1,1)$. One starts with a flat sphere with three cone points. The set of such flat spheres can be parametrised by a complex number $z_1$ such that the area of the associated sphere is $|z_1|^2$ (and recovering that up to projectivising, there is only one such flat sphere). The Devil's surgery consists in removing two portions of cones in a sphere. If $z_0$ is the parameter of the surgery, the area of the resulting torus will be of the form 
 $$ |z_1|^2 - \mu|z_0|^2 $$ 
 for a certain $\mu>0$. This completes the proof of the proposition. 
 \end{proof}

\subsection{Cone angle around a codimension $1$ stratum.}
\label{coneangle}
As we will see in detail later, the surgery maps ${\mathcal{S}}_i$ (for $i=1, \ldots, 4$) describe the cone-manifold structure of the metric completion of $\F$ close to a codimension $1$ stratum, when  Veech's area form endows $\F$ with a complex hyperbolic structure. In particular, they allow the computation of the associated cone-manifold angles.

\begin{enumerate}

\item In case of both Thurston's surgeries 
$ \mathcal S_1$ of $\mathcal S_2$, the cone-manifold angle around the codimension $1$ stratum is the angle of the Euclidean cone angle on which the surgery is performed.\sk

\item In the case of Devil's surgery, when both angles are rational multiples of $2\pi$, say ${2\pi m'}/{M}$ and ${2\pi m''}/{M}$, the cone-manifold angle is $\frac{2\pi \mathrm{lcm}(m',m'')}{M}$.
\sk 

\vspace{-0.35cm}
\item In the case of the Kite surgery $\mathcal S_4$, the cone angle always equals $\pi$ since the parameter space for the surgery is the neighbourhood of a regular point of angle $2\pi$ on which the hyperelliptic involution acts.
\end{enumerate}

\subsection{Geometric convergence.}
We end this section dedicated to surgeries by a paragraph on a notion of geometric convergence for flat surfaces. Whether two Riemannian manifolds (in a moduli space) are close or not depends on an \textit{a priori} definition.

\begin{defi}
Let $N'$ be a flat surface of area $1$ obtained from a surgery ${\mathcal{S}}$ on a surface $N$. The {\bf width} of this surgery at $N$ is :  
\sk 

$\bullet$  the distance between the two new cone points
 in $N'$  if ${\mathcal{S}}
\in \{{\mathcal{S}}_1\, , \, {\mathcal{S}}_2
\}$;\sk

$\bullet$  the length of the short essential curve created on $N'$ if  
${\mathcal{S}}={\mathcal{S}}_3$; \sk

$\bullet$ 
the distance between $p_2$ and $p_3$ in $N'$ (see \S\ref{SS:KiteSurgery} for the notations) if 
 ${\mathcal{S}}={\mathcal{S}}_4$.
\end{defi}

The width of a surgery is a positive parameter 
whose square depends linearly on the Euclidean area of the removed  part of the initial surface on which one performs the surgery. When $g$ and $\theta$ are such that $\F$ has a complex hyperbolic structure, this width has a geometric interpretation in terms of the distance to the strata of the metric completion which will be made explicit in the next section (see Lemma \ref{proj}). 

\begin{defi}[{\bf Geometric convergence}]
\label{convergence}
A sequence of flat surfaces $(M_\ell)_{\ell \in \mathbb N}\in (\F)^{\mathbb N}$ is said to be {\bf geometrically converging} either if it converges in $\F$ or if there exist 
\begin{itemize}
\item  a flat surface $M_{\infty}$ belonging to a leaf\; $\mathscr{F}_{[\rho']}$  for a reduction $\rho'$ of $\rho$; \sk

\item  a small neighbourhood $U$ of $M_{\infty}$ in $\mathscr{F}_{[\rho']}$ on which a surgery map  ${\mathcal{S}}={\mathcal{S}}_i$ (for some $i = 1,2,3, 4$)  is well defined; \sk

\item $(\epsilon_\ell)_{\ell \in \mathbb N}$ a sequence of positive numbers going to $0$; \sk

\item  a truncated sequence $(M'_\ell)_{\ell>>1}$ of flat surfaces elements of  $U$ which converge to $M_{\infty}$ in $U\subset \mathscr{F}_{[\rho']}$, 
\end{itemize}

\noindent such that $M_n$ is obtained after a surgery ${\mathcal{S}}$ of width $\epsilon_\ell$ on $M'_\ell $ for $\ell >>1$. \sk 

 We say that the sequence $(M_\ell )_{\ell \in \mathbb N}$ {\bf geometrically converges} to the pair $(M_{\infty}, {\mathcal{S}})$, or just to $M_{\infty}$ if $(M_\ell )_{\ell \in \mathbb N}$ converges in $\F$.

\end{defi}

The true interest of this definition  is that it will allow us to make the difference between two sequences of surfaces in $\F$ whose limits are isometric metric spaces but lying at different places in (the metric completion of) $\F$. 



\section{\bf THE METRIC COMPLETION}
 \label{completion}

In this section, $N$ stands for  a surface of genus $g = 0$ with $n+3$ marked points or of genus $1$ with $n+1$ marked points (we write the number of marked points this way in order that $\F$ be of complex dimension $n$ in these two cases).
\sk 

  We also make the assumption that 
\begin{itemize}

\item if $g=0$  all the angles $\theta_1, \ldots, \theta_{n+3}$ belong to  $]0,2\pi[$; \sk

\item if $g=1$ then $\theta_1 \in ]2\pi,4\pi[$ and all the other cone angles $\theta_i$ are in $]0,2\pi[$.

\end{itemize}

\noindent  Therefore, for any $\rho \in  \mathrm{H}^1(N, \mathbb{U}, \theta) $ in the image of the linear holonomy map, the leaf $\mathscr{F}_{[\rho]}$ of Veech's foliation in the corresponding moduli space of marked curves is endowed with a complex hyperbolic structure of dimension $n$ (see \cite{Thurston} for the case $g=0$ and Proposition \ref{signature} or  \cite{Veech} for the case $g=1$).  \sk 
 
  In the present section, we are interested in the structure of the metric completion of $\F$ endowed with this complex hyperbolic structure. Therefore every mention of a geometric property of $\mathscr{F}_{[\rho]}$ will now be relative to this structure.  
Our goal here is to prove the following theorem: 
%

\begin{thm}
\label{mcompletion}

 Let $X$ be the metric completion of \,$\F$. If \,$\mathrm{Im}(\rho)$ is finite then:
\sk 

{\rm (1)} $X$ has a stratified structure $ X = X_0 \sqcup X_1 \sqcup  \ldots \sqcup X_n $ with $X_0 = \F$;
\sk 

{\rm (3)} the topological closure of $X_i$ in $X$ is $X_{i} \sqcup X_{i+1} \sqcup 
 \ldots \sqcup X_n$;
\sk 

{\rm (3)}  for $i=0,\ldots,n$, $X_i$ is a smooth complex hyperbolic manifold of complex 

\hspace{0.6cm}dimension $n-i$ which carries a natural $\mathbb C\mathbb H^{n-i}$-structure; 
\sk

{\rm (4)} each $X_i$ is a finite union of finite covers of $\mathscr{F}_{[\tilde{\rho}]}$ for some reductions $\tilde{\rho}$ of $\rho$. 
\end{thm}


 From now on, we  assume that $\mathrm{Im}(\rho)$ is finite.

\subsection{Strata}
In Section \ref{surgeries}, we have introduced various surgeries, describing different ways flat surfaces can degenerate and how to parametrise these surgeries. 
The degenerate flat surfaces we see appearing in these ways belong to the metric completion of $X_0 = \F$. More precisely, they appear in copies of (finite coverings of) $\mathscr{F}_{\tilde{\rho}}$ of complex dimension $n-1$ where $\tilde{\rho}$ is a reduction of $\rho$ in the sense of Definition \ref{D:Reduction}.

We define inductively the strata $X_1, \ldots, X_n$ appearing in the description of $X$. The first stratum $X_1$ is the set of  pairs surface/surgery $(N, {\mathcal{S}})$ such that there exists a sequence of elements of $X_0 = \F$  
geometrically converging (in the sense of Definition \ref{convergence}) to a pair $(N, {\mathcal{S}})$ not already in  $X_0$.

\begin{lemma}
\label{finitecover} The stratum 
$X_1$ is a union of (finite covers of) $\mathscr{F}_{[\tilde{\rho}]}$s for some reductions $\tilde{\rho}$ of $\rho$ and such that the complex dimension of $\mathscr{F}_{[\tilde{\rho}]}$ is $n-1$.
\end{lemma}

\begin{proof} 
The holonomy $\tilde{\rho}$ of the limit of a sequence geometrically converging is a reduction of $\rho$. Since $\mathrm{Im}(\rho)$ is finite, there are only finitely many such reductions and therefore only finitely $\mathscr{F}_{[\tilde{\rho}]}$ to which such a limit can belong.\sk

 Now let $\mathscr{F}_{{}[\tilde{\rho}]}$ be such that one of its elements $(N_0, {\mathcal{S}}_0)$ appears as the limit of a  geometrically converging sequence. We claim that $Z$,  the connected  component of $X_1$ to which $(N_0, {\mathcal{S}}_0)$ belongs,  is a finite cover of a connected component of  $\mathscr{F}_{[\tilde{\rho}]}$. One can define a local homeomorphism from a neighbourhood of $(N_0, {\mathcal{S}}_0)$ in $Z$ to the component of $\mathscr{F}_{{}[\tilde{\rho}]}$ containing $N_0$ associating to any pair $(N, {\mathcal{S}}_0)\in Z$ sufficiently close to $(N_0, {\mathcal{S}}_0)$
 the associated flat surface $N$. This map is a covering map which is finite since the fiber over a point is included in the set of different ways to perform the corresponding surgery (see Remark \ref{different}), which is finite. This fiber is actually trivial for surgeries different from Devil's surgery. \sk

More precisely let $\tilde{\rho}$,  be such that one element $N_0$ of $\mathscr{F}_{{}[\tilde{\rho}]}$  appears as the limit of a geometrically converging sequence  along a surgery ${\mathcal{S}}$. The set of elements of $X_1$ whose associated flat surface belongs to $\mathscr{F}_{{}[\tilde{\rho}]}$ is exactly the covering map over  $\mathscr{F}_{{}[\tilde{\rho}]}$ whose fiber over a point is the number of ways to perform the surgery ${\mathcal{S}}$ at this point.
\end{proof}
\mk

The next step is to prove that the disjoint union $X_0 \sqcup X_1$ embeds into the metric completion of $X_0 = \F$. This is a direct consequence of the following  proposition: 
\begin{prop}
\label{P:GeometricConvergence}
If $(M_\ell)_{\ell\in \mathbb N}\in (\F)^{\mathbb N}$  converges geometrically, then it is 
 a  Cauchy sequence
  for  the metric induced by the complex hyperbolic structure on $\F$. 
\end{prop}
\begin{proof}
The statement is clear if $(M_\ell)_{\ell \in \mathbb N}$ converges in $\F$ so we assume that it is not the case. 
We consider a linear parametrisation $(z_0, z_1, \ldots, z_n)$ such that  
\begin{itemize}
\item $z_0$ is the surgery parameter (\textit{i.e.} the small segment linking the two new cone points after Thurston's surgery, or the small closed broken geodesic segment appearing after a Devil's surgery which is such that the width of the surgery is $|z_0|$ in a normalisation of area $1$, etc.); and 
\sk 
\item $(z_1, \ldots, z_n)$ is a linear parametrisation of the leaf to which the surface on which the surgery is done belongs.
\end{itemize}

For any $\ell \in \mathbb N$, let $z(\ell)=(z_i(\ell))_{i=0}^n$ stand for the coordinates of $M_\ell$ in the considered linear parametrisation, normalised so that the 
 corresponding area of $M_\ell $  is 1. From the very definition of 
metric convergence, we have  that $(z_0(\ell ))_{\ell \in \mathbb N}$ goes to 0 
 and  $((z_i(\ell ) )_{i=1}^n)_{\ell \in \mathbb N}$ converges in $\mathbb C^{n}$ since the associated sequence of flat surfaces in the corresponding stratum converges.    It follows that $z(\ell )$ converges in $\mathbb C^{n+1}$ as $\ell $ tends to infinity. 
 
On the other hand,  the normalisation of the areas of the $M_\ell$'s ensures
 that the $z(\ell )$'s stay away from the boundary of the model of the complex hyperbolic space associated with the considered linear parametrisation 
 $(z_0,  \ldots, z_n)$.  The proposition follows.  
  \end{proof}\bk 

Thanks to the preceding result, one has a map: 
 $$i_1 : X_0 \sqcup X_1 \longrightarrow \overline{\F}.$$
\begin{prop}
The map $i_1$  defined just above 
is injective.
\end{prop}

\begin{proof} 
Clearly, the restriction of $i_1$ to $X_0=\F $ is the identity. Since  
 $i_1(X_1)\cap \F=\emptyset$, it suffices to show that  $i_1\lvert_{X_1}$ is injective to get the proposition.\sk 
 
Let  $(A_1, {\mathcal{S}}_1)$ and $(A_2, {\mathcal{S}}_2)$ be two distinct points in $X_1$.  The surgery maps
$$\begin{array}{ccccc}
{\mathcal{S}}_i & : & C_{\theta_1}^*(\epsilon_i)  \times U_i & \longrightarrow & \mathscr{F}_{[\rho]} \\
& & (z_0,  N)  & \longmapsto & N'\, , 
  \end{array} $$ 
 defined in Section \ref{surgeries}, where $U_i$ is a neighbourhood of $A_i$
for  $i =1,2$, extend continously to $\overline{{\mathcal{S}}_i}  :   C_{\theta_1}(\epsilon_i)  \times U_i  \longrightarrow  \overline{\mathscr{F}_{[\rho]}}$ which are homeomorphisms onto their images and whose respective images are neighbourhoods of $i_1(A_1)$ and $i_1(A_2)$ in $X$. If $U_i$ and $\epsilon_i$ are chosen small enough, the images of ${\mathcal{S}}_1$ and ${\mathcal{S}}_2$ do not overlap which implies that $i(A_1)$ and $i(A_2)$ are separated and therefore different. \end{proof}\mk

The distance induced on $X_1$ by this embedding is nothing else but the one induced by its natural complex hyperbolic structure: a neighbourhood of $X_1$ in $X=\overline{\F} $ can be described by a finite number of linear parametrisations of the form $(z_0, \ldots, z_n)$ in which $X_1$ corresponds to the locus $\{ z_0 = 0 \}$.
\begin{rem}
\label{R:C'est Subtil}
\rm We would like to stress that in the preceding assertion, 
one has to be aware that $(z_0, \ldots, z_n)$ does not induce a local system of coordinates on a neighbourhood  in $X$ 
 of a small open subset of $X_1$. Actually,  what must  be understood is that 
the equation $z_0=0$ cuts out something (a piece of  $X_1$ as it happens) in the boundary of the definition domain  of the chart induced by the linear parametrisation $(z_0, \ldots, z_n)$.   We will not dwell again on this subtlety in what follows but will only make reference to the present remark.
\end{rem}
 \mk 

The definition of $X_2$ is slightly more subtle because it is possible that two essentially different geometrically converging sequences in $X_1$ converge to the same point in $\overline{\F}$: consider for instance a case when $g = 1$ and $n=3$. We can distinguish two types of components in $X_1$ : the one which are moduli spaces of tori with two cone points, and those which are Thurston-Deligne-Mostow's moduli spaces of flat spheres with four cone points. Both can degenerate on flat spheres with three cone points. It can happen that these a priori different limits are identified in $\overline{\F}$. One must think of such points as parts of the intersection locus of the closures of two connected components of $X_1$ on which two different surgeries can be performed, each leading to a different component of $X_1$.  
\sk 

 In order to define correctly $X_2$, we proceed in two steps : we first define in an analogous way $Y_2$, that one shall think to be roughly the set of pairs flat surface/surgery $(M, {\mathcal{S}})$ such that there exists a sequence in $X_1$ geometrically converging to $(M, {\mathcal{S}})$. However, the fact that we are dealing with finite covers of leaves prevents us from giving such a straightforward definition. Bypassing this difficulty is rather easy: $X_1$ is a finite union of finite covers  of some leaves $\mathscr{F}_{[\tilde{\rho}]}$, for some reductions  $\tilde{\rho}$  of $\rho$. Such a finite cover $Z_0$ can be partially metrically completed by adjoining a codimension $1$ stratum $Z_1$ (possibly with several connected components) in order that the covering map $ \pi : Z_0 \longrightarrow \mathscr{F}_{[\tilde{\rho}]}$ extends to a map 
 $$\widetilde{\pi} : Z_0 \sqcup Z_1\longrightarrow \mathscr{F}_{[\tilde{\rho}]} \sqcup X'_1$$
  which is a covering map, possibly ramified along $X'_1$, where $X'_1$ is analogous to $X_1$ associated to $\mathscr{F}_{[\tilde{\rho}]}$ and such that $\widetilde{\pi}^{-1}(X'_1) = Z_1$: $X'_1$ is a finite union of some unramified finite covers of some leaves $\mathscr{F}_{[\tilde{\rho}']}$ for some reduction $\tilde{\rho}'$ of $\tilde{\rho}$. Then one defines  $Y_2$ as the (finite) union of all such $Z_1$'s associated to all the finite covers appearing in $X_1$ and is itself a finite union of finite covers of some leaves  $\mathscr{F}_{[\widehat{\rho}]}$  for some reduction $ \widehat{\rho}$ of $\rho$ (a reduction of a reduction of $\rho$ is again a reduction of $\rho$ as it follows immediately from Definition \ref{D:Reduction}).\sk
 
 Defined this way, $X_1 \sqcup Y_2$ maps into the metric completion of $X_1$ as a complex hyperbolic manifold with each of its connected components endowed with the induced complex hyperbolic distance. But note that this construction does not take into account how close these connected components can be  in $\overline{\F}$. 
 
 As we did for $X_1$, we define a map : 
 \begin{equation}
\label{E:i2}
i_2 : X_1 \sqcup Y_2 \longrightarrow \overline{\F} 
\end{equation}
whose restriction to $X_1$ coincides with the one of $i_1$. 
The main difference between the maps  $i_1$ and  $i_2$ is that the latter is not injective : some components of $Y_2$ are identified. We define $X_2$ as the image of $Y_2$ under that map or equivalently, $Y_2$ with the aforementioned components identified. The crucial point is that $X_0$ has a distance only defined on  each of its connected components by the complex hyperbolic metric, while the distance on $X_1$ takes into account the way in which $X_1$ is embedded in $X$.\sk

The following property allows to identify the irreducible components of $X_2$ with some unramified coverings of some reductions of $\F$.
\begin{prop}
\label{identification}
If $A$ and $B$ are distinct points of $Y_2$ such that $i_2(A) = i_2(B)$ then
\begin{enumerate}
\item the flat surfaces associated to $A$ and $B$ are isometric; \sk
\item there exists $\widehat{\rho}$,  a reduction of $\rho$,  and a connected component $Z$ of $\mathscr{F}_{[\widehat{\rho}]}$ such that the images by $i_2$ of the components of $Y_2$ containing $A$ and $B$ are both equal to the same finite cover of $Z$.
\end{enumerate}
\end{prop}
\begin{proof} 
Let $N_A$ and $N_B$ the two flat surfaces associated to $A$ and $B$ respectively. \sk 

Assume that  $N_A$ and $N_B$ are not isometric. Consider two sequences in $X_1$ geometrically converging to $A$ and $B$ respectively. They can be approximated by two Cauchy sequences in $\F$ converging to the same point in $\overline{\F}$.  But their associated flat surfaces converge towards two different metric spaces, which is impossible.   This proves (1).
\sk

Let $Z$ be the component of  $\mathscr{F}_{[\tilde{\rho}]}$ to which $N= N_A = N_B$ belongs. We first remark that there are neighbourhoods of $A$ and $B$ in $Y_2$ which are identified under $i_2$: there are two surgeries ${\mathcal{S}}_A$ and ${\mathcal{S}}_B$ on $N$ which produce images under $i_2$ of neighbourhoods of $A$ and $B$ in $Y_2$. This identification can be extended to a cover of $Z$. This cover must be finite since it is covered by a component of $Y_2$ which is finite according to Lemma \ref{finitecover}.
\mk
\end{proof}

From the map \eqref{E:i2} and by the very definition of $X_2$, one deduces an injective map $j_2: X_1\sqcup X_2\rightarrow \overline{\F}$. 
Since the restrictions of $i_1$ and $j_2$ to $X_1$ coincide, one can consider their fiber product over $X_1$ in order to get an injective  map: 
$$
X_0 \sqcup X_1 \sqcup X_2 \longrightarrow \overline{\F}.  
$$

Inductively,  one defines $Y_{j+1}$ from $X_j$ in exactly the same way  we defined $Y_2$ from $X_1$. Then one defines $X_{j+1}$ by identifying 
some components  of $Y_{j+1}$ using the natural map $ i_{j+1} : X_j\sqcup Y_{j+1} \longrightarrow \overline{\F}$. 
 Note at this point that the analog of Proposition \ref{identification}  for $Y_{j+1}$ holds true, the proof being completely similar. Since a reduction $\widehat{\rho}$ of a reduction $\tilde{\rho}$ of $\rho$ is still a reduction of $\rho$, we get that $X_i$ is a complex hyperbolic manifold of dimension $n-i$ whose connected components are some coverings of some leaves  $\mathscr{F}_{[\widetilde{\rho}]}$ for some reduction $\widetilde{\rho}$ of $\rho$.
\sk 

Putting all pieces together we get that
\begin{itemize}
\item $X_0 \sqcup X_1 \sqcup \cdots \sqcup X_n$ embeds into  the metric completion $X$ of $X_0=\F$;
\sk 
\item $\forall i$, $X_i$ is a finite union of finite covers of $\mathscr{F}_{\widetilde{\rho}}$ for some reductions $\widetilde{\rho}$ of $\rho$;
\sk 
\item $\forall i$, the distance of $X$ induces on $X_i$ its natural structure of complex hyperbolic manifold of dimension $\mathrm{dim}(X_0) - i$.
\end{itemize}

\subsection{Proof of the surjectivity}
\begin{prop}
\label{complet}
Assume that\;$\mathrm{Im}(\rho)$ is finite.  Then the embedding 
$$ 
X_0 \sqcup X_1 \sqcup \ldots \sqcup X_n \longrightarrow X=\overline{\F} 
$$
is onto.
\end{prop}

This proposition says in substance that the metric space obtained by adding to $\F$ the degenerate surfaces that ones sees when reversing the surgeries studied in Section \ref{surgeries} is complete. Before giving the proof, we have to state two technical lemmas relating the complex hyperbolic geometry of $\F$ to  the geometry of the underlying flat surfaces  parametrized by this leaf. 

\begin{lemma} 
\label{bounded} 
If\;\,${\rm Im}(\rho)$ is finite then the two following assertions hold true for any  sequence $(M_\ell)_{\ell\in \mathbb N}$ of flat surfaces in $\F$ normalised so that their area is 1:\sk 
\begin{enumerate}
\item if $(M_\ell)_{\ell\in \mathbb N}$ is a Cauchy sequence 
  then $\big({D}(M_\ell)\big)_{\ell\in \mathbb N}$ is bounded;
\mk 
\item  if  $\big({D}(M_\ell)\big)_{\ell\in \mathbb N}$ is bounded then  $(M_\ell)_{\ell  \in \mathbb N}$  is  a Cauchy sequence (up to passing to a subsequence).
\end{enumerate}
\end{lemma}

\begin{proof} We postpone the proof of $(1)$  to Section \ref{cusps} in which we provide a description of the parts of $\F$ on which the diameter function $D$ is large.
\sk 

 The proof of $(2)$ consists in remarking that using the Delaunay 
 decomposition
 of $M_\ell$, we can assume that,  up to passing to a subsequence:  
\begin{itemize}

\item all the $M_\ell$'s are recovered by gluing the sides of a pseudo-polygon through the same gluing pattern;
\sk 
\item the side glued together always form the same angle (that this can be assumed follows from the fact that $\mathrm{Im}(\rho)$ is finite by assumption);
\sk 
\item 
the lengths of each side converge (since the lengths of the edges of the Delaunay triangulation are smaller than $2D(M_\ell )$ by 
Proposition  \ref{triangulation}).\sk 
\end{itemize}

In the chart defined by the gluing pattern, the coordinates of the $M_\ell$'s form a Cauchy sequence. 
Then using the fact that their areas  all have been assumed to be 1, one can argue in the same way as at the end of the proof  of Proposition \ref{P:GeometricConvergence} and get that  $(M_\ell )_{\ell \in \mathbb N}$ is a Cauchy sequence for the metric  on $\F$ induced by the complex hyperbolic structure it carries.
\end{proof}
\sk

\begin{lemma}
\label{proj}
There exists a positive constant $K=K_{[\rho]}$ such that if $M \in X_0= \F$ (which is supposed to be normalised such that its area equals $1$) is obtained from a surgery ${\mathcal{S}}$ of width $\epsilon$ from an element of $X_1$ then
$$   {d}(M,X_1) \leq K\, \epsilon\, ,  $$ 
 where $d$ denotes the extension of the complex hyperbolic distance on $X_0$ to $X$.
 \end{lemma}
 \begin{proof} Let $(z_0, z_1, \ldots, z_n)$ be a linear parametrisation compatible with the surgery ${\mathcal{S}}$ (see Section \ref{surgeries}) which is such that 

\begin{itemize}
\item the parameter $z_0$ is the surgery parameter, in particular $|z_0| = \epsilon$ is the width of the surgery; 
\sk 
\item $(z_1, \ldots, z_n)$ is a linear parametrisation of $U \subset X_1$; 
\sk 
\item in the coordinates $z_0,\ldots,z_n$, the area form $A$ writes down 
$$A(z_0, z_1, \ldots, z_n) = A'(z_1, \ldots, z_n) - \mu |z_0|^2 $$ 
where 
\begin{itemize}
\item $\mu=\mu_{{\mathcal{S}}}$ is a positive real constant depending on the surgery ${\mathcal{S}}$ (it is the constant such that $\mu |z_0|^2$ is the area of the part of the surface removed while processing the surgery); 
\sk 
\item $A'$ is the area form on $U$ expressed in the coordinates $z_1, \ldots, z_n$.\sk 
\end{itemize} (Note that since  the image of $\rho$ is assumed  to be finite,  the set of such $\mu_{{\mathcal{S}}}$'s is finite and thus $\mu_{{\mathcal{S}}}$ is 
uniformly bounded from above).\sk 
\end{itemize}

One can compute the complex hyperbolic distance between two points in the complex hyperbolic space using formulas involving $A$ (see \cite[\!p.77]{Goldman} for instance). If $a : \mathbb{C}^{n+1} \times \mathbb{C}^{n+1} \longrightarrow \mathbb{C}$ stands for  the polarisation of $A$, namely the Hermitian form such that $A(X) = a(X,X)$ for every $X\in \mathbb{C}^{n+1}$,  the complex hyperbolic distance $d(X,Y)$ between two points $[X],[Y] \in {\mathbb C}{\mathbb{H} }^n \subset \mathbb{CP}^n $ satisfies 
$$ \cosh^2 \left(\frac{d(X,Y)}{2}\right)  = \frac{a(X,Y)a(Y,X)}{a(X,X)a(Y,Y)}.$$

 If $X =(x_0, X') \in \mathbb{C}^{n+1}$ and $Y =(y_0,Y')\in \mathbb{C}^{n+1}$ with 
$X'=(x_1,  \ldots, x_n)$ and  $Y' =(y_1, \ldots, y_n)$,  the formula for $a(X,Y)$ is 
$$ a(X,Y) = a'\big(X', Y'\big) - \mu \, x_0\overline{y_0} $$
where $a'$ stands for  the polarisation of $A'$.
\sk 

We want to estimate $\alpha = d\big((z_0, z_1, \ldots, z_n), (0, z_1, \ldots, z_n)\big)$. Since $M$ is supposed to have area $1$, it folows from the discussion above that  we have 
$$\cosh^2 \big({\alpha}/{2}\big) = \frac{\big(1 + \mu\, \epsilon^2\big)^2}{1 \cdot (1 + \mu\, \epsilon^2)} = 1+ \mu\, \epsilon^2\, .  $$ 

 Since for all $a > 0$, one has $1 +{a^2}/{2} \leq \cosh(a)$, it comes 
$$ 1 + \frac{\alpha^2}{8} \, \leq \, \sqrt{1 +\mu\, \epsilon^2}\,  \leq \, 1 + \frac{\mu}{2}\epsilon^2 $$ 
 from which we deduce that $ \alpha \leq 2 \sqrt{\mu} \epsilon$. 
 Since $\mu$ is bounded from above by a constant only depending on the image of $\rho$, the proposition is proved.
\end{proof}
\mk

We end this section with the proof of Proposition \ref{complet}. We still suppose that $\mathrm{Im}(\rho)$ is finite, which is the crucial hypothesis on which everything done in this paper relies on. We just say a word on the general strategy. In order to show that any Cauchy sequence accumulates to one  point in a stratum, we first prove that, since the diameter along a Cauchy sequence is bounded, if such a Cauchy sequence does not converge in $\F$, it implies that it degenerates in the sense that either  its systole or  its relative systole goes to zero. If the latter occurs, we show that such a surface having a sufficiently short systole or relative systole is obtained from one of the four surgeries described in Section \ref{surgeries} and therefore is  very close to $X_1$. Then we conclude with an inductive argument.
\mk

\noindent{\bf Proof  of Proposition \ref{complet}.}
Let ${M}_\bullet= (M_\ell)_{\ell\in \mathbb{N}}$ be a Cauchy sequence in $X_0=\F$ for the complex hyperbolic metric. For the remainder of the proof, we set  
$$D_\ell = D\big(M_\ell\big)\, , \qquad \sigma_\ell = \sigma\big(M_\ell\big)\qquad \mbox{ and }\qquad \delta_\ell= \delta\big(M_\ell\big)$$
 for any $\ell \in \mathbb{N}$. We aim at proving that ${M}_\bullet$ converges in $\overline{\F}$ to a point belonging to  the image of the embedding $ X_0 \sqcup 
 \ldots \sqcup X_n \longrightarrow \overline{\F}$.\sk 
 
  The proof goes by induction on $\mathrm{dim}(X_0)$. Recall that the following inequalities hold true for any $\ell$ (see Proposition \ref{compare}):
  $$ \delta_\ell \leq D_\ell \qquad  \text{ and } \qquad  \sigma_\ell \leq 2 D_\ell\, . $$ 
  
  We distinguish three cases: 
\mk

\noindent (1).  \textbf{The two sequences  $\boldsymbol{(\delta_\ell)_{\ell\in \mathbb N}}$ and $\boldsymbol{(\sigma_\ell)_{\ell \in \mathbb N}}$ both do not converge to zero.}\\ 
According to Lemma \ref{bounded}, the sequence of diameters $(D_\ell)_{\ell\in \mathbb N}$ is bounded. The Delaunay decomposition provides polygonal models of $M_\ell$ such that the length of each side is bounded (see Proposition \ref{triangulation}). One can extract a subsequence such that all polygonal models have the same gluing pattern, and therefore, since $\mathrm{Im}(\rho)$ is finite, extract a subsequence whose polygonal model converges towards a non degenerate pseudo-polygon whose associated surface in $X_0$ is the limit of the Cauchy sequence ${M}_\bullet$.\sk

\noindent (2).  \textbf{The sequence $\boldsymbol{(\delta_\ell)_{\ell \in \mathbb N}}$ converges to zero while $\boldsymbol{(\sigma_\ell)_{\ell \in \mathbb N}}$ does not.}\\
 In that case, one proves that $({d}(M_\ell, X_1))_{\ell \in \mathbb N}$ converges to zero. First remark that  necessarily $\mathrm{dim}(X_0) \geq 2$ in that case. Indeed, according to Proposition \ref{collisions2}, if  $\mathrm{dim}(X_0) = 1$ we have that $\delta_\ell$ converging to zero implies that $D_\ell$ goes to infinity which would contradict Lemma \ref{bounded}. For every $\ell$, consider two singular points $p_\ell, q_\ell$ of $M_\ell$   of respective cone angles $\theta_\ell$ and $\theta'_\ell$,  such that $ d(p_\ell, q_\ell) = \delta_\ell $. Three subcases (2.i), (2.ii) and (2.iii) are to  be distinguished in this situation :
\sk
\begin{itemize}
\item[(2.i).] both curvatures $(2\pi - \theta_\ell)$ and $ (2\pi - \theta'_\ell)$ are positive (\textit{i.e.} $p_\ell$ and $q_\ell$  carry positive curvature). In that case one can always reverse Thurston's sur\-gery $\mathcal S_1$, with width of order $\delta_\ell$ (see Section \ref{Thsurgery}). \sk
\end{itemize}

\noindent  In the next two subcases (2.ii) and (2.iii), only one of the two curvatures $(2\pi - \theta_\ell)$ or $ (2\pi - \theta'_\ell)$, say the former,  is assumed to be negative: one has $\theta_\ell > 2\pi$ while $\theta'_\ell<2\pi$.  Note that in this case $(2\pi - \theta_\ell) + (2\pi - \theta'_\ell)$ is negative since there is at least one other cone point which must have positive curvature.\sk 

\begin{itemize}
 \item[(2.ii).]  In this subcase, we assume that there is `enough room to reverse Thurs\-ton's surgery $\mathcal S_2$', {\it i.e.} one can extend the geodesic line from $q_\ell$ to $p_\ell$ after the point of negative curvature on a distance of
  \begin{equation}
\label{E:deltan-}
\frac{ \sin \left(\frac{\theta'_\ell}{2}\right)}{\sin \left(\frac{\theta_\ell + \theta_\ell'}{2}\right)} \cdot \delta_\ell\end{equation}
so that the extended line cuts the cone angle at $p_\ell$ into two equal angles (see Remark \ref{formula}). If this is possible, one can cut along the extended line and fill with an appropriate Euclidean kite, and therefore reverse Thurston's surgery ${\mathcal{S}}_2$ with small width of order $\delta_\ell$. \sk

\item[(2.iii).]  we now prove that if the two subcases (2.i) and (2.ii) do not occur, we are in a situation where the kite surgery $\mathcal S_4$ can be reversed. If we cannot extend  the geodesic line from $q_\ell$ to $p_\ell$, it must be either because it meets another cone point or that the line self-intersects. The latter case cannot happen if $\ell$ is large enough otherwise the systole would be smaller than 
\eqref{E:deltan-}.  
 The fact that the line self-intersecting gives rise to a non-essential curve is not totally obvious. Actually this very curve could turn around a singular point $r_\ell$ of cone angle smaller than $\pi$. But since there can be only one such point\footnote{Indeed, since we have supposed that the cone point of negative curvature has cone angle $< 4 \pi$, there cannot be two cone points of angles $<\pi$, for otherwise the total curvature would exceed $0$ and violate the Gau\ss-Bonnet equality.} whose cone angle is smaller than $\pi$, we can play the same game with $r_\ell$ and $p_\ell$ being sure that this situation will not occur. Since $\theta_\ell$ and $\theta'_\ell$ range in a finite set, that would imply that $\sigma_\ell$ goes to zero. So the extended line meets a singular point $r_\ell$. One can try to reverse Thurston's surgery with $p_\ell$ and $r_\ell$. If $[p_\ell, q_\ell]$ is long enough we are brought back to the previous case. In the case when it is not long enough, we are going to prove that 
$$ \big(2\pi - \theta_\ell\big) + \big(2\pi - \theta'_\ell\big) + \big(2\pi - \theta''_\ell\big) =  0\, .  $$
In that case $\mathrm{dim}(X_0) = 2$ and one can reverse the kite surgery with very small width.
\sk 

Assume that Thurston's surgery $\mathcal S_2$ cannot be reversed neither with $p_\ell$ and $q_\ell$ nor with $p_\ell$ and $r_\ell$. Let $l_\ell = \delta_\ell$ be the length of the geodesic segment from $q_\ell$ to $p_\ell$ and $l'_\ell$ be the length of the one from $r_\ell$ to $p_\ell$ (see Figure \ref{config}).

\begin{figure}[!h]
\centering
\psfrag{l}[][][1]{$\quad \;\, l_\ell$}
\psfrag{lp}[][][1]{$\;\; l_\ell'$}
\psfrag{q}[][][1]{$\textcolor{red}{q_\ell}$}
\psfrag{p}[][][1]{$\textcolor{blue}{p_\ell}$}
\psfrag{r}[][][1]{$\textcolor{green}{r_\ell}$}
\includegraphics[scale=0.7]{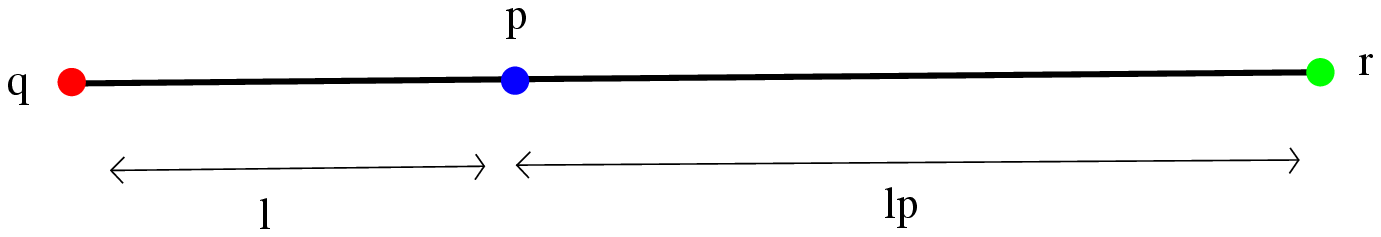}
\caption{}
\label{config}
\end{figure}
\mk 

As a consequence of Lemma \ref{reverse}, we have the following trichotomy:

\begin{enumerate}

\item if $$ l_\ell' > \frac{ \sin \left(\frac{\theta'_\ell}{2}\right)}{\sin \left(\frac{\theta_\ell + \theta'_\ell}{2}\right) }  \, 
 l_\ell$$ then perform the surgery $\mathcal{S}_2$ on the pair of points $p_\ell$, $q_\ell$;
\sk 
\item if $$ l_\ell > \frac{ \sin \left(\frac{\theta''_\ell}{2}\right)}{\sin \left(\frac{\theta_\ell + \theta''_\ell}{2}\right) } \, 
 l'_\ell $$ then perform the surgery $\mathcal{S}_2$ on the pair of points $p_\ell$, $r_\ell$;
\sk 
\item if $$ l_\ell' \leq \frac{ \sin \left(\frac{\theta'_\ell}{2}\right)}{\sin \left(\frac{\theta_\ell + \theta'_\ell}{2}\right) } \, 
 l_\ell \qquad  \text{ and } \qquad  l_\ell \leq \frac{ \sin \left(\frac{\theta_\ell''}{2}\right)}{\sin \left(\frac{\theta_\ell + \theta_\ell''}{2}\right)}\,  
l'_\ell $$ then perform the kite surgery on the points $p_\ell$, $q_\ell$ and $r_\ell$.
\end{enumerate}
\end{itemize}\mk

\noindent (3). \textbf{The sequence $\boldsymbol{(\sigma_\ell)_{\ell\in \mathbb N}}$ converges to zero.}\\
 This case is the easiest. Since $(D_\ell)_{\ell \in \mathbb N}$ is bounded and $\sigma_\ell \rightarrow 0$ as $\ell$ goes to infinity,  Proposition \ref{systol} applies for $\ell$ large enough. This implies that a Devil's surgery $\mathcal S_3$  of small width can be reversed. 
\mk 

We have proven so far that either $(M_\ell)_{\ell \in \mathbb N}$ converges to a point in $X_0$ or that for $\ell$ large enough $M_\ell $ can be recovered from a point of $X_1$ by a surgery of width going to zero as $\ell$ goes to infinity. In that latter case, Proposition \ref{proj} ensures that $d(M_\ell, X_1)$ converges to zero. Applying the induction hypothesis to a sequence $(M'_\ell)_{\ell \in \mathbb N}$ of flat surfaces $M_\ell '\in X_1$ which are such that $ \mathrm{d}(M'_\ell, M_\ell) \leq \mathrm{d}(M_\ell, X_1) \leq  {1}/{\ell}$ for any $\ell>>1$, one gets that the limit of the sequence $(M_\ell)_{\ell \in \mathbb N}$ in $\overline{\F}$ belongs to $ X_0 \sqcup X_1 \sqcup \ldots \sqcup X_n \subset \overline{\F} $.\mk 

The proof of Proposition \ref{complet} is over.
\qed
\mk

\begin{lemma}
\label{reverse}
Let $T$ be a flat torus and $p_1, p_2$ and $p_3$ three distinct singular points on it, of respective cone angles $\theta_1,\theta_2$ and $\theta_3$. Assume that  $p_1$ is the only point of negative curvature among all the cone points of $T$ and  that $p_2, p_1$ and $p_3$ sit,  in this order,  on a geodesic line broken at $p_1$, cutting the cone angle $\theta_1$ into two equal angles. Denote by $l$ the length of the part of the line from $p_1$ to $p_2$, and $l'$ the length of the part of the line from $p_1$ to $p_3$. Assume also that 
\begin{equation}
\label{E:letl'}
 l' \leq \frac{ \sin \left(\frac{\theta_2}{2}\right)}{\sin \left(\frac{\theta_1 + \theta_2}{2}\right) } 
\,  l  \qquad  \text{and} \qquad  l \leq \frac{ \sin \left(\frac{\theta_3}{2}\right)}{\sin \left(\frac{\theta_1 + \theta_3}{2}\right) } 
\,  l' .
\end{equation}

 Then the four following assertions hold true: 
\begin{enumerate}
\item $(2\pi - \theta_1) + (2\pi - \theta_2) + (2\pi - \theta_3) = 0$;  
\sk 

\item $T$ has no other cone point than $p_1$, $p_2$ and $p_3$; \sk

\item both inequalities in \eqref{E:letl'} actually are equalities; 
\sk 
 
 \item $T$ can be recovered by a kite surgery from a regular flat torus. 
\end{enumerate}
\end{lemma}

\begin{proof} The two inequalities of \eqref{E:letl'}  together yield to 
$$ \sin \left(\frac{\theta_1 + \theta_2}{2}\right)  \sin \left(\frac{\theta_1 + \theta_3}{2}\right) \leq \sin \left(\frac{\theta_2}{2}\right) \sin \left(\frac{\theta_3}{2}\right) $$
or equivalently $ \cos \big(\theta_1 + \frac{\theta_2 + \theta_3}{2} \big) - \cos \big(\frac{\theta_2 + \theta_3}{2} \big) \geq  0 $ which in its turn is equivalent to
\begin{equation}
 \label{eqn:1}
\sin\left(\frac{\theta_1 + \theta_2 + \theta_3}{2}\right)  \sin\left(\frac{\theta_1 }{2}\right) \leq 0.
\end{equation}

On the one hand, we have $-2\pi < (2\pi - \theta_1) + (2\pi - \theta_2) + (2\pi - \theta_3) \leq 0 $ because of the Gau\ss-Bonnet formula. This implies that $ 4\pi > ({\theta_1 + \theta_2 + \theta_3})/{2} \geq 3\pi $ and therefore that $\sin( (\theta_1 + \theta_2 + \theta_3)/{2}) \leq 0$, with equality if and only if $\theta_1, \theta_2 $ and $\theta_3$ sum up to $6\pi$. 
 But on the other hand, $2\pi < \theta_1 < 4\pi $ according to our hypothesis  hence $\pi < {\theta_1}/{2} < 2\pi$ and $\sin({\theta_1 }/{2}) < 0$. Inequality \eqref{eqn:1} forces $\sin(({\theta_1 + \theta_2 + \theta_3})/{2}) $ to vanish. Therefore $\theta_1 + \theta_2 + \theta_3 = 6\pi$ or equivalently 
 $$ \big(2\pi - \theta_1\big) + \big(2\pi - \theta_2\big) + \big(2\pi - \theta_3\big) =  0\, .  $$

This implies in particular that 
$$\frac{ \sin \left(\frac{\theta_2}{2}\right)}{\sin \left(\frac{\theta_1 + \theta_2}{2}\right) }  = \left(\frac{ \sin \left(\frac{\theta_3}{2}\right)}{\sin \left(\frac{\theta_1 + \theta_3}{2}\right) }\right)^{-1} $$
and therefore  one obtains that the inequalities in \eqref{E:letl'} actually are equalities.
 \sk 
 
  Note that the lengths of two consecutive sides of a kite of external angles ${\theta_1}/{2}, \theta_2, \theta_3$ satisfy the above equalities	and therefore one can cut along the aforementioned geodesic line and fill with the appropriate kite to reverse the kite surgery. 
\end{proof}

\section{\bf FINITENESS OF THE VOLUME OF $\F$}
\label{cusps}
In this section, we continue to use the notations  of the preceding one : 
$\F$ stands for  a $n$-dimensional moduli space of flat surfaces of genus 0 or $1$.
\sk 

Below, we prove  that the volume of $\F$ is finite under the hypo\-thesis that $\rho$ has finite image. Without the latter assumption (that we shall assume to hold true for the remainder of the section), it is possible to prove that the volume of $\F$ must be infinite. We will only be interested in the genus $1$ case, the genus $0$ case having already been dealt with by Thurston in \cite{Thurston}. 

Proposition \ref{diameter} essentially tells us that the lack of compactness of the metric completion of $\mathscr{F}_{[\rho]}$ is characterised by the property of having large embedded cylinders. Surfaces satisfying this property can be recovered by performing a surgery on a flat sphere along a distinguished geodesic segment between two conical points, see Section \ref{bswec}.

\subsection{Cylindrical coordinates.}

Let $T_0 \in \F$ be a torus containing a flat embedded systolic cylinder. It is built up from a flat sphere $S_0$ on which   the surgery ${\mathcal{S}}_5$ described in Section \ref{surgeries} 
has been performed 
along a geodesic segment between two conical points of $S_0$. Let $\widetilde{\rho}$ be such that $S_0 \in \mathscr{F}_{[\widetilde{\rho}]}$ and let $(z_0, \ldots , z_{n-1}) $ be a local linear parametrisation of $\mathscr{F}_{[\tilde{\rho}]}$ at $S_0$ such that $z_0$ represent the geodesic path along which the surgery is performed, and let $z_n$ be the complex number such that the inserted cylinder has sides $z_0, z_n$. Then $(z_0, \ldots, z_{n})$ is a linear parametrisation of $\mathscr{F}_{[\rho]}$. We call any such parametrisation a \textbf{cylindrical parametrisation} whose existence is guaranteed by Proposition \ref{top}.

 Let $A$ be the area form of the flat tori in $\F$ close to $T_0$ 
 expressed in the coordinates $z_0, \ldots, z_{n}$  
 and denote by $B$ the area form of the associated flat spheres in $\mathscr{F}_{[\tilde{\rho}]}$, expressed in the coordinates $z_0, \ldots , z_{n-1}$.

  The (signed) area of the aforementioned embedded flat cylinder is $\mathrm{Im}(z_n\overline{z_0}) $ therefore the  two area forms 
 $A$ and $B$  are  linked by the following relation : 
$$ A\big(z_0, \ldots, z_{n}\big) = B\big(z_0, \ldots, z_{n-1}\big) + \mathrm{Im}\big(z_n\overline{z_0}\big).$$ 

 Normalising with $ z_0 = 1$, we get a genuine parametrisation of $\mathscr{F}_{[\rho]}$ (resp. $\mathscr{F}_{[\tilde{\rho}]}$)  $(z_1, \ldots z_{n})$ (resp. $(z_1, \ldots z_{n-1})$) and the preceding relation becomes 
$$ A\big(1,z_1, \ldots, z_{n}\big) = B\big(1,z_1, \ldots, z_{n-1}\big) + \mathrm{Im}\big({z_n}\big).$$ 

\subsection{Finiteness of the volume.}
\label{SS:Finiteness}
The strategy to estimate the complex hyperbolic volume of $\F$ is to restrict ourselves to parts of $\F$ where the diameter is large (\textit{i.e.} where corresponding flat tori have large embedded flat cylinders, see Proposition \ref{diameter}) and use the cylindrical coordinates defined above to perform some quasi-explicit estimations.\mk

In what follows,  all the flat  cylinders which we will consider will be assumed to be `{\it maximal}' in the sense that none of them is a  proper subcylinder of an embedded flat cylinder of the same width but of strictly higher length.\sk

For every positive $\epsilon $, one sets
\begin{align*}
A_{\epsilon} = & \, \big\{ \, T \in \F  \ \big| \ \sigma(T) = \epsilon \ \text{and} \ T \ 
\mbox{contains a flat cylinder of width }  \epsilon\,  \big\} \vspace{0.15cm} \\
\mbox{ and }\; 
B_{\epsilon} = & \,  \big\{\,  T \in \F  \ \big|   \ \sigma(T) \leq \epsilon \ \text{and}\ T \ \mbox{contains a flat cylinder of width } \sigma(T)\,  \big\}.
\end{align*}
As  it is often implicitly assumed in a large part of the paper, the points of $\F$ are flat surfaces which are supposed to be normalised in order that their area is $1$. In particular we assume this hypothesis in the definitions above. 
\sk

Both $A_{\epsilon}$ and $B_{\epsilon}$ are closed subsets of $\F$. 
Moreover, from Section \ref{bswec},  it comes that  when non-empty, 
$A_{\epsilon}$ is a smooth real-analytic hypersurface in $\F$.

For a given $\epsilon>0$, the elements of $A_{\epsilon}$ can be modified by thickening of a length $t$ the embedded flat cylinder of width $\epsilon$ (by thickening, we mean  replacing the cylinder of length $l$ by a cylinder of length $l + t$). When renormalising in order
for the area to be $1$,  
  the width of the cylinder becomes smaller than $\epsilon$. This defines a map 
\begin{equation}
\label{E:zeudilla}
  A_{\epsilon} \times \mathbb{R}_{\geq 0}  \longrightarrow B_{\epsilon}
  \, ,  
  \end{equation}
  which is a local diffeomorphism (see Section \ref{bswec}). 
The fact that this map is well-defined relies on the uniqueness of the maximal cylinder of width $\epsilon=\sigma(T)$ (the `{\it systolic cylinder}') in any $T\in A_\epsilon$ for $\epsilon$ small enough. This fact follows 
easily from 
Lemma \ref{unique} which is proved  
in section \S\ref{S:Lemma8.4}. 
Note that from this Lemma, it also can be deduced  
that $A_\epsilon$ is precisely the boundary of $B_\epsilon$.

  \begin{prop}
\label{littlecylinders}  
For any $\epsilon$ sufficiently small,  
the map $   A_{\epsilon} \times \mathbb{R}_{\geq 0}  \longrightarrow B_{\epsilon} $ is onto. 
\end{prop}
 \begin{proof} First remark that it is sufficient to prove the proposition for a fixed $\epsilon_0 >0$, because the statement  will then hold true for every smaller $\epsilon$.
 \sk 
 
Second, since $B_\epsilon$ is the disjoint union of the $A_\eta$'s for $\eta\in ]0,\epsilon]$, 
 the proposition  follows  from the fact that,  for every $t> 0$,
 the  image of  $A_\epsilon\times \{t\}$ by \eqref{E:zeudilla} is the whole 
hypersurface  $A_{\epsilon/\sqrt{1+\epsilon t}}$, as soon as $\epsilon$ is taken sufficiently small. This technical assertion is proved 
in Subsection \ref{SS:AeAeta} below.
\end{proof} 
\mk 

For the remainder of this section, we fix  $\epsilon$ such that  
\eqref{E:zeudilla} is surjective. Now remark that the closure of $\F \setminus B_{\epsilon}$ in $\overline{\F}$ is compact. Indeed, a sequence in $\F \setminus B_{\epsilon}$ must have bounded diameter according to Proposition \ref{diameter}. But then, up to passing to a subsequence, it is a Cauchy sequence by Lemma \ref{bounded} and therefore converges in $\overline{\F}$. Since $A_{\epsilon} = \partial  B_{\epsilon}$, its closure in $\overline{\F}$ must be compact as well. 
\mk 

We are now able to prove the 
\begin{prop}
The volume of $B_{\epsilon}$ is finite.
\end{prop}
\begin{proof} Since the closure of $A_{\epsilon}$ in 
$\overline{\F}$ 
is compact, $A_{\epsilon}$ can be recovered by a finite union of simply-connected open sets $(U_i)_{i\in I}$ such that for each $i\in I$: 
\begin{enumerate}
\item the diameter of  $U_i$ for the complex hyperbolic metric is finite; \sk
\item there are cylindrical coordinates defined on $U_i$. \sk
\end{enumerate}

More precisely, each element in $U_i$ can be recovered from surgery ${\mathcal{S}}_5$ on a sphere of a certain leaf  $\mathscr{F}_{[\widetilde{\rho}]}$ along a geodesic joining two singular points and we have a linear parametrisation $(z_0, \ldots, z_n)$ of $U_i$ such that 

\begin{itemize}

\item $z_0$ parametrises the geodesic along which the surgery is performed; \sk

\item $(z_0, z_n)$ parametrises the added cylinder; \sk

\item $(z_0, \ldots, z_{n-1})$ is a linear parametrisation of $\mathscr{F}_{[\widetilde{\rho}]}$.

\end{itemize}

The area form therefore writes down the following way
\begin{equation}
\label{E:pooli}
 A\big(z_0, \ldots, z_{n}\big) = B\big(z_0, \ldots, z_{n-1}\big) + \mathrm{Im}\big(z_n\overline{z_0}\big) \, .
\end{equation}
 Normalising with $ z_0 = 1$, we get a parametrisation of $\mathscr{F}_{[\rho]}$ (resp.\;of $\mathscr{F}_{[\widetilde{\rho}]}$) by $(z_1, \ldots z_{n})$ (resp.\;by $(z_1, \ldots z_{n-1})$) and the preceding relation becomes 
$$ A\big(1, z_1,\ldots, z_{n}\big) = B\big(1, z_1,\ldots, z_{n-1}\big) + \mathrm{Im}\big(z_n\big) \,.$$

 In this chart the local diffeomorphism $  A_{\epsilon} \times \mathbb{R}_{\geq 0}   \longrightarrow B_{\epsilon}$ is given by
$$\big ( (z_1, \ldots, z_n, \theta), t \big) \longmapsto \big(z_1, \ldots, z_{n-1}, z_n + it + \theta\big) \, .
$$
\noindent where $\theta$ is the twist parameter of the cylinder of width $\epsilon$ in $A_{\epsilon}$ (see Figure \ref{volume} below). At this point, we would like to stress that $(z_1, \ldots, z_n, \theta)$ is not a system of coordinates on $A_{\epsilon}$, but the latter written in these coordinates is a real-analytic submanifold of  codimension 2. 


\begin{figure}
\centering
\psfrag{Ext}[][][0.95]{\qquad \qquad \qquad \begin{tabular}{c}
\vspace{0.2cm} \\ Extension of the \vspace{-0.08cm}\\ original cylinder
\end{tabular}}
\psfrag{F}[][][1]{Former gluing point \qquad \qquad \qquad  \; \; }
\psfrag{N}[][][1]{\qquad \qquad \quad \; New gluing point}
\psfrag{z2}[][][1]{$\;\;z_2$}
\psfrag{t}[][][1]{$t$}
\psfrag{ZZ}[][][1]{$Z_Z$}
\psfrag{theta}[][][0.8]{$\theta\quad \; \; $}
\psfrag{zns}[][][1]{$z_n$}
\psfrag{z1}[][][1]{$\;\, z_1$}
\psfrag{z0=1}[][][1]{$1=z_0\qquad \;\; $}
\includegraphics[scale=0.60]{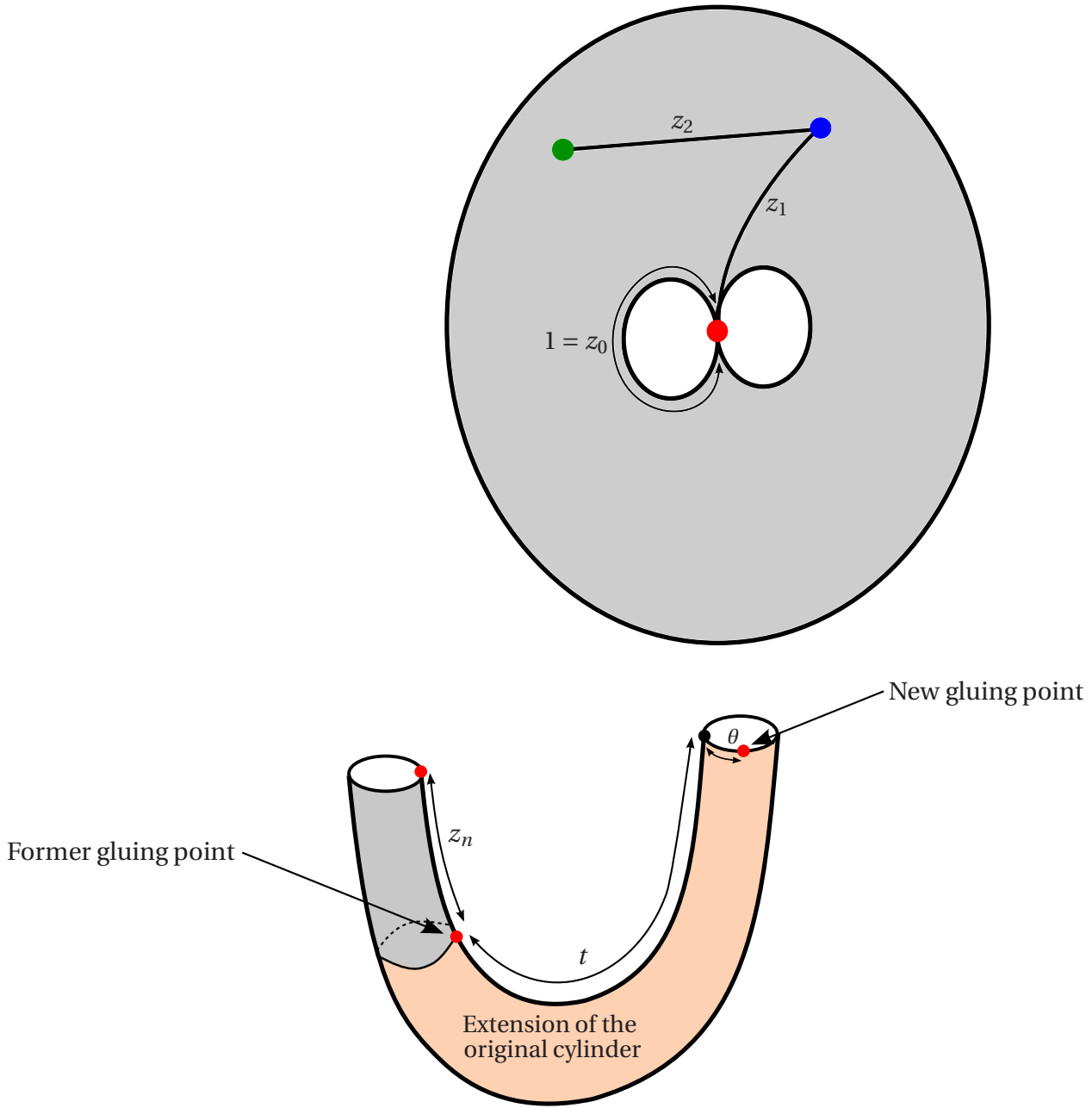}
\caption{Parametrisation of the extension of the cylinder.}
\label{volume}
\end{figure}


In view of \eqref{E:pooli}, $(z_1, \ldots, z_{n-1}, z_n + it + \theta\big)$ is a system of pseudo-horospherical coordinates on $B_\epsilon$ (see Appendix \ref{CHG} where this notion is introduced and discussed).
 Since the diameter of each $U_i$ is finite,  the image of such a map restricted to $B_i =(A_{\epsilon}\cap U_i) \times  \mathbb{R}_{\geq 0}  $ is included into a domain $U_{K_i,\lambda_i}$ introduced in Appendix \ref{CHG}, for some $K_i,\lambda_i>0$.  It follows (from Lemma A.3) that the complex hyperbolic volume ${ \mathrm{vol}(B_i) } $ of $B_i$  is finite for any $i\in I$. There are only a finite number of $U_i$'s covering $A_{\epsilon}$ and since $\epsilon$ 
 has be taken such that   the map  
 \eqref{E:zeudilla}
  is onto, one gets
$$\mathrm{vol}\big(B_{\epsilon}\big) \leq \sum_{i\in I}{ \mathrm{vol}(B_i) }\, ,  $$
which implies that the complex hyperbolic  volume of $B_{\epsilon}$ is finite. 
\bk  
\end{proof}

As mentioned above, the finiteness of the volume of $\F$ follows from the preceding proposition, hence we have proved the following theorem:

\begin{thm}
Assume $g=1$ and $\theta$ is such that $2\pi<\theta_1<4\pi$ and $\theta_i<2\pi$ for $i \geq 2$. If $\rho$ has finite image, then the volume of $\F$ for its complex hyperbolic structure is finite.
\end{thm}


\subsection{A uniqueness result for cylinders of small width.}
\label{S:Lemma8.4}
We now prove the following lemma which implies 
 the result announced in Subsection \ref{SS:Finiteness}  (namely, 
 for $\epsilon$ sufficiently small, the unicity of the systolic cylinder in any element of $A_\epsilon$):


\begin{lemma}
\label{unique} 
There exists $\epsilon_\rho >0$ (depending only on $\rho$) such that for any positive  $\epsilon <\epsilon _\rho$ and any flat torus $T\in \F$, 
the following holds true: if $T$ belongs to $B_\epsilon$, i.e. if $T$ contains an embedded  flat cylinder of width $\sigma(T)\leq \epsilon$, then the latter is unique  among the flat cylinders embedded in $T$ of width strictly less than $\epsilon_\rho$.
\end{lemma}

\newpage

\begin{proof}  
 Since ${\rm Im}(\rho)$ is finite, the number of genus 0 moduli spaces $\mathscr M_{0,\theta'}$ that can be obtained  from elements of $\F$ by reversing the surgery $\mathcal S_5$ is finite.  
  Consequently, the minimum $\kappa$ of the set of positive constants $K(\theta_1',\theta_2')$ given by Lemma \ref{collisions}  for the corresponding angle data $\theta'$ (of course each time with respect to  the two cone points involved in the surgery),   is positive as well.  \sk

Let  $T\in \F$ be an element of $B_\epsilon$, for a fixed $\epsilon>0$ supposed to be strictly less than $\kappa$. By definition of $B_\epsilon$, $T$ contains a flat cylinder $C_1$ of width 
$\epsilon_1=\sigma(T)$ with $\epsilon_1\leq \epsilon$.  To prove the lemma  (with $\epsilon_\rho=\kappa$), we argue by contradiction  by assuming that $T$ 
contains another flat cylinder $C_2$, of width $\epsilon_2<\kappa$.

First assume that the interiors of these  two cylinders intersect.  Then there exists a  boundary component $\partial_2$ of $C_2$ which intersects $C_1$. Since $\partial_2$ is totally geodesic, it 
 must enter $C_1$ through one of its boundary component and exit by the other. But $\partial_2$ has length $\epsilon_2$ so $C_1$ has length at most $\epsilon_2$. 
  On the other hand, since the width of $C_1$ coincides with the systole of $T$, it follows from 
 Lemma \ref{L:RealisationOfTheSystol} that $C_1$ is a `systolic cylinder', that is a flat embedded cylinder as  in Figure \ref{F:S5} with respect to which the surgery $\mathcal S_5$ can be reversed. 
But removing $C_1$ from $T$ and  
inverting the surgery $\mathcal S_5$ would give a flat sphere with a 
short geodesic between two of its cone points. 
After renormalization of  the area, this sphere and this geodesic together would contradict Lemma \ref{collisions} since $\epsilon_2<\kappa$ (the computational details are left to the reader). This shows that $C_1$ and $C_2$ have disjoint interiors. 
\sk

 Let $c_1$ (resp.\;$c_2$) be a closed geodesic of length $\epsilon_1$ (resp.\;$\epsilon_2$)  contained in the interior of $C_1$ (resp.\;of $C_2$).  Cutting $T$ along $c_1$ and $c_2$ gives us two connected components $\Sigma$ and $\Sigma'$, each of them being a flat sphere with a  boundary formed by two disjoint totally geodesic circles (in other terms : both $\Sigma$ and $\Sigma'$ are cylinders too).  Note that since $C_1$ and $C_2$ are of maximal length, any components of their boundary contains a conical point 
 from which it follows that both $\Sigma$ and $\Sigma'$ contain conical points. 
 Now assume  that the unique  cone point  of negative curvature of $T$ belongs to $\Sigma'$.  
 Then $\Sigma$ contains only cone points of positive curvature. But  
this is impossible since, 
according to Gau\ss-Bonnet formula (applied to the flat surface with geodesic boundary $\Sigma$), the total curvature of $\Sigma$ is equal to  $2\pi\cdot \chi(\Sigma)=0$.  The lemma follows from this contradiction.
\end{proof}



\subsection{A technical lemma}
 \label{SS:AeAeta} 
 We fix $\epsilon>0$. For any positive $\eta<\epsilon$, it is easily seen that the preimage of $A_\eta\subset B_\epsilon$ by \eqref{E:zeudilla} is $A_\epsilon\times \{ t_{\eta}\}$ with $t_\eta={(\epsilon^2-\eta^2)}/{(\epsilon\eta^2)}>0$. 
 \begin{lemma}
 \label{L:Technical}
 For $\epsilon$ sufficiently small, any map $A_\epsilon\times \{ t_{\eta}\}\rightarrow A_\eta$ is surjective.
 \end{lemma}
 \begin{proof}
We claim that the statement of the lemma holds true for any  $\epsilon<\kappa$, where 
$\kappa$ stands for the positive constant introduced in the first paragraph of 
 the  proof of Lemma  \ref{unique}. 
Indeed, if it were not the case, there would exist $T\in A_\eta$ for some $\eta<\epsilon$,  which was not in the image of $A_\epsilon\times \{ t_{\eta}\}\rightarrow A_\eta$.  For such a $T$, one verifies  that the length of the systolic flat cylinder  of $T$  (of width $\eta$) is necessarily less than or equal to 
$\eta t_\eta/\epsilon$.  Then removing this cylinder from $T$ and  
inverting the surgery $\mathcal S_5$ would give a flat sphere with a 
short geodesic between two of its cone points. 
After renormalization of  the area, this sphere and this geodesic together would contradict Lemma \ref{collisions} since $\epsilon<\kappa$ (the computational details are left to the reader).
 \end{proof}

\subsection{Proof of Lemma \ref{bounded}}

We finally  explain  how the above description of the parts of $\F$ consisting of tori with long embedded cylinders gives a proof of the first point of Lemma \ref{bounded}, namely that  the diameters of the elements of any Cauchy sequence in $\F$ are uniformly bounded.\sk

 Consider a path $\gamma : [0,1] \longrightarrow B_{\epsilon}$. We have the following estimate 
$$ L(\gamma) \geq \Big| \log\big(c\big(\gamma(1)\big)\big) - \log
\big(c\big(\gamma(0)\big)\big) \Big| . $$ 
where $c(T)$ is the length of the cylinder of width at most $\epsilon$ (with $T \in B_{\epsilon}$). This is a direct consequence of Lemma \ref{cylindricalcalculation} of Appendix A. \sk 

Assume that we have a Cauchy sequence whose diameter goes to infinity. 
We can assume that all its elements  belong to a subset $B_{\epsilon}$  of $\F$ considered above (this follows from Proposition \ref{diameter}).  Applying the above estimate to paths linking elements of the sequence leads to a contradiction.


\section{\bf THE METRIC COMPLETION IS A CONE-MANIFOLD}
\label{cone-manifoldness}

 In this section, we prove a theorem describing the structure of the metric completion of $\F$. We refer to Appendix \ref{conemanifolds} for further precisions and references on the notion of cone-manifold. \sk 

We turn back to the notations used before the two preceding sections : we are dealing with flat surfaces of genus $g=0,1$ with $n$ cone points. More precisely, we assume that  either
  \begin{itemize}
\item $g=0$ and $\theta_i \in ]0,2\pi[$ for all $i=1,\ldots,n$; or
\sk 
\item $g=1$, $\theta_1 \in ]2\pi, 4\pi[$ and $\theta_i  \in ]0,2\pi[$ for $i=2,\ldots,n$, 
\end{itemize}
so that Veech's geometric structure on $\F$ is complex hyperbolic.

\begin{thm}
\label{cone-manifold}

Let $\rho \in \mathrm{H}^1(N, \mathbb{U}, \theta)$ be such that $\mathrm{Im}(\rho)$ is finite. The metric completion of $\mathscr{F}_{[\rho]}$ is a complex hyperbolic cone-manifold.
\end{thm}

The proof goes by induction on $m = \mathrm{dim}(\mathscr{F}_{[\rho]})$. Assume that it has been proven that the theorem holds true for all $\F$ carrying a complex hyperbolic structure such that $\mathrm{dim}(\mathscr{F}_{[\rho]}) \leq m-1$. The case $g=0$ has been dealt with  by Thurston in \cite{Thurston}. The base case of the induction is when $m=0$ that is $\mathscr{F}_{[\rho]}$ is a point in which case the theorem holds. Note that $m=0$ can only happen if $g=0$ and $n=3$.\sk

Consider $\rho \in \mathrm{H}^1(N, \mathbb{U}, \theta)$ such that $ \mathrm{dim}(\mathscr{F}_{[\rho]}) = m$. We have proven in Section \ref{completion} (see Theorem \ref{mcompletion}) that $X=\overline{\mathscr{F}_{[\rho]}} $ is a disjoint union of $X_0, \ldots, X_m$ such that 

\begin{itemize}
\item $X_0 = \mathscr{F}_{[\rho]}$; \sk
\item $X_i$ is a complex hyperbolic manifold of dimension $m - i$; \sk
\item the metric completion of $X_i$ in $X$ is $X_{i+1} \sqcup \ldots \sqcup X_m$. 
\end{itemize}
 
We prove by induction on $i$ that any point in $X_i$ has a neighbourhood in $X$   isometric to a complex hyperbolic cone-manifold. Let $p$ be a point of $X_i$. It has a neighbourhood $U_p$ in $X_i$ which is isometric to an open subset of ${\mathbb{CH}}^{m-i}$. Following Thurston in \cite{Thurston}, we define an \textit{`orthogonal  projection'} $\pi : V_p \longrightarrow U_p $ from a neighbourhood $V_p$ of $p$ in $X$ onto $U_p$ the following way : we have seen in Section \ref{completion} that there exists a neighbourhood $V_p$ of $p$ in $X$ such that for any $q \in V_p$ there exists a unique $r \in U_p $ such that $q$ can be recovered from $r$ by performing a finite number of the four surgeries 
$\mathcal S_1, \mathcal S_2,  \mathcal S_3 $ and $  \mathcal S_4$  
described in Section \ref{surgeries}, and one has $\pi(q) = r$. 

Thurston calls this map '{\it orthogonal projection}' because,  in a sense which is made precise in Appendix \ref{conemanifolds}, the fibers of $\pi$ are orthogonal to its image $U_p$.

\begin{lemma}
\label{foliation}
Let $p$, $U_p$, $V_p$ and $\pi$ be  defined as above. 

\begin{enumerate}

\item For all $r \in U_p$, $V(r) = \pi^{-1}(r) \setminus \{r\}$ is foliated by geodesics ending at $r$.
\sk 
\item For all $r \in  U_p$, the intersection $V(r) \cap (X_0 \sqcup \ldots \sqcup X_{i-1})$ is a totally geodesic sub-cone-manifold of $X_0 \sqcup \ldots \sqcup X_{i-1}$.
\sk 
\item For all $r\in U_p$, $V(r)$ is orthogonal to $U_p$.
\end{enumerate}
\end{lemma}

\begin{proof}  We fix $r\in U_p$ and consider an  
 element $q$ of $V(r) = \pi^{-1}(r)$. By using an appropriate topological 
polygonation (cf. Proposition \ref{top}), one can find a linear parametrisation 
$(z_1,\ldots,z_m)$  at $q$ such that 
\begin{itemize}
\item if  $(\xi_0, \ldots, \xi_{m-i}, \ldots, \xi_m)$ are the coordinates of $q$ in this parametrisation, then $\pi(q)$ has coordinates $(\xi_0, \ldots, \xi_{m-i}, 0, \ldots, 0)$; \sk
\item  the area form $A$ in the $z_i$'s can be written out 
\begin{equation}
\label{E:SplittingOfA}
A(z_0, \ldots, z_m) = A_1(z_0, \ldots, z_{m-i}) - A_2(z_{m-i+1}, \ldots, z_m)
\end{equation}
 where $A_1$ has signature $(1,m-i)$ and $A_2$ is positive-definite ($A_2$ is the total area removed by the successive surgeries).
 \sk 
 \item  in the local coordinates $z_1,\ldots,z_m$,  the stratum $X_J$ for $j>i$ is cut out by the equations $ z_{m - (j+1)} = z_{m -{j+2}} = \cdots = z_m = 0$ (there is a subtlety here about  the precise location of  the locus cut out by these equations with respect to the domain 
 of definition 
 of the  linear parametrisation we consider. 
We let the reader state a remark analogous to Remark \ref{R:C'est Subtil} for the case under scrutiny).


 \end{itemize}
\sk

 The image of 
 $ [0,1]\ni t \longrightarrow (\xi_0, \ldots, \xi_{m-i},   t\xi_{m-i-1},\ldots, t\xi_m)\in \mathbb C^{m+1}$ projects onto a geodesic path in $V_p$ joining $q$ to $r=\pi(q)$ (see Lemma A.1.(2)).
 Remark that this geodesic does not depend on the choice of the $z_i$'s: it is the one  pointing in the direction of $r$ hence it is intrinsic (because there is a unique geodesic segment linking to distinct points in a complex hyperbolic space). The collection of those geodesics for all $q \in V(r)$ gives the announced foliation of $V(r)$. Taking $V_p$ small enough, one can ensures that the foliation is globally well defined by using for instance a finite number of linear parametrisation whose pairwise intersections are $1$-connected.   The first point of the lemma is proved. \sk

 A neighbourhood of $q $ in $V(r) \cap X_0$ consists of the submanifold parametrised by $z_1', \ldots, z_m'$ such that $z_1' = \xi_1, \ldots, z_{m-i}' = \xi_{m-i}$ and therefore projects onto a totally geodesic subspace ${\mathbb{CH}}^i$ (we use the fact that a complex affine submanifold of the complex hyperbolic space is totally geodesic).
\sk

 Finally, the   splitting \eqref{E:SplittingOfA} of $A$ gives us that $V(r)$ and $U_p$ are orthogonal.\mk
\end{proof}

We continue to use the notations of the previous lemma. 
\begin{prop}
 WIth the same notations as before, $V(r)$ is a\;$\mathbb C\mathbb H^i$-cone-manifold with  $r$ as its unique cone point.
\end{prop}
\begin{proof}
Define 
\begin{align*}
B(\epsilon) = & \, \big\{ q \in V(r) \ | \ \exists \ 
\mbox{ a geodesic of length } \leq \epsilon \mbox{ joining } q \mbox{ to } r
 \big\}
\vspace{0.25cm}
\\ 
\mbox{ and }\;\; S(\epsilon) =  &\, \partial B(\epsilon) = \big\{ q \in V(r) \ | \ \exists 
\mbox{ a geodesic of length } \epsilon
\mbox{ joining } q \mbox{ to } r
\big\}
\end{align*}

For $\epsilon$ small enough, $S(\epsilon)$ does not meet $X_j$ for $j \geq i $ then it lives in 
$$X \setminus \cup_{j\geq i }X_j = X_0\sqcup X_1\sqcup\ldots\sqcup X_{i-1}$$ which is a complex hyperbolic cone-manifold
 according to the induction hypothesis. 
 
  In particular $S(\epsilon)$ is locally a totally geodesic sub-cone-manifold intersected with a piece of a complex hyperbolic sphere whose centre belongs to $V(r)$. It is therefore, according to Lemma \ref{stab}, a $(S^{2i-1}, \mathrm{U}(i))$-cone-manifold. According to the previous lemma, $B(\epsilon)$  is a cone over this cone-manifold and the proposition is proved. \end{proof}

\begin{prop}
For any $r \in U_p$, there exists a neighbourhood of $r$ in $X$ which is a complex hyperbolic cone-manifold. 
\end{prop}

\begin{proof} There exists an $\epsilon$ such that for all $r \in U$, the ball of radius $\epsilon$ at $r$ in $V_r$ is an embedded cone. There is a neighbourhood of $U$ in $X$ which has the product structure $ U \times B(\epsilon)$ satisfying the hypothesis of Proposition \ref{highercones}. Hence $ U \times B(\epsilon)$ is a complex hyperbolic cone-manifold and this proves the proposition.\mk 
\end{proof}

The induction process can be carried on which proves that $X = \overline{\mathscr{F}_{[\rho]}}$ is a complete complex hyperbolic cone-manifold.
\bk

\section{\bf LISTING THE $\boldsymbol{\F}$'s AND THEIR CODIMENSION $1$ STRATA}

\label{listing}

We have given so far a rather abstract analysis of the geometric structure of a leaf $\F$ when $\mathrm{Im}(\rho)$ is finite. We now give a list of all such $\F$ associated to a rational angle datum $\theta$ when $g=1$. Let $G_{\theta} \subset \mathbb{U}$ be the subgroup generated by $e^{i\theta_1}, \ldots, e^{i\theta_n}$ and $\omega_{\rho}$ a root of unity such that $\mathrm{Im}(\rho) = \langle \omega_{\rho} \rangle$. 
\sk 

Let $M$ be the smallest  positive  integer such that  $G_{\theta} = \langle \omega_{\rho}^M \rangle$.

\subsection{Listing the $\boldsymbol{\F}$'s associated to $\boldsymbol{\theta}$.}

The starting point of our description is the following lemma :

\begin{lemma}
\label{MCG}
Consider $\rho$ and $\rho'$ two elements of\, $\mathrm{H}^1(N, \mathbb{U}, \theta)$ such that: 
\begin{enumerate}
\item $\mathrm{Im}(\rho) = \mathrm{Im}(\rho')$; 
\sk
\item $\mathrm{Im}(\rho)$ {\rm \big(}and therefore $\mathrm{Im}(\rho')${\rm \big)} is finite.
\end{enumerate} Then $\rho$ and $\rho'$ are equivalent under the action of the pure mapping class group.
\end{lemma}

\begin{proof} 

Let $\rho \in \mathrm{H}^1(N, \mathbb{U}, \theta)$ be such that ${\rm Im}(\rho)$ is finite, equals to $\langle \exp({2i\pi}/{q})\rangle$ for some positive integer $q$.
\sk 

We consider be two simple closed curves  $a,b$  avoiding the marked points of $N$ which form a symplectic basis of the homology of the unmarked torus $N$, and $c_1, \ldots, c_n$  curves that circle the marked points, 
chosen in such a way that  $\rho_k=\rho(c_k) = \exp(i\theta_k)$ for $k=1,\ldots,n$.
 Up to the action of an element of the pure mapping class group, we can replace 
$a$ and $b$ by $a^\alpha b^\beta$ and $a^\gamma b^\delta$ with 
 ${\tiny{\begin{bmatrix}
\alpha \!\! &\!\!\!\! \beta \\
\gamma  \!\!& \!\! \!\!\delta \end{bmatrix}}}
\in \mathrm{SL}_2(\mathbb{Z}) $.  We can this way arrange that $\rho_a=\rho(a) = 1$ with $\rho_b=\rho(b)$ such that $\mathrm{Im}(\rho)=\big\langle G_{\theta} \, , \, \rho_b
\big\rangle$. From this equality and thanks to our assumption, there exist integers $\nu_1,\ldots,\nu_n$ and $\nu\neq 0$ such that $(\rho_b)^\nu \cdot  \rho_1^{\nu_1}\cdots
 \rho_n^{\nu_n}=e^{2i\pi/q}$ hence   
 generates $\mathrm{Im}(\rho)$.   Then replacing $a$ and $b$ by $ab^{\nu-1}$ and $ab^{\nu}c_1^{\nu_1}\cdot c_n^{\nu_n}$ respectively, one can assume that $\mathrm{Im}(\rho)=\langle \rho_b\rangle $ with $\rho_b = e^{2i\pi/q}$. Then $\rho_a=(\rho_b)^\mu$ for a certain integer $p$. Then considering $ab^{q-p}$ instead of $a$, one eventually obtains
that the element of $\mathrm{H}^1(N, \mathbb{U}, \theta)$ uniquely characterized by 
assignating the values 1 and $e^{2i\pi/q}$ to $a$ and $b$ respectively, 
 is a representative of the orbit of the initial character $\rho$ under the action of the pure mapping class group.\sk 
 As an immediate consequence of the preceding fact,  we get that the class of  an element $\rho\in  H^1(N, \mathbb{U}, \theta)$ under the action of the pure mapping class group only depends on its image, if the latter is finite.  This gives the lemma.  \end{proof}
\sk

\vspace{2mm}
A leaf $\F$ is therefore only determined by its associated angle data $\theta$ and the smallest integer $M\geq 1$ which is such that $G_{\theta}$ is generated by $\omega_{\rho}^M$. From now on, we refer to such a leaf/moduli space as $\mathscr{F}_{\theta}(M)$. We are now going to give a description of its codimension $1$ strata.

\vspace{2mm}

We distinguish three types of such strata :

\begin{itemize}

\item the $\boldsymbol{P}$\textbf{-strata} which are obtained from Devil's surgery. Here \textit{`P'} stands for \textit{pinching};
\sk
\item the  $\boldsymbol{C}$\textbf{-strata} which are obtained from a Thurston's type surgery. Here  \textit{`C'} stands for \textit{colliding};
\sk 
\item the $\boldsymbol{K}$\textbf{-strata} which are obtained from a kite surgery. Here \textit{`K'} stands for \textit{kite}. 
%
\sk 
\end{itemize}
(Note that since a $K$-stratum  appears when three cone points collide together, it can be considered as a particular kind of $C$-stratum).
\mk 

At this point we must mention an aspect of the description that we have so far ignored: we have been using throughout the article the terminology \textit{`leaf'} in a non-standard way. While in foliation theory, a leaf is automatically supposed to be connected, the definition we use allows $\mathscr{F}_{\theta}(M)$ not to be. 
Such a non connectedness phenomena
 can indeed happen when $N$ has genus $1$ as the explicit description of the case $g=1$, $n=2$ carried on in \cite{GhazouaniPirio1} reveals (see \cite[\S4.2.5]{GhazouaniPirio1} for an explicit example). The determination of the connected components in the general case is an open problem that seems interesting to the authors for the reasons explained in Section \ref{holonomy}. 
\mk

In what follows, we fix a leaf $\F$ with ${\rm Im}(\rho)$ finite  and we  explain below  several algorithms to determine the strata of complex codimension 1 appearing in the completion $\overline{\F}$.

\subsection{Finding the $\boldsymbol{P}$-strata.}

Let $m$ be the positive integer such that $G_{\theta}$ is the subgroup of $\mathbb{U}$ generated by $e^{\frac{2i\pi}{m}}$.  
 With these notations, one has
$$
\mathrm{Im}\big(\rho\big)=
\left\langle e^{\frac{2i\pi}{mM}}
\right\rangle .
$$
 
 Since $\theta_1 > 2\pi$ there exists $p$ such that 
 $$ \theta_1 = 2\pi\left(1 + \frac{p}{m}\right)\, .$$ 
 
 A $P$-stratum of codimension 1 is a (finite cover) of a moduli space of flat spheres whose angles datum is $(\theta_1', \theta_1'', \theta_2, \ldots, \theta_n)$ with $\theta_1'$ and $\theta_1''$  such that 
\sk
\begin{enumerate}
\item $\theta_1' + \theta_1'' = \theta_1 - 2\pi$; 
 \sk
\item both $e^{i\theta_1'}$ and $e^{i\theta_1''}$ belong to $\mathrm{Im}(\rho) = \big\langle e^{\frac{2i\pi}{mM}} \big\rangle$.
\end{enumerate}

 This condition is sufficient for such a moduli space of flat spheres to appear as a  stratum of the metric completion of $\F$. There is therefore a $P$-stratum for each way of decomposing the integer $pM$ as a sum of two positive integers, this number being $\big\lfloor \frac{pM}{2}\big\rfloor$.
 
 The stratum associated to a decomposition 
\begin{equation}
\label{decomposition}
\tag{$\mathscr{D}$}  pM = r' + r''  
\end{equation}
with $r',r''>0$ is a finite cover of the moduli space of flat spheres whose angles datum is 
$$\left(e^{\frac{2i\pi r'}{mM}}, e^{\frac{2i\pi r''}{mM}}, \theta_2, \ldots, \theta_n\right).$$

According to Section \ref{coneangle}, the {cone angle} around the stratum associated to the decomposition \eqref{decomposition} is 
$2\pi \cdot {\mathrm{lcm}(r',r'')}/{(mM)}$.

\subsection{Finding the $\boldsymbol{C}$-strata.}

A $C$-stratum of codimension $1$ is a moduli space of flat tori with $n-1$ cone points corresponding to the collision of two cone points $p_k$ and $p_l$ of respective angles 
$\theta_k$ and $\theta_l$. 
The new angle datum $\theta'=(\theta_i')_{i=1}^{n-1}$ is such that $\theta_k$ and $\theta_l$ have been replaced by $\theta_k + \theta_l - 2\pi$, the other $\theta_i$'s staying unchanged. \sk 

A holonomy character $\rho' \in \mathrm{H}^1(N, \mathbb{U}, \theta') $ is such that a finite cover of $\mathscr{F}_{[\rho']}$ is a stratum of $\F$ if and only if $\mathrm{Im}(\rho)$ is generated by $\mathrm{Im}(\rho')$, $e^{i\theta_k}$ and $e^{i\theta_l}$ (see Subsection \ref{Thsurgery}) . We describe now the positive integers $M'$ which are such that $\mathscr{F}_{\theta'}(M')$ appears as a $C$-stratum of the metric completion of $\F=\mathscr{F}_{\theta}(M)$. \sk 

 Let $m'$ be the positive integer such that $G_{\theta'}$ is generated by $e^{\frac{2i\pi}{m'}}$. Remark that $m'$ divides $m$. We are trying to find the integers $M'$ such that $e^{\frac{2i\pi}{m'M'}}$ and $e^{\frac{2i\pi}{m}}$ generate $\mathrm{Im}(\rho) = \langle e^{\frac{2i\pi}{mM}} \rangle$. This is equivalent to find the integers $M'$ such that 
 $ \mathrm{lcm}\big(m, m'M'\big) = mM$. 
 Since $m'\lvert m$,  this is equivalent to determine the positive  integers $M'$
verifying
$$ \mathrm{lcm}\left(M', \frac{m}{m'}\right) = M \frac{m}{m'}\, .  $$

The list of solutions to the preceding relation viewed as an equation in $M'$,  provides the list of leaves  $\mathscr{F}_{\theta'}(M')$, associated to the collision between 
$p_k$ and $p_l$  which appear as $C$-strata of $\F$. 
 For any such $C$-strata, it comes from Section \ref{coneangle} that the {conifold angle} around it is  $\theta'=\theta_k-\theta_l-2\pi$.
 \sk 
 
Remark that,  as  in the genus 0 case considered by Thurston, 
the complex hyperbolic conifold angle $\theta'$  coincides with the new  cone angle of  the flat surfaces  whose isomorphism classes belong to the considered stratum.

\subsection{\bf Finding the $\boldsymbol{K}$-strata.} 
A $K$-stratum  appears in codimension 1 in the metric completion of a leaf $\F$ if and only if $n=3$   and $G_{\theta} = \mathrm{Im}(\rho)$.  In this case, the conifold angle around such a stratum is $\pi$ according to the third point of  Section \ref{coneangle}.

\subsection{The $1$-dimensional case.}
\label{1dimension}

We now consider the case when $n=2$. 
We  assume $\theta = ( \theta_1, \theta_2)\in 2\pi \mathbb{Q}^2$ with  $\theta_1 + \theta_2 = 4\pi$. A leaf $\mathscr{F}_{\theta}(M)$ is a $1$-dimensional complex hyperbolic manifold or equivalently, a real hyperbolic surface.

 According to Theorem \ref{cone-manifold}, the metric completion  of $\mathscr{F}_{\theta}(M)$ is a  hyperbolic surface of finite volume, with a finite number of cone points and a finite number of cusps. We give in this section a refinement of the description of the $P$-strata appearing in 
 $\overline{\mathscr{F}_{\theta}(M)}$ 
 (there is actually no $C$-stratum in the metric completion of $\mathscr{F}_{\theta}(M)$  when $n=2$ according to Proposition \ref{collisions2}) and give a list of the cusps by geometric means. We finally give explicit details in the case when $\theta = (3\pi, \pi)$.  \sk
 
  A leaf $\mathscr{F}_{\theta}(M)$ when $n=2$ shall be thought of as a generalisation of the modular surface $\mathbb{H}/\mathrm{PSL}(2,\mathbb{Z})$ which is the moduli space of regular flat tori. Veech's hyperbolic structure matches its standard one. It is not very surprising that the analytical analysis carried on in \cite
 {GhazouaniPirio1} shows that the connected components of $\mathscr{F}_{\theta}(M)$ are conformally equivalent to modular curves of the form 
 $Y_1(N)=\mathbb{H}/ \Gamma_1(N)$   for certain  integers  $N\geq 2$, see \cite[\S4.2.4]{GhazouaniPirio1} for more details.
\mk

\paragraph*{\bf Cone points.}
Strata of $\mathscr{F}_{\theta}(M)$ correspond to flat spheres  $S$ with 3 cone points
whose associated angle datum 
 $(\theta_1', \theta''_1, \theta_2)$ is such that $\theta_1' + \theta''_1 = \theta_1 - 2\pi$. Listing such flat spheres has been done in the preceding section. We also saw that $P$-strata are finite covers of moduli spaces of flat spheres, in this specific case such a finite cover is a union of points. There are as many copies  of $S$ appearing in the  metric completion $\overline{\mathscr{F}_{\theta}(M)}$ as ways of performing Devil's surgery on $S$.  \sk 

 Let $T \in \mathscr{F}_{\theta}(M)$ be a torus built by Devil's surgery on $S$. Let $\widehat \beta$ be a simple curve on $T$ avoiding the singular points that intersects the systole only once (see Section \ref{devil} for more details and pictures). We remark that $\rho(\widehat \beta)$ has to be such that $\big\langle e^{i\theta_1'}, e^{i\theta_1''}, e^{i\theta_2}, \rho(\widehat \beta) \big\rangle =  \mathrm{Im}(\rho)$.\footnote{Note that here and in what follows, the unitary holonomy character $\rho$ corresponds to the one which was denoted by $\widehat \rho$ in \S\ref{devil}.}
 \sk 
 
  With the notation of the previous section, if 
   $$\theta_1' =  \frac{2\pi r'}{mM}
   \qquad \mbox{ and } \qquad 
   \theta_1'' =  \frac{2\pi r''}{mM},$$  then there are exactly $\gcd(r',r'')$ different ways to perform the Devil's surgery in such a way that that the holonomy along $\widehat \beta$ belongs to $\mathrm{Im}(\rho)$. Amongst these ways, only $\varphi\big(\gcd(r',r'',mM)\big)$ (where $\varphi$ is Euler's totient function) are such that  $\langle e^{i\theta_1'}, e^{i\theta_1''}, e^{i\theta_2} \rangle =  \mathrm{Im}(\rho)$ and this number is the exact number of times that $S$ appears in the metric completion of $\mathscr{F}_{\theta}(M)= \F$. \sk
   
We now explain how to perform this counting. As explained in Subsection \ref{devil} (to which we refer for the notations used below), the parameters for Devil's surgery are  pairs of points $(q',q'')$ on two circles (respectively identified with) $\mathbb{R}/\theta_1'\mathbb{Z}$ and $\mathbb{R}/\theta_1''\mathbb{Z}$,  around the two singular points of the flat sphere involved in the surgery and a positive parameter $r$ equal to the radius of the aforementioned circle (up to constants depending on $\theta_1'$ and $\theta_1''$). These points $q'$ and $q''$ are the points which are going to be identified together  to create  a new point of negative curvature.  
Remark that all such parameters do not give rise to surfaces necessary belonging to $\mathscr F_{[\rho]}$: for this to happen, $q'$ and $q''$ have to move along $\mathbb{R}/\theta_1'\mathbb{Z}$ and $\mathbb{R}/\theta_1''\mathbb{Z}$ by the same amount. This way the holonomy along  $\widehat \beta$ (of the flat structure of the surfaces obtained after surgery) remains constant. The number of ways of performing Devil's surgery is the number of connected components of such parameters 
$(q',q'')$
 giving rise to flat surfaces belonging to the same  $\mathscr F_{[\rho]}$. Remark that any point 
 $q'$ in $\mathbb{R}/\theta_1'\mathbb{Z}$ belongs to a pair 
 $(q',q'')$ for which the associated surface belongs to such a chosen connected component of parameters (if this set is non-empty in the first place). In order to count the number of ways to invert Devil's surgery, it suffices then to count, given a point 
 $q'\in \mathbb{R}/\theta_1'\mathbb{Z}$, the number of points $
q'' \in  \mathbb{R}/\theta_1''\mathbb{Z}$ such that the couple 
 $(q',q'')$ gives rise to an element of $\mathscr F_{[\rho]}$ and distinguish those belonging to different components of parameters. \sk	
 
 Chose any point $ q''_0$ such that 
 $ (q',q''_0)$ 
 gives rise to an element of $\mathscr{F}_{\theta}(M)$. 
    Denote by $\widehat\beta(q',q'')$ the curve $\widehat\beta{\,}$\footnote{`The' curve $\widehat \beta$ is a priori not well-defined. One can make an arbitrary choice for $(q',q''_0)$ and then define $\widehat \beta$ for every nearby parameter using the Gauss-Manin connection. The initial choice does not really matter because the computation of its holonomy would give results well determined up to multiplication by an element of $\big\langle e^{i\theta_1'}, e^{i\theta_1''}\big\rangle$.} on the torus associated to the pair $(q',q'')$. 
    Up to multiplication by a power of $e^{i\theta_1''}$, we have
$$\rho\big(\widehat\beta(q', q''_0 + \vartheta)\big) = \rho\big(\widehat \beta(q', q''_0) \big) \cdot e^{i\vartheta}$$
 for any $\vartheta \in  \mathbb{R}/\theta_1''\mathbb{Z}$ 
(the addition refering to the standard group law of $\mathbb{R}/\theta_1''\mathbb{Z}$).
\sk

 We find this way that a pair 
  $(q',q'')$ gives rise to a surface such that $\rho(\widehat \beta) \in \langle e^{\frac{2i\pi}{mM}}\rangle$ (such a pair of parameters is said to be {\bf admissible}) if and only if $q'' - q''_0 \in \frac{2i\pi}{mM}\mathbb{Z}$ which gives exactly $r''$  possibilities 
for a point $q''$ to pair with $q'$ and giving rise to an element of $\F$  (recall that $\theta''_1 = \frac{2\pi r''}{mM}$).  Amongst these possibilities, are in the same components of the set of admissible parameters  those differing by a multiple of $r'$ modulo $r''$. This gives at most $\gcd(r',r'')$ different components of admissible parameters (see Remark \ref{different}). \sk

 The last thing that we have to ensure to make sure that the parameters we are considering truly correspond to elements of $\mathscr{F}_{\theta}(M)= \F$ is that  
 $$\big\langle e^{i\theta_1'}, e^{i\theta_1''}, e^{i\theta_2}, \rho(\widehat \beta) \big\rangle =  \big\langle e^{\frac{2i\pi}{mM}}\big\rangle\; .$$
 
  If $\beta$ is such that 
 $\rho( \widehat \beta) = e^{\frac{2ik\pi}{mM'}}$ for a certain positive integer $k$,   then
 $$\left\langle e^{i\theta_1'}, e^{i\theta_1''}, e^{i\theta_2}, \rho\big(\widehat \beta\big) \right\rangle =
   \left\langle e^{\frac{2i\pi \gcd(r',r,k,mM)}{mM}}\right\rangle$$
  which leads to $\varphi\big(\gcd(r',r,k,mM)\big)$ essentially different ways to perform the surgery.
  Finally, the $\mathbb C\mathbb H^1$-cone angle around such a point is $2\pi{\mathrm{lcm}(r',r'')}/({mM})$ (it is a direct application of the description of Section \ref{devil}).
\bk

\paragraph*{\bf Cusps}
\label{Cusps1}

According to Subsection \ref{bswec}, the  cusps of $\mathscr{F}_{\theta}(M)$ are in one-to-one correspondence with pairs $(S,\gamma)$ where

\begin{enumerate}

\item $S$ is a flat sphere with three cone points of angles $\theta_1'$,  $\theta_1''$ and $\theta_2$,  such that $\theta_1' + \theta_1'' = \theta_1 - 2\pi$ and $\mathrm{Im}(\rho) = \big\langle e^{i\theta_1'},  e^{i\theta_1''},  e^{i\theta_2} \big\rangle$; \sk

\item $\gamma$ is a regular geodesic in $S$  between the cone point of angle $\theta_1'$ and the one of angle $\theta_1''$.

\end{enumerate}

Such a geodesic always exists and is unique, we are therefore reduced to count the number of flat spheres with three cone points such that $\mathrm{Im}(\rho) = \big\langle e^{i\theta_1'},  e^{i\theta_1''},  e^{i\theta_2} \big\rangle$. This reduces to counting the number of pairs of positive integers  $(r',r'')$ such that $ r' + r'' = pM$ and $\gcd(r',r'') = 1$.  

\subsection{An example : $\theta = (3\pi, \pi)$}

We are now going to compute the number of conical points and cusps of $\mathscr{F}_{\theta}(M)$ in the special case when $\theta = (3\pi, \pi)$. In this case $p=1$ and $m=2$.

\begin{itemize}

\item Each $\mathbb C\mathbb H^1$-cone point of $\mathscr{F}_{\theta}(M)=\mathscr{F}_{ (3\pi, \pi)}(M)$ corresponds to a partition $r + s = M$ with $r, s > 0$. To 
such a partition are associated 
$\varphi\big(\gcd(r,s) = \gcd(r,M)\big)$ cone points, all of the same cone angle $2\pi {\mathrm{lcm}(r, M-r)}/{M}$.   
\sk

A particular case is when $M=2M'$ is even. In this case the partition $M = M' + M'$ only gives rise to half of the cone points predicted by the above paragraph, namely $\frac{1}{2}\varphi(M')$. Indeed the underlying sphere on which the surgery is performed has angles $(\pi,{\pi}/{2},{\pi}/{2}) $ and has a symmetry of order two. This symmetry permutes the different inversions of the surgery, except for the case $M=4$ where there is only $\varphi(2) = 1$ way to perform the surgery and in this particular case it only halves the cone angle at the underlying $\mathbb C\mathbb H^1$-cone point.
\sk
\item  In particular the number of cone points is  
$$ \frac{1}{2}\sum_{r=1}^{2M-1}{\varphi\big(\gcd(r,M)\big)} + \frac{1}{2}\varphi(M) $$ for $M > 4$ and is $2$ when $M=3,4$.  
\sk 
\item There are as many cusps as proper partitions $ M=r+s$  such that $r>0$ and $M$ are coprime. Thus the number of cusps is exactly 
$$\frac{1}{2}\varphi(M)\, $$ 
for $M \geq 3$ and is equal to $1$ for $M=2$.  
\end{itemize}\sk

The number of cusps and cone points put together gives us the number of punctures of $\mathscr{F}_{(3\pi, \pi)}(M)$ which is  equal to 
\begin{equation}
\label{E:sum}
 \frac{1}{2}\sum_{r=1}^{M}{\varphi\big(\gcd(r,M)\big)}  
 \end{equation}
for $M > 4$,  to $2$ for $M=2$ and to $3$ for $M=3,4$. When $M > 4$, a reordering of the sum 
 \eqref{E:sum}
 gives us  that the total number of punctures of $\mathscr{F}_{(3\pi, \pi)}(M)$ is actually
$$\frac{1}{2}\sum_{d | M}{\varphi(d)\varphi\big({M}/{d}\big)}\, . $$  
This number is equal to the number of cusps of the modular curve $Y_1(M) = \mathbb{H}/ \Gamma_1(M)$ (see \cite{DiamondShurman}). This is not a coincidence: it is proved in \cite[\S4.2.5.3]{GhazouaniPirio1} that the conformal type of $\mathscr{F}_{(3\pi, \pi)}(M)$ is actually the same as the one of $Y_1(M)$.



\section{\bf HOLONOMY OF THE $\mathbb C\mathbb H^{n-1}$-STRUCTURE: DISCRETENESS}

\label{holonomy}
\subsection{Previous results in the genus 0 case.}

Thurston proves in \cite{Thurston} (recovering by geometric methods results of Deligne and Mostow from \cite{DeligneMostow}) that when $g=0$ and $n\geq 4$, if the angle datum $\theta = (\theta_1, \ldots, \theta_n) \in ]0,2\pi[^n $ verifies 
\begin{equation}
\label{E:11(INT)}
\tag{\rm INT}
\forall i,j =1,\ldots, 
\,  i \neq j,   \quad 
2\pi< \theta_i + \theta_j  \;\;
\Longrightarrow \;\; 
\theta_i + \theta_j - 2\pi \ \mbox{ divides } \ 2\pi \,, 
\end{equation}
then the metric completion of $\mathscr{F}_{\!\theta} \simeq \mathscr{M}_{0,n}$ (the unique leaf of Veech's foliation in this case) is a connected complex hyperbolic orbifold of finite volume and therefore a quotient ${\mathbb{CH}}^{n-3}/ \Gamma_{\!\theta}$ where $\Gamma_{\!\theta}$ is a lattice in $\mathrm{PU}(1,n-3)$. This lattice $\Gamma_{\!\theta}$ is exactly the image of holonomy morphism 
$$ \mathrm{hol} : \pi_1 \big(\mathscr{M}_{0,n}\big) \longrightarrow  \mathrm{PU}(1,n-3) $$ 
of the $({\mathbb{CH}}^{n-3}, \mathrm{PU}(1,n-3))$-structure of $\mathscr{F}_{\!\theta}$. Sometimes it happens that the image $\Gamma_{\!\theta}=\mathrm{hol}(\pi_1( \mathscr{F}_{\!\theta}) )$ of the holonomy   is a lattice in $\mathrm{PU}(1,n-3) $ even if the metric completion of $\mathscr{F}_{\!\theta}$ is not an orbifold. The combined works of Picard, LeVavasseur, Terada, Deligne-Mostow, Mostow, Thurston and Sauter (see \cite{LeVavasseur,Terada,Thurston,DeligneMostow,Mostow,Sauter}) lead to the following results:
\sk 
\begin{enumerate}
\item there exist 94 angles data $\theta$ for which $\mathscr{F}_{\!\theta}$ is an orbifold, and therefore $\Gamma_{\!\theta}=\mathrm{hol}(\pi_1 (\mathscr{F}_{\!\theta}))$ is a lattice; 
\mk 
\item this builds lattices in $\mathrm{PU}(1,N) $ for all $N=n- 3  = 1, \ldots, 9$, some of them being non-arithmetic for $N =1, 2$ and $3$ (these have been for a long time the only known examples of non-arithmetic complex hyperbolic lattices until the recent work of Deraux,  Parker and Paupert \cite{DerauxParkerPaupert});
\mk 
\item if $N \geq 3$,  $\Gamma_{\!\theta}$ is a lattice if and only if $\theta$ verifies $(\Sigma\mathrm{INT})$, apart from one exception (we recall that $(\Sigma\mathrm{INT})$ is the  refinement of  $\mathrm{(INT)}$  stated in the Introduction above; see also 
 \cite[\S1]{MostowIHES} or  \cite[Theorem 0.2]{Thurston});
\mk 
\item when $N=2$,  there exist $9$ angles data failing $(\Sigma\mathrm{INT})$ for which  $\Gamma_{\!\theta}=\mathrm{hol}(\pi_1 (\mathscr{F}_{\!\theta}))$ is a lattice.

\end{enumerate}

\subsection{Genus $\boldsymbol{1}$ and $\boldsymbol{n = 3}$.}

We now  address the following question, which must seem natural at this point:

\begin{quest}
Let $T$ be a torus with three marked points, and $\theta$ an admissible rational angle datum. 
Does there exist $\rho \in \mathrm{H}^1(T, \mathbb{U}, \theta) $ such that $\mathrm{hol}(\pi_1( \mathscr{F}_{[\rho]})) $ is a lattice in $\mathrm{PU}(1,2)$ ? 
\end{quest}

As we have seen in Section \ref{listing}, such a leaf $\F$ is only determined by $\theta$ and an integer $M$. We denote such a leaf by $\mathscr{F}_{\theta}(M)$. The first difficulty to address is the question of the connectedness of the leaves $\F$: $\mathscr{F}_{\theta}(M)$ may have several connected components (see Subsection \ref{invariance} for a short discussion of this matter) and it is possible that the holonomy of one of these is a lattice and that it is not the case for the others. \sk 

 The following lemma, whose proof in the genus $0$ case can be found in \cite{Mostow}, outlines a strategy to search for connected components of $\mathscr{F}_{\theta}(M)$ whose holonomy is a lattice: 
\begin{lemma}
\label{criterion1}
Let\, $\mathscr{F}$ be a connected component of \, $\mathscr{F}_{\theta}(M)$ whose complex hyperbolic holonomy is a lattice in $\mathrm{PU}(1,2)$. Then the complex hyperbolic holonomy of every codimension $1$ stratum is a lattice in $\mathrm{PU}(1,1)$. 
\end{lemma}

This lemma provides necessary conditions on such a connected component $\mathscr{F}$ to have discrete holonomy in $\mathrm{PU}(1,2)$, conditions which hold true in several cases.  To find candidates to have holonomy a lattice in $\mathrm{PU}(1,2)$, an optimistic  strategy goes as follows:

\begin{enumerate}

\item identify the different connected components of $\mathscr{F}_{\theta}(M)$; \sk

\item verify if the criterion given by Lemma \ref{criterion1} is verified, using the list  of Deligne-Mostow and Thurston for genus $0$ type codimension $1$ strata or the strategy suggested in the next paragraph for genus $1$ codimension $1$ strata; \sk


\item amongst the isolated candidates, compute the complex hyperbolic holonomy and verify that it is discrete.
\mk 
\end{enumerate}
(We call  `{\bf genus $\boldsymbol{g}$ strata}'   a strata whose elements are flat surfaces of genus $g$). 
\mk

The last step seems to be the most difficult to achieve so far, since the methods used in the genus $0$ case (see for example \cite{Parker})  tends to become algorithmically too complicated in our case and strongly rely on the knowledge of simple generators of the fundamental group of $\mathscr{M}_{0,n}$ (a finite family of distinguished Dehn twists). 


\subsection{Some cases when the holonomy is an arithmetic lattice.}
We remark  that for a certain number of connected components $\mathscr{F}$ of leaves of Veech's foliation, the holonomy is an arithmetic lattice in $\mathrm{PU}(1,2)$. 

This follows from the following lemma: 

\begin{lemma}
Let $\mathscr{F}$ be a connected component of a leaf $\F$. Then, up to a suitable conjugation, the coefficients of the matrices of\;$\mathrm{hol}(\pi_1 (\mathscr{F}))$ lie in $\mathbb{Z}[\mathrm{Im}(\rho)]$.
\end{lemma}

This lemma is an easy consequence of the fact that the matrices of the transitions maps of an atlas of linear parametrisations coming from topological polygonations must have coefficients in $\mathbb{Z}[\mathrm{Im}(\rho)]$ (see Section \ref{Charts} and the proof of Proposition \ref{linear}).   In particular if $\mathbb{Z}[\mathrm{Im}(\rho)]$ is discrete in $\mathbb{C}$, then for any connected component $\mathscr{F}$ of $\F$, the image of its holonomy is discrete in $\mathrm{PU}(1,2)$. \sk 

\vspace{-0.5cm}
The 
  developing map  $\widetilde{\mathscr{F}} \rightarrow \mathbb{CH}^2$ 
 of such a $\mathscr F$  factors through a local isometry $\mathscr{F} \rightarrow \mathbb{CH}^2 / \mathrm{hol}(\pi_1(\mathscr{F}))$. Since $\mathscr{F}$ has finite volume ({\it cf.} Section 8), $\mathrm{hol}(\pi_1(\mathscr{F}))$ is necessarily a lattice which must be arithmetic since it belongs to $\mathrm{PU}(1,2) \cap \mathrm{SL}_3(\mathbb{Z}[\mathrm{Im}(\rho)])$.
This situation actually happens: if $\mathrm{Im}(\rho) = \langle \exp({{2i\pi}/{m}}) \rangle$ for $m = 3, 4$ or $6$ then $\mathbb{Z}[\mathrm{Im}(\rho)]$ is discrete. Note that the argument does apply to higher dimensions as well. 
\sk 

 By straightforward computations, these cases can be completely determined: 
\begin{prop} \begin{enumerate}
\item  The data $\theta$ and $M$ such that $\mathrm{Im}(\rho) = \langle \exp({{2i\pi}/{m}}) \rangle$ with  $m \in \{ 3, 4,6\}$ are exactly those given  in {\sc Table \ref{Tata}}  below.  \sk
\item  For any such  $\theta$ and $M$, the (image of the)  complex hyperbolic  holonomy of 
any connected component  of\;\,$\mathscr F_\theta(M)$ is an arithmetic lattice.
 \end{enumerate}
\end{prop}

\begin{table}[h!]
\begin{tabular}{|c|c|c|c|c|}
\hline 
  $\boldsymbol{n}$  & 
\begin{tabular}{c}  \vspace{-0.25cm}\\
 $\boldsymbol{m \, \theta/2\pi=\big(m \, \theta_i/2\pi\big)_{i=1}^n}$ \vspace{0.15cm}
 \end{tabular}
  & $\boldsymbol{m}$ 
&  $\boldsymbol{M}$  & {\bf label}
  \\ \hline \hline 
      \multirow{9}{*}{3}        & 
 \begin{tabular}{c}  \vspace{-0.35cm}\\
   $\big(  5 \, , \,  2  \, , \,   2 \big)$    \vspace{0.1cm}
 \end{tabular}
  &     3  &    1   &  $a$       \\ \cline{2-5}
     &     \begin{tabular}{c}  \vspace{-0.35cm}\\
   $\big(   6  \, , \, 3  \, , \,  3\big)$    \vspace{0.1cm}
 \end{tabular}   &     4  &     1  &    $b$    \\ \cline{2-5}
       &    \begin{tabular}{c}  \vspace{-0.35cm}\\
   $\big(   {8}  \, , \,  {5}  \, , \,   5\big)$    \vspace{0.1cm}
 \end{tabular}    &     6     &   1    &     $c$   \\ \cline{2-5}
          &   
            \begin{tabular}{c}  \vspace{-0.35cm}\\
 $\big(   {7}  \, , \,  {3}  \, , \,   2 \big)$    \vspace{0.1cm}
 \end{tabular}
              &    4    &  1      &      $d$    \\ \cline{2-5}
                      &    \begin{tabular}{c}  \vspace{-0.35cm}\\
 $\big(   {9}  \, , \,  {5}  \, , \,   4 \big)$    \vspace{0.1cm}
 \end{tabular}     &     6   &     1   &  $e$      \\ \cline{2-5}
                           &    \begin{tabular}{c}  \vspace{-0.35cm}\\
 $\big(   {10}  \, , \,  {5}  \, , \,   3 \big)$    \vspace{0.1cm}
 \end{tabular}    &     6   &  1      &  $f$      \\ \cline{2-5}
         &       \begin{tabular}{c}  \vspace{-0.35cm}\\
 $\big(   {11}  \, , \,  {5}  \, , \,   2 \big)$    \vspace{0.1cm}
 \end{tabular}     &     6   & 1 &  $g$      \\ \cline{2-5}
         &  \begin{tabular}{c}  \vspace{-0.35cm}\\
 $\big(   {11}  \, , \,  {4}  \, , \,   3 \big)$    \vspace{0.1cm}
 \end{tabular}       &  6      &  1     &   $h$      \\  \cline{2-5}
              &     \begin{tabular}{c}  \vspace{-0.35cm}\\
 $\big(   {5}  \, , \,  {2}  \, , \,   2 \big)$    \vspace{0.1cm}
 \end{tabular}    &     3   &     2  &    $i$    \\ \hline  \hline 
   \multirow{4}{*}{4}  
     & 
    \begin{tabular}{c}  \vspace{-0.35cm}\\
 $\big(   {7}  \, , \,  {3}  \, , \,   3   \, , \,   3  \big)$    \vspace{0.1cm}
 \end{tabular}     &   4     &     1   &      $j$     \\ \cline{2-5}
       &      \begin{tabular}{c}  \vspace{-0.35cm}\\
 $\big(   {9}  \, , \,  {5}  \, , \,   5   \, , \,   5  \big)$    \vspace{0.1cm}
 \end{tabular}             &6 &   1      &      $k$   \\ \cline{2-5}
                 &     \begin{tabular}{c}  \vspace{-0.35cm}\\
 $\big(   {10}  \, , \,  {5}  \, , \,   5   \, , \,   4  \big)$    \vspace{0.1cm}
 \end{tabular}     &  6      &  1     &    $l$    \\ \cline{2-5} 
 &   
 \begin{tabular}{c}  \vspace{-0.35cm}\\
 $\big(   {11}  \, , \,  {5}  \, , \,   4   \, , \,   4  \big)$    \vspace{0.1cm}
 \end{tabular}       &    6    &    1   &    $m$    \\ \hline  \hline 
 \multirow{2}{*}{5} 
                           &   
    \begin{tabular}{c}  \vspace{-0.35cm}\\ 
  $\big(   10  \, , \,    5  \, , \,   5  \, , \,  5 \, , \,  5 \,     \big)$    \vspace{0.1cm}
 \end{tabular}    &    6    &  1     &     $n$   \\ \cline{2-5} 
         &\begin{tabular}{c}  \vspace{-0.35cm}\\ 
         $\big(   11  \, , \,  
  5  \, , \,  
  5  \, , \,  
    5 \, , \,  
      4 \,     \big)$  \vspace{0.1cm}
 \end{tabular}       &   6     &  1       &   $o$     \\ \hline  \hline
6                   &   \begin{tabular}{c}  \vspace{-0.35cm}\\
   $\big(   11  \, , \,  
  5  \, , \,  
  5  \, , \,  
    5 \, , \,  
      5 \, , \,  
   5 \,     \big)$    \vspace{0.1cm}
 \end{tabular}      &    6         & 1      &     $p$   \\ \hline 
\end{tabular}\bk
\caption{Data $\theta$ and $M$ such that $\mathrm{Im}(\rho) = \langle \exp({{2i\pi}/{m}}) \rangle$ with  $m \in \{ 3, 4,6\}$ ($n$ stands for the number of cone points and coincides with the dimension of the leaf  $\mathscr F_\theta(M)$ plus one).}
\label{Tata}
\vspace{-0.5cm}
\end{table}

 We think it is relevant  to make the following remarks  regarding {\sc Table \ref{Tata}} : 
\begin{rem}
{\rm 
\begin{enumerate}
\item  using the same approach as the one considered in  \cite[\S 4.3.1]{GhazouaniPirio1}, it is not difficult to establish that 
$\mathscr F_\theta(M)$  is  connected  for any one of the elements $(\theta,M)$ of  
 {\sc Table \ref{Tata}}. 
\sk  
\item From the preceding remark, it follows that to any label  $\ell$ among the sixteen of {\sc Table \ref{Tata}}  corresponds precisely one arithmetic complex hyperbolic lattice in ${\rm PU}(1,n-1)$,  which will be denoted by $\Gamma_{\!\ell}$.  
\sk
  \item  The lattices $\Gamma_{\!\ell}$  associated to the first three labels $a,b$ and $c$ 
  were previously known. 
   Indeed, for such a label, the associated angle datum $\theta(\ell)=(\theta_i(\ell))_{i=1}^3$ is such that $\theta_2(\ell)=\theta_3(\ell)$. From this and because  $M=1$, one deduces that 
the flat structure of a flat tori $T$ whose class belongs to    $\mathscr F_{\theta(\ell)}(1)$ 
    is invariant by the elliptic involution $i \circlearrowright T$. 
   Consequently, the flat structure of $T$ comes from a flat structure on $T/i\simeq \mathbb P^1$ with five cone points. It follows  that there exists a (possibly orbifold) covering $\mathscr F_{\theta(\ell)}(1)
   \rightarrow \mathscr M_{0,\tilde \theta(\ell)}$ with 
   $ \tilde \theta(\ell)=\big(\theta_1(\ell)/2, \theta_2(\ell), \pi , \pi,\pi\big)$ (see \cite[\S4.2.5]{GhazouaniPirio1} for more details).\sk

  This eventually gives us that $\Gamma_{\!\ell}$ coincides with a  Picard/Deligne-Mos\-tow lattice $\Gamma_{\mu(\ell)}$ associated to a $5$-tuple $\mu(\ell)$ (we use here the notations of \cite{DeligneMostow} and \cite{Mostow}) which depends only on $\ell$, see 
   {\sc Table {\ref{tato}}} below. Note that it shows  that $\Gamma_{\!a}$ and $\Gamma_{\!c}$ coincide (up to conjugacy).
    \sk
      \begin{table}[h!]
\begin{tabular}{|c|c|c|c|}
\hline 
  $\boldsymbol{\ell}$  & 
\begin{tabular}{c}  \vspace{-0.25cm}\\
 $\boldsymbol{\theta(\ell)}$ \vspace{0.15cm}
 \end{tabular}
  & $\boldsymbol{\mu(\ell)}$ 
&  \begin{tabular}{c}  \vspace{-0.35cm}\\ 
{\bf label in [DM86, {p.\,86}]} \end{tabular}
  \\ \hline \hline 
$ a$                   &   \begin{tabular}{c}  \vspace{-0.35cm}\\
   $\big(  \frac{10\pi}{3}\, , \,  \frac{4\pi}{3}\, , \,   \frac{4\pi}{3} \big)$    \vspace{0.1cm}
 \end{tabular}      &    $ \big( \frac{1}{6}\, , \, \frac{1}{3}\, , \, \frac{1}{2}\, , \, \frac{1}{2}\, , \, \frac{1}{2}  \big)$         & 6  
 \\ \hline 
 $b$                  &   \begin{tabular}{c}  \vspace{-0.35cm}\\
   $\big(   {3\pi}\, , \,  \frac{3\pi}{2}\, , \,   \frac{3\pi}{2} \big)$    \vspace{0.1cm}
 \end{tabular}      &  $ \big( \frac{1}{4}\, , \, \frac{1}{4}\, , \, \frac{1}{2}\, , \, \frac{1}{2}\, , \, \frac{1}{2}  \big)$         & 2     
 \\ \hline 
$c$                   &   \begin{tabular}{c}  \vspace{-0.35cm}\\
   $\big(   \frac{8\pi}{3}\, , \,  \frac{5\pi}{3}\, , \,   \frac{5\pi}{3}  \big)$    \vspace{0.1cm}
 \end{tabular}      &    $ \big( \frac{1}{3}\, , \, \frac{1}{6}\, , \, \frac{1}{2}\, , \, \frac{1}{2}\, , \, \frac{1}{2}  \big)$         & 6 
 \\ \hline 
\end{tabular}\bk
\caption{}
\label{tato}
\end{table} 
   \item Excluding the cases associated to the labels $a,b$ and $c$, one gets thirteen {\it a priori} new arithmetic complex hyperbolic lattices. It would be interesting to know more about them.   \sk 
   
   To achieve this goal, the following approach, albeit computational, could be fruitful :  for any label $\ell\in \{a,b,\ldots,p\}$, using the results  of \cite[\S4.2.4]{GhazouaniPirio1},  it is possible to give an explicit basis of the fundamental group of the corresponding leaf $\mathscr F_{\!\ell}=\mathscr F_{\theta(\ell)}(M(\ell))$. 
   Then, combining the results of \cite[\S4.4]{GhazouaniPirio1} with the monodromy formulae  of 
   \cite[\S6]{Mano}, it should be possible to construct an  explicit version
   $h_\ell : \pi_1(\mathscr F_{\!\ell})\rightarrow {\rm PU}(1,n-1)$ 
    of the holonomy map 
     which would  allow to study  $\Gamma_{\!\ell}={\rm Im}(h_\ell)$ quite concretely.  
   \end{enumerate} 
}
\end{rem}

\subsection{Holonomy in the $1$-dimensional case.}

Subsection \ref{1dimension} is the first step towards a geometric description of the moduli spaces $\mathscr{F}_{[\rho]}$ when $g = 1$ and $n=2$. A comprehensive description of this case is carried on in the paper \cite{GhazouaniPirio1}.

 The following Proposition, which is just a suitable reformulation of 
 Poincar\'e's theorem on fundamental domains of Fuchsian groups, gives an easily verifiable sufficient criterion for a (connected component of a) leaf $\mathscr{F}$ to have discrete holonomy in $\mathrm{PU}(1,1)$.

\begin{prop}
The metric completion\, $\overline{\mathscr{F}} $ of \,$\mathscr{F}$ is a lattice quotient of \,${\mathbb{CH}}^1$ if and only if all the cone angles at points of\, $\overline{\mathscr{F}} \setminus \mathscr{F}$ are integer parts of \,$2\pi$. 
\end{prop}

This Proposition combined with the analysis carried on in \cite{GhazouaniPirio1} allows us to find several such $\mathscr{F}$ which are lattice quotients ({\it cf.}  \cite[\S6.1]{GhazouaniPirio1} for more details).



\appendix 

\section{\bf Complex hyperbolic geometry.}
\label{CHG}

\subsection{Complex hyperbolic space.}
On the complex vector space $\mathbb{C}^{n+1}$ of dimension $n+1$, 
we consider the  Hermitian form $ \langle \cdot , \cdot \rangle $ of signature $(1,n) $ defined by $$ \langle z, w \rangle =  z_0 \overline{w_0} - \sum_{i=1}^n{z_i\overline{w_i}} $$
for $z=(z_0,\ldots,z_n)$  and $w=(w_0,\ldots,w_n)$ in $ \mathbb C^{n+1}$.\sk

All the definitions to come do not depend on the choice of the Hermitian metric of signature $(1,n)$ since two such forms are linearly conjugate.  Recall that $\mathbb{CP}^n$ is the set of complex lines in  $\mathbb{C}^{n+1}$. We define $\mathbb{C}\mathbb H^n$,  the \textbf{complex hyperbolic space of dimension} $\boldsymbol{n}$,  to be the subset of $\mathbb{CP}^n$ formed by the  lines in $\mathbb C^{n+1}$ on which $ \langle \cdot , \cdot \rangle$ is positive: 
$$
\mathbb{C}\mathbb H^n=\Big\{ \, [z]\in \mathbb{CP}^n\, \big\lvert \, z\in \mathbb C^{n+1}, \;  \langle z , z \rangle >0 
\Big\}\, .
$$

 We denote by $\mathrm{PU}(1,n)$ the set of linear automorphisms of $\mathbb{C}^{n+1}$ which preserve $\langle \cdot, \cdot \rangle$.  It acts  projectively  on $\mathbb{C}\mathbb H^n$ and satisfies the following properties:

\begin{itemize}

\item its action on $\mathbb{C}\mathbb H^n$ is  transitive; \sk
\item $\mathrm{PU}(1,n)$ is  exactly the group ${\rm Aut}(\mathbb{C}\mathbb H^n)$ of biholomorphisms of $\mathbb{C}\mathbb H^n$; \sk
\item there exists a Riemannian metric on $\mathbb{C}\mathbb H^n$ (unique up to rescaling), for which $\mathrm{PU}(1,n)$ is exactly the set of holomorphic isometries. This metric is called the \textbf{complex hyperbolic metric};  \sk
\item this metric has sectional curvature comprised between $-\frac{1}{4}$ and $-1$. 
Its holomorphic sectional curvature is constant. 
\end{itemize}

The stabiliser of a point in $\mathbb{C}\mathbb H^n$ (which is exactly the stabiliser of a positive line $\mathrm{PU}(1,n)$) is conjugate to $\mathrm{U}(n) \subset \mathrm{PU}(1,n)$, which is the maximal compact subgroup of $\mathrm{PU}(1,n)$. 
The complex hyperbolic space 
$\mathbb{C}\mathbb H^n$ is therefore isometric to the rank one (Hermitian) symmetric space $\mathrm{PU}(1,n)/\mathrm{U}(n) $. It is the non-compact dual of $\mathbb C\mathbb P^n$. \mk

The distance for the complex hyperbolic metric can be explicitly computed by means of the initial Hermitian form:

\begin{lemma}
\label{formula1} 
Let $[{z}]$ and $[{w}]$ be two points in $\mathbb{C}\mathbb H^n \subset \mathbb{CP}^n$ with $z,w\in \mathbb C^{n+1}$. 
\begin{enumerate}
\item 
The complex hyperbolic distance $\alpha$ between $[{z}]$ and $[{w}]$ satisfies 
$$ \cosh^2\left(\frac{\alpha}{2}\right) = \frac{\langle {z}, {w} \rangle  \cdot \langle w,  {z} \rangle }{ \langle {z}, {z} \rangle \cdot \langle {w}, {w} \rangle  }\, .  $$ 
\item The geodesic curve linking $[z]$ to $[w]$ in $\mathbb{C}\mathbb H^n$ is 
the projectivisation of the linear segment $[z,w]= \{   z+tw\, \lvert t\in [0,1]\, \}$
linking $z$ to $w$ in $\mathbb C^{n+1}$.
\end{enumerate}
\end{lemma} 
\noindent  (For some proofs, see \cite[\S3.3.5]{Goldman})

\subsection{Coordinates.}
\subsubsection{\bf The ball model}
In order to have coordinates on  $\mathbb{C}\mathbb H^n $,  one can take affine coordinates of $\mathbb{CP}^n$.  Since $z_0\neq 0$ if $[z]=[z_0:\cdots:z_n]$ belongs to 
$ \mathbb C\mathbb H^n$, the latter is contained in the affine chart 
$
\{ 
 z_0 \neq 0 \}$ of $\mathbb C\mathbb P^n$. 
 
 In the $z_0 =1$ normalisation, it comes that 
$z_1,\ldots,z_n$ provide a global system of holomorphic coordinates which identify 
$ \mathbb C\mathbb H^n$ with the complex $n$-ball: 
$$ \bigg \{ \ \big(z_i\big)_{i=1}^n \in \mathbb{C}^n \; \big  | \; \sum_{i=1}^n{ \big|z_i\big|^2} < 1  \bigg\}\, .  $$ 

In this model of the complex hyperbolic space, the hyperbolic metric identifies with the 
Bergman metric of the complex $n$-ball. \sk 

Although we do not use it  in the present text, the complex ball is a 
very classical model for $\mathbb{C}\mathbb H^n $ which is worth being mentioned. We will not say anything more about it but one can find  a comprehensive presentation in \cite{Goldman}.

\subsubsection{\bf Pseudo-horospherical coordinates.}
More important for our purpose is a special kind of affine coordinates on 
$\mathbb{C}\mathbb H^n $ which are very close,  in spirit,  to the \textit{horospherical coordinates} introduced by Goldman and Parker  in \cite{GoldmanParker}.\sk 

Let $\xi=(\xi_0,\ldots,\xi_n)$ be a system of linear coordinates on $\mathbb C^{n+1}$ such that the expression of the Hermitian form $\langle \cdot, \cdot \rangle$ in these  can be written out 
$$ \langle {\xi}, {\xi} \rangle =  \frac{i}{2}\Big(\xi_n\overline{\xi_0} - \xi_0 \overline{\xi_n}\Big) + a\big(\widehat{\xi},\widehat{\xi}\big) $$
for  a Hermitian form $a$ of signature $(1,n-1)$ and where $\widehat{\xi}$ stands for 
$(\xi_0, \ldots, \xi_{n-1})$.

\begin{lemma}
If $\xi=(\xi_i)_{i=0}^n$ is such that $\langle {\xi}, {\xi} \rangle>0$ then $\xi_0\neq 0$. 
\end{lemma}
\begin{proof} One verifies that, up to a linear change of coordinates letting $\xi_0$ invariant, one can assume that $\langle {\xi}, {\xi} \rangle=\frac{i}{2}(\xi_n\overline{\xi_0} - \xi_0 \overline{\xi_n})+\sum_{j=0}^{n-1} \epsilon_j \xi_j\overline{\xi_j}$
 for some $\epsilon_j $ belonging to $\{-1,0,1\}$. By assumption, $a(\widehat{\xi},\widehat{\xi})=\sum_{j=0}^{n-1} \epsilon_j \xi_j\overline{\xi_j}$ has signature $(1,n-1)$ hence exactly one of the $\epsilon_j$'s  is equal to 1, all the others being equal to -1.
 
If  $\epsilon_0=-1$, then $\frac{i}{2}(\xi_n\overline{\xi_0} - \xi_0 \overline{\xi_n})-\xi_0\overline{\xi_0}$ has signature $(1,1)$.  Since $\sum_{j=1}^{n-1} \epsilon_j \xi_j\overline{\xi_j}$ has signature $(1,n-2)$ (because $\epsilon_j=1$  for some $j\geq 1$), this would imply that $\langle \cdot, \cdot\rangle$ has signature $(2,n-1)$, a contradiction. 
  \end{proof}
\mk

From the preceding lemma, it follows that the complex hyperbolic space admits a model contained in the affine chart $
\{ 
\xi_0 \neq 0 \}$ of $\mathbb C\mathbb P^n$.  
Then,  under the normalization $\xi_0 = 1$, the $\xi_k$'s for $k=1,\ldots,n$ provide global affine coordinates on this model which will be called \textbf{pseudo-horospherical coordinates}.  

In such coordinates, the associated quadratic form is given by 
$ \langle {\xi}, {\xi} \rangle =  \mathrm{Im}(\xi_n) + a(\widehat{\xi},\widehat{\xi}) $ with 
$\widehat{\xi}=(1,\xi_1,\ldots,\xi_{n-1})$  and consequently, this model  of the complex hyperbolic space $\mathbb C\mathbb H^n$ consists in    the set of $\xi=(\widehat{\xi},\xi_n) \in \mathbb C^n$ such that  
$$ \mathrm{Im}\big(\xi_n\big) >  -  a\big( \widehat{\xi}, \widehat{\xi}\big)\, .$$

In the standard (homogeneous) coordinates $z=(z_0,z_1,\ldots,z_n)$ on $\mathbb C^{n+1}$, the formula for the complex hyperbolic metric is the following 
$$ g = - \frac{4}{\langle z, z \rangle^2}  \begin{vmatrix}
\langle z, z \rangle & \langle dz, z \rangle \\
\langle z, dz \rangle & \langle dz, dz \rangle
\end{vmatrix}  \, . $$

 A straightforward calculation gives the following formula for the expression 
 of this metric in pseudo-horo\-spherical coordinates: 
$$ g = - \frac{4}{\langle \xi,\xi \rangle^2} \Big( \langle \xi,\xi \rangle \cdot a\big(d\hat{\xi}, d\hat{\xi}\big) - a\big(\hat{\xi}, d\hat{\xi}\big) \cdot a\big(d\hat{\xi}, \hat{\xi}\big) - \mathrm{Im}\Big(d\xi_n \cdot a\big(\hat{\xi}, d\hat{\xi}\big) \Big) - \big|d\xi_n\big|^2 \Big)\, .  $$

 Introducing $u = \langle \xi, \xi \rangle$ and $s = \mathrm{Re}(\xi_n)$, we therefore have $\xi_n = s + i(u - a(\hat{\xi}, \hat{\xi}) )$.  In the coordinates system $(s,u,\xi_1, \ldots, \xi_{n-1})$ on the pseudo-horospherical model of $\mathbb C\mathbb H^n$ we are considering, the metric tensor $g$  writes down 
\begin{equation}
\label{E:gInPseudoHorosphericalCoordinates}
 g = \frac{4}{u^2} \left( \frac{du^2}{4} +  \Big(\frac{ds}{2} + \mathrm{Im}(\omega)\Big)^2+ \mathrm{Re}(\omega)^2 - u\cdot \Omega \right) 
 \end{equation}
where $\omega = a(\hat{\xi}, d\hat{\xi})$ and $\Omega =a(d\hat{\xi}, d\hat{\xi})$. 
\sk 

We now introduce the family of open sets in $\mathbb{C}\mathbb H^n$: 
$$ U_{K,\lambda} = \left\{ \big[1,\xi_1, \ldots, \xi_n\big] \in \mathbb{C}\mathbb H^n \,  \Big| \,  \big|\xi_1\big|, \ldots, \big|\xi_{n-1}\big|, \big|\mathrm{Re}(\xi_n)\big| < K \, \text{ and } \,  \mathrm{Im}\big(\xi_n\big) > \lambda  \right\rbrace $$
\noindent with $K, \lambda > 0$.

\begin{lemma}
\label{cylindricalcalculation}
Let $K$ and $\lambda$ be arbitrary positive constants. 
\begin{enumerate}
\item  The complex hyperbolic volume  of  $ U_{K,\lambda}$  is finite. \sk
\item If $\gamma : [0,1] \longrightarrow  U_{K,\lambda}$ is path such that  $\gamma(t) = (\xi_1(t), \ldots, \xi_n(t))$ for any $t\in [0,1]$, then its length $L(\gamma)$ for the complex hyperbolic metric satisfies
 $$ L(\gamma) \geq  \Big|\log\big(\xi_n(1)\big) - \log\big(\xi_n(0)\big)\Big|\, .  $$
\end{enumerate}
\end{lemma}
\begin{proof}
In the coordinates system $(s,u,\xi_1, \ldots, \xi_{n-1})$ on $U_{K,\lambda}$, 
the complex hyperbolic volume element writes down 
$$\sqrt{\mathrm{det}(g)}\,  ds  du d\xi_1d\overline{\xi_1} \cdot \cdots  \cdot d\xi_{n-1}d\overline{\xi_{n-1}}.$$ 

Since both  $\omega$ and $\Omega$ depend  continuously on $\xi_1, \ldots, \xi_{n-1}$,  one gets that 
$$ \sqrt{\mathrm{det}(g)} = \frac{f\big(\xi_1, \ldots, \xi_{n-1}\big)}{u^{2n+2}} $$ 
for some positive and continuous function $f$ which thereby is bounded on $U_{K,\lambda}$. 
 The finiteness of the volume of $ U_{K,\lambda} $ follows directly from evaluating the associated integral. 
\sk 

The second point of the lemma follows directly from the fact that $g \geq u^{-2}{du^2}$ on $ U_{K,\lambda} $. To see this, one has to prove that $\Omega$ is negative. But if $\Omega$ was not, since $  {du^2}/{4} +  ({ds}/{2} + \mathrm{Im}(\omega))^2+ \mathrm{Re}(\omega)^2 $ does not depend on $u$,  
one would deduce from \eqref{E:gInPseudoHorosphericalCoordinates} that $g$  
would not be positive for large values of $u$, a contradiction. \end{proof}


\section{\bf Cone-manifolds}
\label{conemanifolds}

\subsection{Generalities}

This section strongly builds on \cite{McMullen}, in particular the use of joints for describing spherical cone-manifolds. \sk 

Let $X$ be a complete homogeneous Riemannian manifold and let  $G$ be its isometry group (or more generally a subgroup of its isometry group). We develop material on cone-manifolds in this specific case.
For any point $p\in X$, one denotes by $X_p$ the set of geodesic rays emanating from it and  $G_p=  \mathrm{Stab}_G(p)$ stands for its stabiliser. 
 
  A $\boldsymbol{(X,G)}${\bf-cone-manifold} is a geometric object built inductively as follows:
\begin{itemize}
\item if $X$ is $1$-dimensional, a $(X,G)$-cone-manifold is just a $(X,G)$-manifold;\sk
\sk 
\item otherwise, a $(X,G)$-cone-manifold is  a topological space such that any point in it has a neighbourhood isomorphic to a cone over a $(X_p,G_p)$-cone-manifold. 
\end{itemize}

One just remarks that $X_p$ is just the unit sphere at $p$ in $X$ and therefore $G_p$ can naturally be seen as a subgroup of $\mathrm{O}(n)$ where $n $ is the 
dimension of $X$.\sk

 A simple example of a non trivial cone-manifold is a Euclidean cone. If $X=\mathbb{R}^2$ and $G = \mathrm{Iso}(\mathbb{R}^2)$, $X_p = S^1$ and $G_p = \mathrm{O}(2)$. A $(X_p, G_p)$-manifold is nothing else but a  circle of length $\theta$ and a cone over it is a cone of angle $\theta$. Finally, remark that any $(X,G)$-manifold is also  a $(X,G)$-cone-manifold in a natural way.

\subsection{Cones are cone-manifolds.}
Let $X$ be a connected Riemannian manifold such that $G$ is the component of the identity of its isometry group. Let $X'$ be a totally geodesic submanifold of codimension $2$ in $X$ such that $\mathrm{Stab}_G(X')$ is $S^1=\mathbb R/\mathbb Z$, \textit{i.e.} it acts by rotation of angle $\theta$ around $X'$ for any $\theta \in S^1$. 

We explain the general construction of the \textit{cone of angle} $\theta$ \textit{over} $X'$. The metric completion $Y$ of the universal covering of $X \setminus X'$ is an infinite cyclic cover of $X$ branched along $X'$. There is a group $\mathbb{R}$ of isometries lifting the action of $S^1$ by rotation to $Y$ and if $\theta \in ]0, +\infty[$, one defines $X_{\theta}$ the \textbf{cone of angle} $\theta$ \textbf{over} $X'$ to be the quotient of $Y$ by the action of the rotation of angle $\theta$ on $Y$. The image in $X_{\theta}$  of the preimage of $X'$ in $Y$ 
is called the {\bf singular locus of the cone}.

\begin{prop}
$X_{\theta}$ is a $(X,G)$-cone manifold.
\end{prop}

\begin{proof} The proof goes by induction on the dimension of $X$. Away from its  singular locus, $X_{\theta}$ is a $(X,G)$-manifold hence the proposition is clear here. \sk 

Let $p$ be a point of the singular locus. The set $W$  of points of $X_{\theta}$ that can be joined to $p$ by a geodesic path of length $1$ happens to be a cone of angle $\theta$ for a sphere $S$ of radius $1$ at a point $q \in X'$ with isometry group $\mathrm{Stab}_G(q)$.  A neighbourhood of $p$ in $X_{\theta}$ is then the cone over $W$.  We want to show that $W$ is actually a $(S,\mathrm{Stab}_G(q)$ cone-manifold. This will be done by showing that $W$ is actually a cone of angle $\theta$ and applying the induction hypothesis.
\sk 

The intersection $S' = X' \cap S$ is a totally geodesic submanifold of $S$ for the metric induced by $X$ and $S^1 \subset \mathrm{Stab}_G(X') \subset \mathrm{Stab}_G(q)$. The universal cover of $S \setminus S'$ embeds in the one of $X \setminus X'$ and therefore the metric completion of the universal cover of $S \setminus S'$ embeds in the metric completion $Y$ of the universal cover of $X \setminus X'$. $W$ is then the quotient of the metric completion of the universal cover of $S \setminus S'$ by the rotation of angle $\theta$. Hence $W$ is a $(S,\mathrm{Stab}_G(q))$-cone and since $\mathrm{dim}(S) = \mathrm{dim}(X) - 1$, is a 
$(S,\mathrm{Stab}_G(q))$-cone-manifold. \end{proof}

\subsection{Joints}
We now restricts to the case when $X = {\mathbb{CH}}^n$ and $G = \mathrm{PU}(1,n)$. The unit sphere at a point$p$ in $X$ is $S^{2n-1} = \partial(B^n)$ where $B^n$ is the unit ball at $p$ and its isometry group is $\mathrm{U}(n) \subset G$. For every $k$ in $\{1, \ldots, n \}$, we can carry on the construction detailed below. 

\mk

The joint $A * B$ of two topological spaces $A$ and $B$ is the space you get by adjoining to every pair of points $(a,b) \in A \times B$ a segment $[a,b]$. This operation can be made geometrical if $A$ and $B$ are spherical manifold. One remarks that $S^{2(n+k)-1}$ is the joint of $S^{2n-1} * S^{2k-1}$ where $S^{2n-1}$ and $S^{2k-1}$ are embedded in $S^{2(n+k)-1}$ in a essentially unique way such that each points $x \in S^{2n-1}$ and $y \in S^{2k-1}$ are joined by a unique geodesic path of length $\frac{\pi}{2}$. This makes it very clear how one can endow the joint of $X$ a $(S^{2n-1}, \mathrm{U}(n))$-manifold and $Y$ a  $(S^{2k-1}, \mathrm{U}(k))$-manifold the structure of a $(S^{2(n+k)-1}, \mathrm{U}(n+k))$-manifold.  A good reference that deals with this construction is \cite[ Chapter I.5, p.63]{BridsonHaefliger}. 

This property of naturality extends in some way to cone-manifolds.

\begin{lemma}
\label{joint}
Let $M$ be a $(S^{2k-1}, \mathrm{U}(k))$-cone-manifold. Then the joint $S^{2(n-k)-1}*M$ has a natural structure of $(S^{2n-1}, \mathrm{U}(n))$-cone-manifold.
\end{lemma}

\begin{proof}
The proof goes by double induction on $n$ and $i  = (n-k)$. To be more precise, we assume that the lemma is true for all $(n',k')$ such that either $n' < n$ or $n'=n$ and and $k<k'$. Take $p \in S^{2(n-k)-1}*M$. We distinguish two cases :\mk 

\begin{enumerate}
\item {\bf $\boldsymbol{p}$ does not belong to $\boldsymbol{S^{2(n-k)-1}}$}. In that case $p$ belongs to an arc $]x,y]$ with $x \in S^{2(n-k)-1}$ and $y \in M$. Denote by $M_i$ the union of the strata of codimension $i$ of $M$. If $i=0$, \textit{i.e.} if $y$ is a regular point in $M$, then $p$ is a regular point of $S^{2(n-k)-1}*M$. For $i \geq 1$,    $S^{2(n-k)-1}*M_i$ is a $(S^{2n-3},U(n-1))$-cone-manifold by the induction hypothesis. In that case $p$ has a neighbourhood which is a cone over the joint $S^{2(k+i)-1}*V(y)$ where $V(y)$ is a $(S^{2(n-k-i)-1}, U(n-k-i))$-cone-manifold over which a neighbourhood of $y$ in $M_i$ is a cone.
\mk 

\item {\bf $\boldsymbol{p}$ belongs to $\boldsymbol{S^{2(n-k)-1}}$}. In that case a neighbourhood of $p$ in $S^{2(n-k)-1}*M$ is a cone over the joint $S^{2(n-k-1)-1}*M$ and the induction hypothesis allows to conclude.
\end{enumerate}
\end{proof}

\subsection{Strata}

A ${\mathbb{CH}}^n$-cone-manifold $X$ has a stratified structure $X_0 \sqcup X_1 \sqcup \cdots \sqcup X_n$ where $X_k$ is a ${\mathbb{CH}}^{n-k}$-manifold whose metric completion is $X_k \sqcup \cdots \sqcup X_n$. Every point $p \in X$ has a neighbourhood which is the cone over the joint $S^{2(n-i)-1}*X(p)$ where $X_p$ is a $(S^{2i - 1}, \mathrm{U}(i))$-cone-manifold. $X_k$ is defined to be the set of points for which the biggest integer $i$ for which a neighbourhood of $p$ has the latter structure is equal to $k$.

\subsection{Totally geodesic subcone-manifolds.}
 We assume here that $X$ is a Riemannian manifold which is either ${\mathbb{CH}}^n$ or $S^{k}$ and $G$ is either $\mathrm{PU}(1,n)$ or a subgroup of $\mathrm{O}(k)$. $X_p$ is the unit sphere at a point $p \in X$ and $G_p = \mathrm{Stab}_G(\{p\}) $. If $X$ is a $(X,G)$-cone-manifold, a totally geodesic sub-cone-manifold $Y$ of $X$ is a subset of $X$ such that the intersection of $Y$ with each stratum of $X$ is a totally geodesic submanifold of the stratum.

\begin{lemma}
\label{stab}
Let $p$ be a point of $X$ and $Y$ be a totally geodesic submanifold of $X$ such that $p \in Y$. Then $X_p \cap Y$ is a totally geodesic submanifold of $X_p$.

\end{lemma}

\begin{proof} This is a consequence that in all the cases we are considering there exists a subgroup $G'_p$ of $G_p$ such that $\mathrm{Stab}(G'_p) = Y$.
\end{proof}

\begin{prop}
A totally geodesic subcone-manifold $Y$ of a Riemannian cone-manifold $M$ endowed with the natural metric structure coming from its embedding is also a cone-manifold.
\end{prop}

\begin{proof} The proof goes by induction on $\mathrm{dim}(Y)$. Take $q$ in $Y$. $q$ has neighbourhood in $M$ which is a cone over a $(X_p,G_p)$-manifold $X'$, where $X_p$ is the unit sphere at a point $p \in X$ and $G_p = \mathrm{Stab}_G(\{p\})$. According to Lemma \ref{stab} $X' \cap Y$ is also a totally geodesic cone manifold of dimension $\mathrm{dim}(Y) -1$. The induction hypothesis ensures that $X' \cap Y$ is also a cone-manifold and therefore $p$ has a neighbourhood which is a cone over a cone-manifold.
\end{proof}

\subsection{\bf  Higher dimensional complex hyperbolic cones.}

We now give local models for some specific complex hyperbolic cone manifolds. In particular we generalise the notion of cone previously defined in the particular case of complex hyperbolic geometry. Let $X$ be a complete complex hyperbolic cone-manifold of dimension $k$ and let $p$ a point being a stratum of codimension $k$. We denote by $X_0$ the set of regular points, which is open in $X$. Consider the trivial product ${\mathbb{CH}}^n \times X_0$. There is a unique complex hyperbolic structure on ${\mathbb{CH}}^n \times X_0$ such that
\begin{itemize}

\item  each fiber $ \{ *\} \times X_0$ is locally totally geodesic;\sk 

\item any fiber $ \{ *\} \times X_0$ intersects ${\mathbb{CH}}^n \times \{ p\}$ orthogonally.\mk 
\end{itemize}

\noindent The metric completion of ${\mathbb{CH}}^n \times X_0$ is then ${\mathbb{CH}}^n \times X$. Here is the good moment to explain the notion of orthogonality in a $(\mathbb{C}\mathbb H^n, \mathrm{PU}(1,n))$-cone-manifold. Let $Y$ and $Z$ be two totally geodesic sub-cone-manifold of $X$,  a $(\mathbb{C}\mathbb H^n, \mathrm{PU}(1,n))$-cone-manifold. Assume that $Y$ and $Z$ intersect only at a point $p$. We say that they intersect orthogonally if every pair of regular points $p \in Y$ and $q \in Z$ is contained in an open set $U$ of $X$ such that 
\begin{itemize}
\item $U$ is isometric to an open subset of $\mathbb{C}\mathbb H^n$; \sk
\item $Y\cap U$ and $Z\cap U$ are respectively identified with open subsets of copies of $\mathbb{CH}^i$ and $\mathbb{CH}^{j}$ in $\mathbb{C}\mathbb H^n$ which intersect orthogonally.
\end{itemize}

\begin{prop}
\label{highercones}
${\mathbb{CH}}^n \times X$ seen as the metric completion of  $ {\mathbb{CH}}^n \times X_0$ is a complex hyperbolic cone-manifold.
\end{prop}

\begin{proof} Let $q$ be a point in $X$ which has maximal codimension. There is a neighbourhood of $q$ in $X$ which is a cone over a $(S^{2k-1}, U(k))$-cone-manifold $X'$. According to Lemma \ref{joint},  the spherical joint $ X' * S^{2n-1}$ has a natural structure of $(S^{2(n+k) - 1}, U(n+k))$-cone-manifolds, of which a neighbourhood of $q$ in ${\mathbb{CH}}^n \times X$ is a cone over.
\end{proof}

\bibliographystyle{alpha}
\bibliography{biblio}

\end{document}